\documentclass[12pt,oneside,english,shortlabels]{amsart}
\usepackage[T1]{fontenc}
\usepackage[latin9]{inputenc}
\usepackage{color}
\usepackage{babel}
\usepackage{amstext}
\usepackage{amssymb}
\usepackage{esint}
\usepackage[all]{xy}
\usepackage[unicode=true,pdfusetitle,
 bookmarks=true,bookmarksnumbered=false,bookmarksopen=false,
 breaklinks=false,pdfborder={0 0 0},backref=false,colorlinks=true]
 {hyperref}
\hypersetup{
 linkcolor=blue}
\usepackage{breakurl}

\makeatletter

\newcommand{\lyxmathsym}[1]{\ifmmode\begingroup\def\b@ld{bold}
  \text{\ifx\math@version\b@ld\bfseries\fi#1}\endgroup\else#1\fi}

\usepackage{amsthm}
\usepackage{enumitem}		

\usepackage{amsfonts}\usepackage{amscd}\usepackage[all]{xy}\usepackage{color}\usepackage{tikz}\usepackage{amsthm}
\font\calli=rsfs10 at 12 pt

\newtheorem{thm}{Theorem}[section]
\newtheorem*{thm*}{Theorem}
\newtheorem{cor}[thm]{Corollary}
\newtheorem*{cor*}{Corollary}
\newtheorem{lem}[thm]{Lemma}
\newtheorem*{lem*}{Lemma}
\newtheorem{prop}[thm]{Proposition}
\newtheorem*{prop*}{Proposition}

\newtheorem*{conjecture*}{Conjecture}

\newtheorem*{fact*}{Conjecture}

\newtheorem*{criterion*}{Criterion}

\newtheorem*{algorithm*}{Algorithm}

\newtheorem*{ax*}{Axiom}

\newtheorem*{assumption*}{Assumption}

\newtheorem*{question*}{Question}

\theoremstyle{remark}

\newtheorem{rem}[thm]{Remark}
\newtheorem*{rem*}{Remark}
\newtheorem{rems}[thm]{Remarks}
\newtheorem*{rems*}{Remarks}

\newtheorem*{claim*}{Claim}

\newtheorem*{exercise*}{Exercise}

\newtheorem*{note*}{Note}
\newtheorem{notation}[thm]{Notation}
\newtheorem*{notation*}{Notation}

\newtheorem*{summary*}{Summary}

\newtheorem*{acknowledgement*}{Acknowledgement}

\newtheorem*{conclusion*}{Conclusion}

\theoremstyle{definition}

\newtheorem{defn}[thm]{Definition}
\newtheorem*{defn*}{Definition}
\newtheorem{example}[thm]{Example}
\newtheorem*{example*}{Example}

\newtheorem*{examples*}{Examples}

\newtheorem*{problem*}{Problem}

\newtheorem*{xca*}{Exercise}

\newtheorem*{xcas*}{Exercises}

\newtheorem*{condition*}{Condition}

\theoremstyle{plain}

\newcommand{\A}{\mathcal{A}}
\newcommand{\B}{\mathcal{B}}

\newcommand{\E}{\mathcal{E}}
\newcommand{\F}{\mathcal{F}}

\renewcommand{\S}{\mathcal{S}}   %

\newcommand{\NN}{\mathbb{N}}
\newcommand{\ZZ}{\mathbb{Z}}
\newcommand{\QQ}{\mathbb{Q}}
\newcommand{\RR}{\mathbb{R}}
\newcommand{\CC}{\mathbb{C}}
\newcommand{\PP}{\mathbb{P}}

\renewcommand{\bar}[1]{\overline{#1}} %
\newcommand{\tld}[1]{\widetilde{#1}}

\DeclareMathOperator{\an}{an}

\newcommand{\PD}[3]{\frac{\partial^{#1}#2}{\partial {#3}^{#1}}}

\DeclareMathOperator{\dist}{dist}

\DeclareMathOperator{\vol}{vol}
\DeclareMathOperator{\naive}{naive}
\DeclareMathOperator{\INT}{Int}

\DeclareMathOperator{\Si}{Si}
\DeclareMathOperator{\abs}{abs}

\title[Oscillatory and Subanalytic Functions]{Integration of Oscillatory and Subanalytic Functions}

\author[Cluckers]{Raf Cluckers}
\address{Universit\'e de Lille,
Laboratoire Painlev\'e,
 CNRS - UMR 8524, Cit\'e Scientifique, 59655
Villeneuve d'Ascq Cedex, France, and,
KU Leuven, Department of Mathematics,
Celestijnenlaan 200B, B-3001 Leu\-ven, Bel\-gium}
\email{Raf.Cluckers@math.univ-lille1.fr}
\urladdr{http://rcluckers.perso.math.cnrs.fr/}

\author[Comte]{Georges Comte}
\address{   Univ. Grenoble Alpes, Univ. Savoie Mont Blanc, CNRS, LAMA, 73000 Chamb\'ery, France}
\email{Georges.Comte@univ-smb.fr}
\urladdr{http://gcomte.perso.math.cnrs.fr/}

\author[Miller]{Daniel~J.~Miller}
\address{Emporia State University, Department of Mathematics, Computer Science and Economics, 1200 Commercial Street, Campus Box 4027, Emporia, KS 66801, U.S.A.}
\email{dmille10@emporia.edu}

\author[Rolin]{Jean-Philippe Rolin}
\address{Universit\'e de Bourgogne Franche-Comt\'e, IMB,
CNRS UMR 5584,
9 Avenue Savary
BP47870,
F-21078 Dijon Cedex, France}
\email{jean-philippe.rolin@u-bourgogne.fr}
\urladdr{http://rolin.perso.math.cnrs.fr/}

\author[Servi]{Tamara Servi}
\address{Institut de Math\'ematiques de Jussieu - Paris Rive Gauche
Universit\'e Paris Diderot (Paris 7)
UFR de Math\'ematiques - B\^atiment Sophie Germain
Case 7012,
F-75205 Paris Cedex 13,
France}
\email{tamara.servi@math.univ-paris-diderot.fr}
\urladdr{http://www.logique.jussieu.fr/~servi/index.html}



\makeatother

\begin{document}
\global\long\def\U{\mathcal{U}}
\global\long\def\C{\mathcal{C}}
\global\long\def\G{\mathcal{G}}
\maketitle
\begin{abstract}
We prove the stability under integration and under Fourier transform
of a concrete class of functions containing all globally subanalytic
functions and their complex exponentials. This paper extends the investigation
started in \cite{LiRo1998} and \cite{cluckers-miller:stability-integration-sums-products}
to an enriched framework including oscillatory functions. It provides
a new example of fruitful interaction between analysis and singularity
theory.
\end{abstract}
\maketitle

\tableofcontents

\section{Introduction}

\label{s:intro}

In this paper we prove the stability
under parameterized integration of a class of functions containing
all globally subanalytic functions and their complex exponentials,
with methods pertaining to subanalytic geometry.
Note that the theories of holonomic $D$-modules and holonomic distributions (and, further away, of $\ell$-adic cohomology and of motivic integration) have the richness of combining geometry with Fourier transforms, and that these theories all have found far reaching applications. Applications of our setting are to be expected, but are not the content of the present paper.
Let us just mention that, in the context of motivic and $p$-adic integration \cite{cluckers_loeser:constructible_exponential_fuctions},
similar stability results have found recent applications
in the Langlands program \cite{cluckers_hales_loeser:transfer_principle_fundamental}, \cite{CGH2}.
The stability under integration of certain classes of real functions
was already considered in \cite{LiRo1998}, \cite{clr}, \cite{cluckers-miller:stability-integration-sums-products},
\cite{kaiser:integration_semialgebraic_nash}, but none of these classes
allows oscillatory behavior, let alone stability under Fourier transforms.
Let us explain our results in detail.
\begin{defn}
\label{def: subanaltic} A set $X\subseteq\RR^{m}$ is \textbf{\emph{globally
subanalytic }} if in any standard euclidean chart $\RR^{m}$ of $\PP^{m}\left(\RR \right)$,
the image of $X$ in $\PP^{m}\left(\RR \right)$ is a subanalytic subset of this
chart in the sense of \cite{bm_semi_subanalytic} and \cite{stasica_denkowska}.
Equivalently, $X\subseteq\mathbb{R}^{m}$ is globally subanalytic
if it is the image under the canonical projection from $\mathbb{R}^{m+n}$
to $\mathbb{R}^{m}$, of a globally semianalytic subset of $\mathbb{R}^{m+n}$
(i.e. a set $Y\subseteq\mathbb{R}^{m+n}$ such that in a neighbourhood
of every point of $\mathbb{P}^{1}\left(\mathbb{R}\right)^{m+n}$,
$Y$ is described by a finite number of analytic equations and inequalities).
Given a set $X\subseteq\mathbb{R}^{m}$, a map $f:X\rightarrow\mathbb{R}^{n}$
is globally subanalytic if its graph is a globally subanalytic subset
of $\RR^{m+n}$ (this definition implies that $X$ is a globally subanalytic
subset of $\RR^{m}$, since the collection of globally subanalytic
sets is closed under projections).
\end{defn}
In model-theoretic terms, a set is globally subanalytic if and only
if it is definable in the structure $\RR_{\an}$, the expansion of
the ordered real field by all restricted analytic functions (as defined
in \cite{dmm:exp}). By \cite{gabriel:proj,vdd:tarski:general,vdd:elementary_theory_restricted_elementary},
this is an o-minimal structure and therefore the reader may refer
for instance to \cite{vdd:tame}, \cite{vdd:mill:omin} for the basic
geometric properties of globally subanalytic sets and functions that
we will use in the sequel.

For the sake of brevity, from now on we will use the word ``subanalytic''
as an abbreviation for the phrase ``globally subanalytic''. So in
this usage of the word, the natural logarithm $\log\colon\left(0,+\infty\right)\to\RR$
and the trigonometric functions $\sin\colon\RR\to\RR$ and $\cos\colon\RR\to\RR$
are not subanalytic, although the restriction of any one of these
functions to any compact subinterval of its domain is subanalytic.\medskip{}

Given a subanalytic set $X\subseteq\mathbb{R}^{m}$, we denote by
$\S\left(X\right)$ the algebra of all real-valued subanalytic functions on $X,$
and we write
\[
\S:=\left\{\S\left(X\right):m\in\mathbb{N},\ \ensuremath{X\subseteq\mathbb{R}^{m}}\text{\ subanalytic}\right\}
\]
for the system of all real-valued subanalytic functions.

Our aim is to provide a full description of the smallest system
\[
\E:=\{\E\left(X\right):m\in\mathbb{N},\ \ensuremath{X\subseteq\mathbb{R}^{m}}\text{\ subanalytic}\}
\]
such that $\E\left(X\right)$ is a $\CC$-algebra of complex-valued functions
on $X\subseteq\mathbb{R}^{m}$ satisfying
\begin{equation}
\S\left(X\right)\cup\{\mathrm{e}^{\mathrm{i}f}:f\in\S\left(X\right)\}\subseteq\E\left(X\right)\label{eq:E}
\end{equation}
and such that $\E$ is stable under integration.

Here stability under integration for $\E$ means that if $X\subseteq\mathbb{R}^{m}$
is a subanalytic set, $n\in\NN$, and $f\in\E\left(X\times\RR^{n}\right)$ is
such that $f\left(x,\cdot\right)\in L^{1}\left(\RR^{n}\right)$ for all $x\in X$, then
the function $F\colon X\to\CC$ defined by
\begin{equation}
F\left(x\right)=\int_{y\in\RR^{n}}f\left(x,y\right)\,\mathrm{d}y,\quad\text{for \ensuremath{x\in X},}\label{eq:paramInteg}
\end{equation}
is in $\E\left(X\right)$.

Note that the existence of $\mathcal{E}$ is guaranteed by the fact
that the collection on the left side of (\ref{eq:E}) is contained
in the class of all complex-valued measurable functions, a class stable
under parameterized integration.
 We will describe in detail the system $\E$ in the next section. Our
main result is that $\E$ coincides with the system $\C^{\text{exp}}$
of $\CC$-algebras $\C^{\text{exp}}\left(X\right)$ defined in Definition \ref{def:Cexp}
(see Remark \ref{rem: important}(\ref{enu:Cexp=00003DE}), for which
we have an explicit description of the generators (Definition \ref{def:generator}).
It is worth noting that the generators of the algebra $\mathcal{E}\left(X\right)$ are defined in terms of $1$-variable integrals of a particularly simple form (see Definition \ref{def: gamma function}).

A strong motivation to allow oscillatory functions in our system comes
from singularity theory, where oscillatory integrals have been heavily
investigated for decades (for an introduction, and among numerous
other references, see in particular \cite{arnold_gusein_varchenko:singularities_2},
\cite{malgrange:integrales_asymptotiques}, \cite{varchenko:newton_polyhedra_oscillatory_integrals}).
A series of preparation and monomialization results (\cite{parusinski:preparation},
\cite{lr:prep}, \cite{cluckers-miller:stability-integration-sums-products},
\cite{cluckers_miller:loci_integrability}, \cite{miller_dan:preparation_theorem_weierstrass_systems})
for subanalytic functions and their logarithms, provides a powerful
tool to study the nature of oscillatory integrals with subanalytic
phase and amplitude.

As indicated in the introduction of \cite{LiRo1998}, the idea of using
a preparation theorem to understand the integration of subanalytic
functions was suggested by L. van den Dries, and indeed successfully
used in \cite{LiRo1998} and \cite{clr}, where it is proved (use  \cite[Theorem 1]{LiRo1998} and \cite[Proposition 1]{clr}, or directly  \cite[Theorem 1']{clr}) that the parameterized
integrals of subanalytic functions belong to the class $\C:=\left(\C\left(X\right)\right)_{X}$
of constructible functions (the algebra $\C\left(X \right)$ of functions on the
subanalytic set $X$ is generated as a $\CC$-algebra by the subanalytic
functions on $X$ and their logarithm, see Definition \ref{def: constructible}).
In particular, the function volume of fibres of a subanalytic family and
the density function along a subanalytic set also belong to the class
$\C$ (see \cite{clr}).

The question of finding a system of $\CC$-algebras of functions containing
$\S$ and stable under parameterized integration has been attacked
and solved in \cite{cluckers-miller:stability-integration-sums-products}
(see also \cite{cluckers_miller:loci_integrability}), where the authors
show that the class $\C$ itself is stable under parameterized integration
(see \cite{kaiser:integration_semialgebraic_nash} for an interesting
subcollection of $\C$, also stable under integration). Here again
the main tool of proof is a preparation theorem for functions of $\C$.
Note that the class $\C$ is a class of functions definable in the
o-minimal structure $\mathbb{R}_{\text{an,exp}}$, the expansion of
$\mathbb{R}_{\text{an}}$ by the full real exponential function.

As already mentioned, the problem we address and solve here is the
problem of explicitly describing a system of $\CC$-algebras (actually
the smallest), stable under parameterized integration, containing
$\mathcal{S}$ and containing the complex-valued oscillatory functions
$\mathrm{e}^{\mathrm{i}f}$, for all subanalytic functions $f$. Since
we consider oscillatory functions, we are no longer in an o-minimal
setting. However, the preparation results mentioned above (see Section
\ref{s:prepSubConstr}) prove extremely useful and powerful even for
dealing with oscillatory functions. 
To prove our results, we combine
these preparation techniques with the theory of continuously uniformly distributed maps (see Section \ref{s: proof of main results}),
a new ingredient in this context.\medskip{}

Oscillatory integrals are central in many branches of mathematics
and physics. Following Stein \cite{stein:harmonic_analysis}, an oscillatory
integral of the first kind is a parameterized integral $I\left(x\right)$, $x\in\RR$,
defined by
\begin{equation}
I\left(x\right)=\int_{y\in\RR^{n}}f\left(y\right)\mathrm{e}^{\mathrm{i}x\Phi\left(y\right)}\,\mathrm{d}y,\label{eq:first kind}
\end{equation}
where the \emph{amplitude } $f$ and the \emph{phase} $\Phi$ are in
general $C^\infty$ functions. The principle of stationary phase asserts,
when the phase $\Phi$ has no critical point on the support of $f$
(assume for simplicity that $f$ has compact support), that $x\mapsto I\left(x\right)$
is in $\hbox{\calli S}\hskip0.6mm\left(\RR\right)$, the Schwartz space of rapidly
decreasing functions. As a consequence, the asymptotic behaviour of
$I\left(x\right)$ at $+\infty$, modulo $\hbox{\calli S}\hskip0.6mm\left(\RR\right)$,
presents some interest only at critical points of the phase. If the
phase is analytic, one can show that this asymptotic behaviour only
depends on the Taylor series of the amplitude function at critical
points of the phase, and that $I\left(x\right)$ can be expanded in an asymptotic
series
\[
\sum_{p}x^{-p/r}\sum_{k=0}^{n-1}c_{p,k}\log^{k}\left(x\right),
\]
where $r$ is a positive integer not depending on $f$ and $p$is $\in\NN\setminus\{0\}$
(see \cite{malgrange:integrales_asymptotiques} Section 7 and \cite{arnold_gusein_varchenko:singularities_2},
Chapter 7). Using Hironaka's resolution of singularities on the phase
function, one can prove this result by reducing to the case of a monomial
phase. The exponents $-p/r$ and $k$ are related to the monodromy
of the phase, in case the phase has an isolated singular point in
the complex domain: $\mathrm{e}^{2\pi\mathrm{i}\left(\frac{p}{r}-1\right)}$ is actually
an eigenvalue of multiplicity $\ge k+1$ of the monodromy operator
of the phase (see \cite{malgrange:integrales_asymptotiques} for more
details). Furthermore the principal part of the exponents $-p/r$,
called the oscillation index (see \cite{arnold_gusein_varchenko:singularities_2},
Section 6.1.9), can be computed in terms of Newton's diagram of the
Taylor expansion of the phase at its critical point (see \cite{arnold_gusein_varchenko:singularities_2},
\cite{varchenko:newton_polyhedra_oscillatory_integrals}).

Similarly, in this paper, we estimate and compare the asymptotics
at infinity of different terms appearing in our parameterized integrals,
namely integrals as in (\ref{eq:paramInteg}), and in this situation
the preparation theorem for constructible functions (Proposition \ref{prop:prepSubConstr})
appears as the counterpart of Hironaka's theorem. Of course in our
general context, no geometric interpretation for exponents appearing
in the asymptotics considered can be given, but there might be connections
with the classical cases still to be discovered.

An oscillatory integral of the second kind has the form
\begin{equation}
I\left(x\right)=\int_{y\in\RR^{n}}f\left(x,y\right)\mathrm{e}^{\mathrm{i}\Phi\left(x,y\right)}\,\mathrm{d}y,\label{eq:second kind}
\end{equation}
where now $x=\left(x_{1},\ldots,x_{m}\right)$ is a tuple of variables.
A classical example of oscillatory integral of the second kind is
given by Fourier transforms. A second more complicated example is
given by the Fourier Integral Operator (see \cite{hormander:fourier_integral_operators,stein:harmonic_analysis}),
which plays a role in approximating the solutions of a large class
of PDEs (for example, the wave equation). A natural question arises:
how to describe the nature of (\ref{eq:second kind}), according to
the nature of the amplitude and of the phase?

Note that in (\ref{eq:second kind}) the parameters $x$ are \textquotedblleft{}intertwined\textquotedblright{}
with the integration variables $y$ in the expressions for the amplitude
$f$ and the phase $\Phi$. If we consider oscillatory integrals of
the second kind with subanalytic amplitude and phase, then the aforementioned
preparation results prove a very powerful tool to monomialize the
phase while respecting the different nature of the variables $x$
and $y$. \medskip{}

The main result of this paper (Theorem \ref{thm:main interp}) implies
that oscillatory integrals (of the first and second kind) with subanalytic
phase and amplitude belong to the system $\E$. Moreover, still by
stability of $\E$ under integration, oscillatory integrals with subanalytic
phase and amplitude in $\E$, still belong to $\E$.

In particular, for $X\subseteq\mathbb{R}^{m}$ subanalytic, the algebra
$\E\left(X\times\mathbb{R}\right)$ is stable under taking parametric
Fourier transforms:

\begin{equation}
\begin{aligned}\text{if}\ f\left(x,t\right)\in\mathcal{E}\left(X\times\mathbb{R}\right)\ \text{and}\ \forall x\in X,\ f\left(x,\cdot\right)\in L^{1}\left(\mathbb{R}\right),\\
\text{then}\ \hat{f}\left(x,y\right)=\int_{\RR}f\left(x,t\right)\mathrm{e}^{-2\pi\mathrm{i} yt}\,\mathrm{d}t\in\mathcal{E}\left(X\times\mathbb{R}\right).
\end{aligned}
\label{eq: fourier tranf param}
\end{equation}

On the other hand, $\E$ can also be viewed as the smallest system
of $\CC$-algebras containing the class $\mathcal{C}$ of constructible
functions and stable under composition with subanalytic functions
and parametric Fourier transform (see Remark \ref{rem: important}(\ref{enu:fourier})).
Since there are not many systems (of algebras) of functions which
are stable under Fourier transforms, we would like to insist in this
introduction on the fact that $\E$ is such a system, which is moreover
fully described by its generators. \medskip{}

Like $\mathcal{E}\left(\mathbb{R}^{n}\right)$, the space of Schwartz
functions $\hbox{\calli S}\hskip0.6mm\left(\RR^{n}\right)$ is also an algebra
stable under taking Fourier transforms. Since the Fourier transform
operator
\[
\hbox{\calli F}\colon\left(\hbox{\calli S}\hskip0.6mm\left(\RR^{n}\right),\Vert~\Vert_{2}\right)\to\left(\hbox{\calli S}\hskip0.6mm\left(\RR^{n}\right),\Vert~\Vert_{2}\right)
\]
is continuous, using the density of $\hbox{\calli S}\hskip0.6mm\left(\RR^{2}\right)$
in the space $L^{2}\left(\RR^{n}\right)$, one can extend $\hbox{\calli F}\colon\hbox{\calli S}\hskip0.6mm\left(\RR^{n}\right)\to\hbox{\calli S}\hskip0.6mm\left(\RR^{n}\right)$
to
\begin{equation}
\widetilde{\hbox{\calli F}\ }\hskip-0.9mm\colon L^{2}\left(\RR^{n}\right)\to L^{2}\left(\RR^{n}\right).\label{eq: F tilda}
\end{equation}
One obtains thus the classical stability of $L^{2}\left(\RR^{n}\right)$ under
the Fourier-Plancherel extension $\widetilde{\hbox{\calli F}\ }$ of the Fourier transform
$\hbox{\calli F}$. In Section \ref{sec:Asymptotic-expansions-of}
we prove that $\E$ is even stable under the extension $\widetilde{\hbox{\calli F}\ }$
of the Fourier transform: the image of $\mathcal{E}\left(\mathbb{R}^{n}\right)\cap L^{2}\left(\mathbb{R}^{n}\right)$
under $\widetilde{\hbox{\calli F}\ }$ is $\mathcal{E}\left(\mathbb{R}^{n}\right)\cap L^{2}\left(\mathbb{R}^{n}\right)$
(Theorem \ref{cor:Plancherel}).
To this end, we need to develop in Section \ref{sub:-completeness-and-the}
elements of a theory of uniformly distributed family of maps.

\medskip{}

Let us also mention that, since the function $\mathrm{e}^{-|x|}$
is in $\mathcal{E}\left(\mathbb{R}\right)$ (see Example \ref{exa: fourier transform of flat exp}),
one may interpolate \emph{families of exponential periods} with functions
from $\mathcal{E}$. More precisely, following \cite{BlochEsn} and
Section 4.3 of \cite{kontsevitch_zagier:periods}, a real number $a$
is called an exponential period if there exist $\Delta\subset\RR^{n}$
(for some $n\in\mathbb{N}$) and functions $f,g\colon\Delta\rightarrow\mathbb{R}$
such that $\Delta,f,g$ are semi-algebraic over $\mathbb{Q}$ (i.e.
they are described by first order formulas in the language of ordered
rings with no other constant symbols than rational numbers) and
\[
a=\int_{y\in\Delta}f\left(y\right)\mathrm{e}^{g\left(y\right)}\,\mathrm{d}y.
\]
A natural version in families of this concept is the following: let
$X\subseteq\RR^{m}$, $\Delta\subseteq X\times\RR^{n}$ and $f,g\colon\Delta\to\RR$
be semi-algebraic over $\QQ$. Suppose that for each $x\in X$,
\[
a\left(x\right)=\int_{y\in\Delta_{x}}f_{x}\left(y\right)\mathrm{e}^{g_{x}\left(y\right)}\,\mathrm{d}y
\]
is finite, where $\Delta_{x}=\{y\in\RR^{n}:\left(x,y\right)\in\Delta\}$ and
$f_{x}\left(y\right)=f\left(x,y\right),\ g_{x}\left(y\right)=g\left(x,y\right)$. Then the
collection $\{ a\left(x\right):\ x\in X\cap\QQ^{m}\} $ forms a natural
family of exponential periods. Suppose that there is a constant $N$
such that $g<N$ on $\Delta$. It then follows from stability under
integration of $\mathcal{E}$ (Theorem \ref{thm:main interp}) and
Example \ref{exa: fourier transform of flat exp} that the interpolating
function $\mathbb{R}\ni x\mapsto a\left(x\right)\in\mathbb{R}$ belongs to $\E\left(X\right)$.
\medskip{}

Finally, the work in this paper can be seen as addressing a question
raised by D. Kazhdan at the 2009 Model Theory Conference in Durham,
about a possible model-theoretic understanding of real oscillatory
integrals, in analogy to the understanding of motivic oscillatory
integrals in \cite{cluckers_loeser:constructible_exponential_fuctions}
and \cite{HK}.
\begin{acknowledgement*}
The authors would like to thank the Forschungsinstitut f\"ur Mathematik
(FIM) at ETH Z\"urich, the MSRI with the program Model Theory, Arithmetic
Geometry and Number Theory, and the Isaac Newton Institute for the
hospitality during part of the research for this paper.
The author R.Cluckers~is supported by the European Research
Council under the European Community's Seventh Framework Programme
(FP7/2007-2013) with ERC Grant Agreement nr. 615722 MOTMELSUM, and
would like to thank the Labex CEMPI (ANR-11-LABX-0007-01).
The author G. Comte is supported by ANR-15-CE40-0008. 
The authors G.Comte and J.-P. Rolin are supported by ANR-11-BS01-0009
STAAVF. 

The authors are very grateful to the anonymous referees for their careful reading and for many insightful suggestions. In particular, the use of the theory of continuously uniformly distributed maps, which was suggested by one of the referees, allowed us to shorten considerably the final step of the proof of the main result (Theorem \ref{thm:main interp}).

\end{acknowledgement*}

\section{Notation, main results and layout of this paper}

\label{s:mainResults}

This section states the main definitions, theorems and corollaries
of the paper.

We proceed to construct $\E$ by first defining some systems of rings
of functions intermediary between $\S$ and $\E$.
\begin{defn}
\label{def: constructible}For each subanalytic set $X\subseteq\mathbb{R}^{m}$,
define $\C\left(X\right)$ to be the ring of real-valued functions on $X$ generated
by
\[
\S\left(X\right)\cup\{\log f\left(x\right):f\in\S\left(X\right),f>0\}.
\]
We call $\C\left(X\right)$ the ring of constructible functions on $X$, and
we say that a function is \textbf{\emph{constructible}} if it has
a subanalytic domain $X$ and is a member of $\C\left(X\right)$. Write
\[
\C:=\left(\C\left(X\right)\right)_{\text{\ensuremath{X}\ is subanalytic}}
\]
for the system of all constructible functions.

Thus $f\in\C\left(X\right)$ if and only if $f$ can be expressed as a finite
sum of finite products of the form
\begin{equation}
f\left(x\right)=\sum_{j}f_{j}\left(x\right)\prod_{k}\log f_{j,k}\left(x\right)\label{eq: form of constructible fnct}
\end{equation}
with $f_{j},f_{j,k}\in\S\left(X\right)$ and $f_{j,k}>0$.
\end{defn}

It is easy to see that any constructible function can be defined as
a parameterized integral of a subanalytic function, and it was shown
in \cite[Theorem 1.3]{cluckers-miller:stability-integration-sums-products}
that the constructible functions are stable under integration. Therefore
the constructible functions form the smallest class of functions defined
on the subanalytic sets that is stable under integration and that
contains all subanalytic functions.

It follows that
\[
\C\left(X\right)\cup\{\mathrm{e}^{\mathrm{i}f\left(x\right)}:f\in\S\left(X\right)\}\subseteq\E\left(X\right)
\]
for each subanalytic set $X$. This leads us to the following definition.
\begin{defn}
\label{def: naive}For each subanalytic set $X\subseteq\mathbb{R}^{m}$,
define $\C_{\naive}^{\exp}\left(X\right)$ to be the ring of functions on $X$
generated by
\[
\C\left(X\right)\cup\{\mathrm{e}^{\mathrm{i}f\left(x\right)}:f\in\S\left(X\right)\}.
\]
Write
\[
\C_{\naive}^{\exp}:=\left(\C_{\naive}^{\exp}\left(X\right)\right)_{\text{\ensuremath{X}\ is subanalytic}}.
\]

\end{defn}
Thus $f\in\C_{\naive}^{\exp}\left(X\right)$ if and only if $f$ can be written
as a finite sum
\[
f\left(x\right)=\sum_{j=1}^{J}f_{j}\left(x\right)\mathrm{e}^{\mathrm{i}\phi_{j}\left(x\right)},\quad\text{with \ensuremath{f_{j}\in\C\left(X\right)}\ and \ensuremath{\phi_{j}\in\S\left(X\right)}.}
\]

The elements of $\C_{\naive}^{\exp}\left(X\right)$ are complex-valued functions.
Hence, it is convenient to give the following definition.
\begin{defn}
\label{def complex valued}If $f\colon X\to\CC$ is such that its real and
imaginary components are in $\mathcal{S}\left(X\right)$ (in $\C\left(X\right)$,
respectively), then we call $f$ a complex-valued subanalytic (constructible,
respectively) function.

Notice that, if $\phi\left(x\right)$ is a bounded subanalytic function,
then $\mathrm{e}^{\mathrm{i}\phi\left(x\right)}$ is a complex-valued
subanalytic function.\end{defn}
\begin{rem}
\label{rem: mtivation for gamma}We will see in Section \ref{sec:Asymptotic-expansions-of}
that the elements of $\C_{\naive}^{\exp}\left([0,+\infty)\right)$ have certain
\emph{convergent} asymptotic expansions at $+\infty$. This implies
that there are no Schwartz functions in $\C_{\naive}^{\exp}\left([0,+\infty)\right)$.
In particular, the function $f\left(x\right)=\mathrm{e}^{-x}$
is not in $\C_{\naive}^{\exp}\left([0,+\infty)\right)$, while it can be easily
shown that $f\in\mathcal{E}\left([0,+\infty)\right)$. Now consider the function
$\text{Si}\left(x\right)=\int_{0}^{x}\frac{\sin\left(t\right)}{t}\,\mathrm{d}t$,
which is clearly in $\E\left([0,+\infty)\right)$. However, $\text{Si}\left(x\right)$
is easily seen to have a \emph{divergent} asymptotic expansion at
$+\infty$, therefore $\Si$ cannot be in $\C_{\naive}^{\exp}\left([0,+\infty)\right)$
(the details of the proof of this remark will be carried out in Section
\ref{sec:Asymptotic-expansions-of}).
\end{rem}
This example suggests that to construct $\E$, we cannot avoid including
functions computable from single-variable integrals. Our main claim
is that we shall only need to consider single-variable integrals of
the following special form.
\begin{defn}
\label{def: gamma function}For each $\ell\in\NN$, subanalytic set
$X\subseteq\mathbb{R}^{m}$ and $h\in\S\left(X\times\RR\right)$ such that $\forall x\in X,\ t\mapsto h\left(x,t\right)\in L^{1}\left(\mathbb{R}\right)$,
define $\gamma_{h,\ell}\colon X\to\CC$ by
\[
\gamma_{h,\ell}\left(x\right)=\int_{\RR}h\left(x,t\right)\left(\log|t|\right)^{\ell}\mathrm{e}^{\mathrm{i}t}\,\mathrm{d}t\text{.}
\]

This definition makes sense because for each $x\in X$, requiring
that $t\mapsto h\left(x,t\right)$ is in $L^{1}\left(\RR\right)$ is equivalent to requiring
that $t\mapsto h\left(x,t\right)\left(\log|t|\right)^{\ell}$ is in $L^{1}\left(\RR\right)$. This
is easily justified using elementary calculus and expanding $t\mapsto h\left(x,t\right)$
as in Remark \ref{rem: lojas}.\end{defn}
\begin{rem}
\label{rem: gamma subanalytic}If $g\in\mathcal{S}\left(X\right)$,
sometimes it will be convenient to see $g$ as a function of type
$\gamma_{h,\ell}$. To see this, take $\ell=0$ and $h\left(x,t\right)=\frac{1}{2}g\left(x\right)\chi\left(t\right)$,
where $\chi\left(t\right)$ is the characteristic function of the
interval $\left[-\frac{\pi}{2},\frac{\pi}{2}\right]$. In particular,
the constant function 1 can be viewed as a function of type $\gamma_{h,\ell}$
(and for the rest of the paper we will implicitly assume so).
\end{rem}
Note that for $x\in X$, $\gamma_{h,\ell}\left(x\right)=\int_{\mathbb{R}}\widetilde{h}(x,t)\,\mathrm{d}t$,
	where $\widetilde{h}\left(x,t\right)=h\left(x,t\right)\left(\log\left|t\right|\right)^{\ell}\mathrm{e}^{\mathrm{i}t}$
	if $t\ne0$ and $\widetilde{h}\left(x,0\right)=0$. Since $\widetilde{h}\in\mathcal{C}_{\mathrm{naive}}^{\exp}\left(X\times\mathbb{R}\right)$ and $\C_{\naive}^{\exp}\left(X\times\RR\right)\subseteq\E\left(X\times\RR\right)$,
we must have $\gamma_{h,\ell}\in\E\left(X\right)$. This leads us to the following
definition.
\begin{defn}
\label{def:Cexp}For each subanalytic set $X\subseteq\mathbb{R}^{m}$,
define $\C^{\exp}\left(X\right)$ to be the $\C_{\naive}^{\exp}\left(X\right)$-module of
functions on $X$ generated by
\[
\{\gamma_{h,\ell}:\text{\ensuremath{\ell\in\NN}\ and \ensuremath{h\in\S\left(X\times\RR\right)}\ with }t\mapsto h\left(x,t\right)\text{\ in\ }L^{1}\left(\RR\right)\}.
\]
We write
\[
\C^{\exp}:=\left(\C^{\exp}\left(X\right)\right)_{\text{\ensuremath{X}\ is subanalytic}}.
\]
Thus $f\in\C^{\exp}\left(X\right)$ if and only if $f$ can be written as a finite
sum
\[
f\left(x\right)=\sum_{j=1}^{J}f_{j}\left(x\right)\gamma_{h_{j},\ell_{j}}\left(x\right),\quad\text{with \ensuremath{f_{j}\in\C_{\naive}^{\exp}\left(X\right)}, \ensuremath{h_{j}\in\S\left(X\times\RR\right)}\ and \ensuremath{\ell_{j}\in\NN},}
\]
where $\forall x\in X,\ t\mapsto h_{j}\left(x,t\right)\in L^{1}\left(\mathbb{R}\right)$
for each $j$. \end{defn}
\begin{rem}
\label{rem: table under right composition with sa} Notice that $\mathcal{C}^{\exp}$
is stable under composition with subanalytic functions, in the following
sense: if $X\subseteq\mathbb{R}^{m},\ Y\subseteq\mathbb{R}^{n}$ are
subanalytic sets, $G:Y\rightarrow X$ is a subanalytic map and if
$f\in\mathcal{C}^{\exp}\left( X\right)$, then $f\circ G\in\mathcal{C}^{\exp}\left(Y \right)$.
\end{rem}
For each subanalytic set $X\subseteq\mathbb{R}^{m}$, it is clear
that $\C^{\exp}\left(X\right)\subseteq\E\left(X\right)$. Hence, our next task is to study
the parametric integrals of functions $f\in\C^{\exp}\left(X\times\mathbb{R}^{n}\right)$.
\begin{notation}
Write $\left(x,y\right)=\left(x_{1},\ldots,x_{m},y_{1},\ldots,y_{n}\right)$ for the standard
coordinates on $\RR^{m+n}$. Define $\Pi_{m}\colon\RR^{m+n}\to\RR^{m}$
by $\Pi_{m}\left(x,y\right)=x$. For each set $D\subseteq\RR^{m+n}$, define
the fibre of $D$ over $x$ by
\[
D_{x}=\{y\in\RR^{n}:\left(x,y\right)\in D\}.
\]
\end{notation}
\begin{defn}
For any Lebesgue measurable function $f\colon D\to\CC$ with $D\subseteq\RR^{m+n}$
and $\Pi_{m}\left(D\right)=X$, define the \textbf{\emph{locus of integrability
of $f$ over $X$}} by
\[
\INT\left(f,X\right):=\left\{x\in X: f\left(x,\cdot\right)\in L^{1}\left(D_{x}\right)\right\}.
\]
\end{defn}
\begin{rem}
\label{rem: intetrability locus}Let $f\in\C^{\exp}\left(X\times\RR^{n}\right)$
and suppose that $f\left(x,\cdot\right)\in L^{1}\left(\RR^{n}\right)$ for all $x\in X$.
To compute $F\left(x\right)=\int_{y\in\RR^{n}}f\left(x,y\right)\,\mathrm{d}y$, one typically
works by induction on $n$, using Fubini's theorem to express it as
an iterated integral
\[
\int_{\RR^{n-1}}\left(\int_{\RR}f\left(x,y_{1},\ldots,y_{n-1},y_{n}\right)\,\mathrm{d}y_{n}\right)\,\mathrm{d}y_{1}\wedge\ldots\wedge\mathrm{d}y_{n-1}.
\]
But then one is confronted with the fact that
\begin{equation}
\left(x,y_{1},\ldots,y_{n-1}\right)\longmapsto\int_{\RR}f\left(x,y_{1},\ldots,y_{n-1},y_{n}\right)\,\mathrm{d}y_{n}\label{eq:iterateInteg}
\end{equation}
might not be defined on all of $X\times\RR^{n-1}$; all we know is
that (\ref{eq:iterateInteg}) is defined for all $x\in X$ and almost
all $\left(y_{1},\ldots,y_{n-1}\right)\in\RR^{n-1}$. So in order to have a stable
framework that considers (\ref{eq:iterateInteg}) to be a ``parameterized
integral'' as well, it is useful to consider the more general situation
from the start where one drops the assumption that $f\left(x,y\right)$ is integrable
in $y$ for all $x\in X$, but one then additionally studies the locus
of integrability of $f$ over $X$ (see Theorem \ref{thm:interp-1}).
\end{rem}
We are now ready to state the main result of this paper.
\begin{thm}
[Stability under integration]\label{thm:main interp}Let $f\in\C^{\exp}\left(X\times\RR^{n}\right)$
for some subanalytic set $X\subseteq\mathbb{R}^{m}$ and $n\in\NN$.
Then there exists $F\in\C^{\exp}\left(X\right)$ such that
\[
F\left(x\right)=\int_{\RR^{n}}f\left(x,y\right)\,\mathrm{d}y,\quad\text{for all \ensuremath{x\in\INT\left(f,X\right)}.}
\]

\end{thm}
It is clear from the definition that the module $\C^{\exp}\left(X\right)$ is
closed under addition, but it is not so apparent from the definition
alone whether $\C^{\exp}\left(X\right)$ is closed under multiplication. That
$\C^{\exp}\left(X\right)$ is a ring is in fact a consequence of our main result.
\begin{cor}
\label{cor:CexpRing} For each subanalytic set $X$, $\C^{\exp}\left(X\right)$
is a ring. \end{cor}
\begin{proof}
For any $n\in\NN$ and functions $\gamma_{h_{1},\ell_{1}},\ldots,\gamma_{h_{n},\ell_{n}}$,
writing $y=\left(y_{1},\ldots,y_{n}\right)$ we have
\begin{eqnarray*}
\prod_{j=1}^{n}\gamma_{h_{j},\ell_{j}}\left(x\right) & = & \prod_{j=1}^{n}\left(\int_{\RR}h\left(x,y_{j}\right)\left(\log|y_{j}|\right)^{\ell_{j}}\mathrm{e}^{\mathrm{i}y_{j}}\,\mathrm{d}y_{j}\right)\\
 & = & \int_{\RR^{n}}\left(\prod_{j=1}^{n}h\left(x,y_{j}\right)\left(\log|y_{j}|\right)^{\ell_j}\mathrm{e}^{\mathrm{i}y_{j}}\right)\,\mathrm{d}y,
\end{eqnarray*}
which is in $\C^{\exp}\left(X\right)$ by Theorem \ref{thm:main interp}. It
follows that $\C^{\exp}\left(X\right)$ is closed under multiplication. \end{proof}
\begin{rems}
\label{rem: important}$\ $
\begin{enumerate}
\item \label{enu:Cexp=00003DE} Theorem \ref{thm:main interp} and Corollary
\ref{cor:CexpRing} imply that $\C^{\exp}$ is indeed the smallest
collection of $\mathbb{C}$-algebras containing $\S \cup\{\mathrm{e}^{\mathrm{i}f}:\ f\in\S \}$
and stable under parametric integration. Hence $\C^{\exp}=\E$.
\item \label{enu: real and imag parts, fourier} Notice that $\C^{\exp}$
is closed under complex conjugation, hence the real and imaginary
parts of functions in $\C^{\exp}$ are also in $\C^{\exp}$.
Moreover, $\C^{\exp}$ is closed under taking Fourier transforms
(over $\mathbb{R}^{m}$ and over $\mathbb{R}$ with parameters, as
in Equation (\ref{eq: fourier tranf param})).
\item \label{enu:fourier}$\mathcal{C}^{\exp}\left( X \right)$ can also be
described as the smallest $\mathbb{C}$-algebra $\mathcal{A}\left(X\right)$
containing $\mathcal{C}\left(X\right)$ and stable by composition
with subanalytic functions (the operation defined in Remark \ref{rem: table under right composition with sa},
where we take $n=m$) and by taking parametric Fourier transform (the
operation defined in Equation (\ref{eq: fourier tranf param})). To
see this, notice that Remark \ref{rem: table under right composition with sa}
and the previous remark imply that $\mathcal{A}\left(X\right)\subseteq\mathcal{C}^{\exp}\left(X\right)$.
To prove the other inclusion, notice first that a function of type
$\gamma_{h,\ell}$ (as in Definition \ref{def: gamma function}) is
a parametric Fourier transform of a function in $\mathcal{C}\left(X\right)$.
To see this, remark that the parametric Fourier transform of the function
$t\mapsto h\left(x,t\right)\left(\log t\right)^{\ell}$ is the function
\[
F\left(x,y\right)=\int_{\mathbb{R}}h\left(x,t\right)\left(\log t\right)^{\ell}\mathrm{e}^{-2\pi\mathrm{i}ty}\,\mathrm{d}t
\]
and we have $\gamma_{h,\ell}\left(x\right)=F\left(x,-\frac{1}{2\pi}\right)$,
where evaluating $F$ at the points $\left(x,-\frac{1}{2\pi}\right)$
is allowed, thanks to the stability by composition with subanalytic
functions. Moreover, the function $\left(x_{1},\ldots,x_{m}\right)\mapsto\mathrm{e}^{\mathrm{i}x_{1}}$
belongs to $\mathcal{A}\left(X\right)$, since the functions $\frac{\sin x_{1}}{x_{1}},\frac{\cos x_{1}}{x_{1}}$
are Fourier transforms of the characteristic function of a suitable
interval (see for example \cite{gasquet-witomski:fourier-analysis}).
Finally, by stability under composition with subanalytic functions,
if $\varphi\in\mathcal{S}\left(X\right)$, then $\mathrm{e}^{\mathrm{i}\varphi\left(x\right)}\in\mathcal{A}\left(X\right)$.
\end{enumerate}
\end{rems}
\bigskip{}

We now illustrate the main steps of the proof of Theorem \ref{thm:main interp}.
\begin{defn}
\label{def:generator} Consider a subanalytic set $X$. Call $f\colon X\to\CC$
a \textbf{\emph{generator for $\C^{\exp}\left(X\right)$}} if $f$ is of the
form
\begin{equation}
f\left(x\right)=g\left(x\right)\mathrm{e}^{\mathrm{i}\phi\left(x\right)}\gamma\left(x\right),\label{eq:generator}
\end{equation}
where $g\in\C\left(X\right)$, $\phi\in\S\left(X\right)$, and $\gamma=\gamma_{h,\ell}$
for some $\ell\in\NN$ and $h\in\S\left(X\times\RR\right)$ with $t\mapsto h\left(x,t\right)$
in $L^{1}\left(\RR\right)$. When $\gamma=1$, we shall also call $f$ a \textbf{\emph{generator
for $\C_{\naive}^{\exp}\left(X\right)$}}. Note that a function is in $\C^{\exp}\left(X\right)$
if and only if the function can be expressed as a finite sum of generators
for $\C^{\exp}\left(X\right)$, and likewise for $\C_{\naive}^{\exp}\left(X\right)$.\end{defn}
\begin{rem}
\label{rem: presentation of generators}The function $f$ given in
(\ref{eq:generator}) is determined by the data $\left(g,\phi,h,l\right)$. However,
the choice of underlying data is not uniquely determined by the function
$f$ itself (see Remark \ref{rem: gamma subanalytic}). In what follows,
we shall always assume that when we have a generator $f$, a choice
of underlying data has been specified.
\end{rem}

The purpose of the next two definitions is to identify a particular
type of generators for $\C^{\exp}\left(X\times\RR^{n}\right)$ which are integrable
everywhere and whose integral can be computed using the Fubini-Tonelli
Theorem.
\begin{defn}
\label{abs} To the function (\ref{eq:generator}) we associate the
function $f^{\abs}\colon X\to[0,+\infty)$ defined by
\[
f^{\abs}\left(x\right):=\left|g\left(x\right)\right|\gamma^{\abs}\left(x\right),
\]
where $\gamma^{\abs}\colon X\to[0,+\infty)$ is defined by
\[
\gamma^{\abs}\left(x\right)={ \int_{\RR}\left|h\left(x,t\right)\left(\log|t|\right)^{\ell} \right|\,\mathrm{d}t.}
\]

\end{defn}
For $f$ as in (\ref{eq:generator}), note that for any $x\in X$
we have $|\gamma\left(x\right)|\leq\gamma^{\text{abs}}\left(x\right)$, so $\left|f\left(x\right)\right|\leq f^{\text{abs}}\left(x\right)$,
and these inequalities can be strict. Observe that for any given generator
$f$ for $\mathcal{C}^{\exp}\left(X\right)$, $f^{\text{abs}}$ is uniquely determined
by the underlying data used to define $f$ as in (\ref{eq:generator}),
not by the function $f$ itself.
\begin{defn}
\label{def:superinteg} We say that a generator $f$ for $\C^{\exp}\left(X\times\RR^{n}\right)$
is \textbf{\emph{superintegrable over $X$}} if $f^{\abs}\left(x,\cdot\right)\in L^{1}\left(\RR^{n}\right)$
for all $x\in X$.
\end{defn}
In Section \ref{s:integGen} we will prove the following result.
\begin{prop}
[Integration of superintegrable generators]\label{prop:integGen-1}
Let $f$ be a generator for $\C^{\exp}\left(X\times\RR^{n}\right)$ that is superintegrable
over $X$, and define $F\colon X\to\CC$ by
\[
F\left(x\right)=\int_{\RR^{n}}f\left(x,y\right)\,\mathrm{d}y.
\]
Then $F\in\C^{\exp}\left(X\right)$.
\end{prop}
The key step to the proof of Theorem \ref{thm:main interp} is given
by the following interpolation result, which holds whenever we integrate
with respect to a single variable $y\in\mathbb{R}$. This result also
gives a structure theorem for the locus of integrability of functions
in $\C^{\exp}\left(X\times\RR\right)$.
\begin{thm}
[Interpolation and locus]\label{thm:interp-1} Let $f\in\C^{\exp}\left(X\times\RR\right)$
for some subanalytic set $X\subseteq\mathbb{R}^{m}$. Then there exists
$g\in\C^{\exp}\left(X\times\RR\right)$ such that
\[
\INT\left(g,X\right)=X
\]
and
\[
f\left(x,y\right)=g\left(x,y\right),\quad\text{for all \ensuremath{x\in\INT\left(f,X\right)}\ and\ all \ensuremath{y\in\RR}.}
\]
Moreover, $g$ can be written as a finite sum of generators for $\C^{\exp}\left(X\times\RR\right)$
that are superintegrable over $X$.

Finally, there exists $h\in\C^{\exp}\left(X\right)$ such that
\[
\INT\left(f,X\right)=\{x\in X:h\left(x\right)=0\}.
\]

\end{thm}
Once we have established Theorem \ref{thm:interp-1}, the proof of
Theorem \ref{thm:main interp} follows easily: the case $n=1$ is
implied by Theorem \ref{thm:interp-1} and Proposition \ref{prop:integGen-1}.
For $n>1$ we will use Fubini's Theorem and induction on the number
of variables with respect to which we integrate, as explained below.
\begin{notation}
\label{notation about tuple}Write $\left(x,y\right)=\left(x_{1},\ldots,x_{m},y_{1},\ldots,y_{n}\right)$
for coordinates on $\RR^{m+n}$. For each $k\in\{1,\ldots,n\}$ and
$\Box\in\{<,>,\leq,\geq\}$, write $y_{\Box k}$ for $\left(y_{j}\right)_{j\Box k}$.
For example, $y_{<k}=\left(y_{1},\ldots,y_{k-1}\right)$ and $y_{\leq k}=\left(y_{1},\ldots,y_{k}\right)$,
and also $\Pi_{k}\left(y\right)=y_{\leq k}$ and $\Pi_{m+k}\left(x,y\right)=\left(x,y_{\leq k}\right)$. \end{notation}
\begin{proof}
[Proof of Theorem \ref{thm:main interp}]If $n=1$, then by Theorem
\ref{thm:interp-1} there exists $g\in\C^{\exp}\left(X\times\RR\right)$ such
that
\[
f\left(x,y\right)=g\left(x,y\right),\quad\text{for all \ensuremath{x\in\INT\left(f,X\right)}\ and\ all \ensuremath{y\in\RR}.}
\]
Moreover, $g$ is a finite sum of superintegrable generators. The
sum of their integrals belongs to $\C^{\exp}\left(X\right)$, thanks to Proposition
\ref{prop:integGen-1}, and gives us the required $F$.

Let $n>1$. By Fubini's Theorem, for all $x\in\text{Int}\left(f,X\right)$,
the function $g_{x}\colon y_{<n}\mapsto\int_{\mathbb{R}}f\left(x,y\right)\,\mathrm{d}y_{n}$
is defined for all $y_{<n}$ belonging to some set $E_{x}\subseteq\mathbb{R}^{n-1}$
such that the set $\mathbb{R}^{n-1}\setminus E_{x}$ has measure zero.
Moreover, $g_{x}$ is integrable with respect to $y_{<n}$, and $\int_{\mathbb{R}^{n}}f\left(x,y\right)\,\mathrm{d}y=\int_{E_{x}}g_{x}\left(y_{<n}\right)\,\mathrm{d}y_{<n}$.
If we apply the case $n=1$ just proved to the function $f$ seen
as an element of $\C^{\exp}\left(\widetilde{X}\times\mathbb{R}\right)$, where $\widetilde{X}=X\times\mathbb{R}^{n-1}$,
then we obtain the existence of $F_{1}\in\C^{\exp}\left(\widetilde{X}\right)$ such
that $\forall\left(x,y_{<n}\right)\in\text{Int}\left(f,\widetilde{X}\right),\ F_{1}\left(x,y_{<n}\right)=\int_{\mathbb{R}}f\left(x,y\right)\,\mathrm{d}y_{n}$.
So in particular 
$$\int_{\mathbb{R}^{n-1}}F_{1}\left(x,y_{<n}\right)\,\mathrm{d}y_{<n}=\int_{E_{x}}g_{x}\left(y_{<n}\right)\,\mathrm{d}y_{<n},$$
for all $x\in\text{Int}\left(f,X\right)$. By the inductive hypothesis
applied to $F_{1}$, we obtain the existence of $F\in\C^{\exp}\left(X\right)$
such that $\forall x\in\text{Int}\left(F_{1},X\right),\ F\left(x\right)=\int_{\mathbb{\mathbb{R}}^{n-1}}F_{1}\left(x,y_{<n}\right)\,\mathrm{d}y_{<n}$.
Note that this argument shows that $\text{Int}\left(f,X\right)\subseteq\text{Int}\left(F_{1},X\right)$,
hence we are done.
\end{proof}
The structure of the paper is the following.

In Section \ref{s:prepSubConstr}, we establish some notation and
we review a series of known results about subanalytic and constructible
functions. Such results are mainly due to \cite{lr:prep} and \cite{cluckers-miller:stability-integration-sums-products,cluckers_miller:loci_integrability}.

In Section \ref{s:integGen} we prove Proposition \ref{prop:integGen-1}.

Section \ref{s:prepExpConstr} is the core of the paper. In this section
we prove a preparation theorem for functions in $\C^{\exp}\left(X\times\RR\right)$,
Theorem \ref{thm:prepExpConstr}. This states that for each $f$ there
is a partition of $X\times\mathbb{R}$ into finitely many subanalytic
sets such that on each of these sets, $f$ can be written as a finite
sum of generators, each of which is either superintegrable, or \textquotedblleft{}naive
in the last variable\textquotedblright{} (see Definition \ref{def: naive in y}).
As a consequence of the proof of this theorem we obtain that the functions
in $\mathcal{C}^{\exp}$ are piecewise analytic (see Remark \ref{rem:piecewise analytic}).

In Section \ref{s: proof of main results} we complete the proof of
Theorem \ref{thm:interp-1}. In order to do this, we apply Theorem
\ref{thm:prepExpConstr}. Subsequently, we show that any nonzero linear
combination of non-integrable generators for $\C_{\naive}^{\exp}\left(\mathbb{R}\right)$
such that the arguments of the exponentials are distinct polynomials,
cannot be integrable (Proposition \ref{prop:goal1}(\ref{V big})). The proof of
this latter result uses the theory
continuously uniformly distributed maps and
is postponed to Section 
\ref{s: proof of main results}.

Finally, in Section \ref{sec:Asymptotic-expansions-of} we deduce
a series of consequences of our main results: we prove an asymptotic
result for elements of $\C_{\naive}^{\exp}\left(\mathbb{R}\right)$, we give
two examples of functions that are in $\C_{}^{\exp}\left(\mathbb{R}\right)$
but not in $\C_{\naive}^{\exp}\left(\mathbb{R}\right)$, we prove that $\mathcal{C}^{\exp}$
is stable under taking pointwise limits and also has an analogue for
parametric families of the completeness theorem for $L^{p}$-spaces.
Moreover, we prove that the extension of the Fourier transform to
$L^{2}\left(\RR^{n}\right)$ sends $\C^{\exp}\left(\RR^{n}\right)\cap L^{2}\left(\RR^{n}\right)$ onto
$\C^{\exp}\left(\RR^{n}\right)\cap L^{2}\left(\RR^{n}\right)$.\medskip{}

For the reader's convenience, we describe the dependence relations
between the results in the two following diagrams.

The first diagram concerns the stability of $\C^{\exp}$ under integration (Theorem \ref{thm:main interp}). 
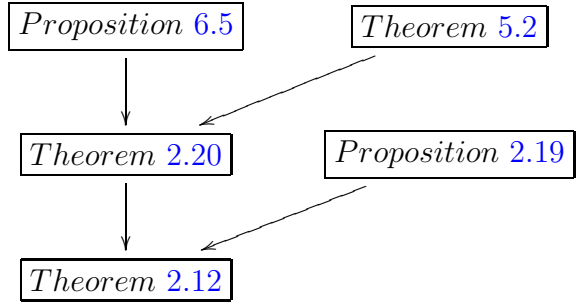
\begin{figure}[h!]
\[
\xymatrix{ &  & *+[F]+{Proposition\ \ref{prop:goal1}}\ar[d] & *+[F]+{Theorem\ \ref{thm:prepExpConstr}}\ar[ld]\\
 &  & *+[F]+{Theorem\ \ref{thm:interp-1}}\ar[d] & *+[F]+{Proposition\ \ref{prop:integGen-1}}\ar[ld]\\
 &  & *+[F]+{Theorem\ \ref{thm:main interp}}
}
\]
\end{figure}
\vskip1cm
\vglue0cm
 The second diagram concerns $L^p$-completeness
 and the 
stability of $\C^{\exp}$ under the Fourier-Plancherel transform (Proposition \ref{prop:complete} and Theorem 
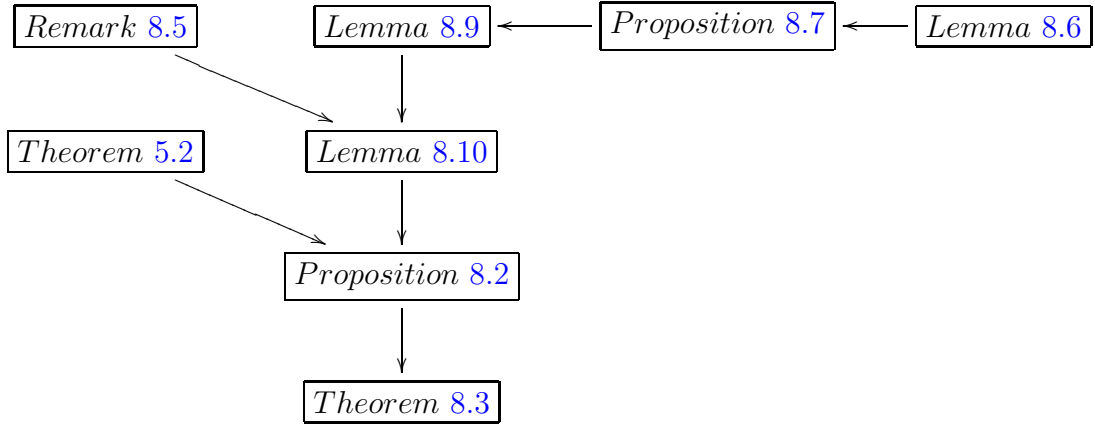
\ref{cor:Plancherel}).
\vglue0cm
\begin{figure}[h!]
\[
\xymatrix{ 
   *+[F]+{Remark \ \ref{rem:Param:geometricIntervals}}\ar[rd] 
& *+[F]+{Lemma\ \ref{lem:Param:OscPrep}}\ar[d] & 
  *+[F]+{Proposition \ \ref{prop:WeylToCUD}} \ar[l] & *+[F]+{Lemma\ \ref{lemma:Weyl}} \ar[l] 
\\
  *+[F]+{Theorem \ \ref{thm:prepExpConstr}}\ar[rd] 
 & *+[F]+{Lemma \ \ref{lem:Param:Dovetail}}\ar[d] 
 &  & \\
   & *+[F]+{Proposition \ \ref{prop:complete}}\ar[d] & \\
   & *+[F]+{Theorem\ \ref{cor:Plancherel}} &
}
\]
\end{figure}
\section{Preparation of subanalytic and constructible Functions}
\label{s:prepSubConstr}

This section gives the version of the preparation theorem for subanalytic
and constructible functions that we shall use throughout the paper.
It is mostly a review of ideas from \cite{lr:prep,cluckers_miller:loci_integrability}
but is formulated in a way that is convenient for our current purposes.
\begin{rem}
\label{rem: lojas}It is well known that every subanalytic function
of one variable admits a convergent Puiseux expansion at $+\infty$
(see for example \cite{lo:ana,vdd:tarski:general}). More precisely,
if $g\in\mathcal{S}\left(\mathbb{R}\right)$, then there are $c\in\mathbb{R},\ d\in\mathbb{N}$,
$r\in\mathbb{Q}$ (which can be chosen as an integer multiple of $\frac{1}{d}$)
and an absolutely convergent power series $H\in\mathbb{R}\{ y\} $,
with $H\left(0\right)=0$, such that for $x$ sufficiently large,
\begin{equation}
g\left(x\right)=cx^{r}\left(1+H\left(x^{-\frac{1}{d}}\right)\right).\label{eq:loja}
\end{equation}

In particular, for $x$ large, $g$ can be written as
\begin{equation}
g\left(x\right)=p\left(x^{\frac{1}{d}}\right)+g_{0}\left(x\right),\label{eq: poly + rest}
\end{equation}
where $p\in\mathbb{R}[y]$, with $p\left(0\right)=0$,
and $g_{0}$ is a \emph{bounded } subanalytic function.
\end{rem}
The subanalytic preparation theorem given in \cite[Théorème 1]{lr:prep}
can be viewed as a parametric version (in several variables) of the
above remark, and the constructible preparation theorem given in \cite[Corollary 3.5]{cluckers_miller:loci_integrability}
is the natural extension of this latter result to the context of constructible
functions.

We fix some notation.
\begin{defn}
\label{def:openOver} A set $A\subseteq\RR^{m+n}$ is \textbf{\emph{open
over $\RR^{m}$}} if the fibre $A_{x}$ is open in $\RR^{n}$ for
all $x\in\Pi_{m}\left(A\right)$.

For any set $X\subseteq\RR^{m}$, call a map $f\colon X\to\RR^{n}$
\textbf{\emph{analytic}} if $f$ extends to an analytic map on
a neighbourhood of $X$ in $\RR^{m}$.
\end{defn}
Recall Notation \ref{notation about tuple}.
\begin{defn}
\label{def:cell} A set $A\subseteq\RR^{m+n}$ is a 
\textbf{\emph{cell
over $\RR^{m}$}} if $A$ is subanalytic and for each $j\in\{1,\ldots,n\}$,
$\Pi_{m+j}\left(A\right)$ is either the graph of an analytic function in $\S\left(\Pi_{m+j-1}\left(A\right)\right)$
or else
\[
\Pi_{m+j}\left(A\right)=\left\{\left(x,y_{\leq j}\right):\ \left(x,y_{<j}\right)\in\Pi_{m+j-1}\left(A\right),a_{j}\left(x,y_{<j}\right)<y_{j}<b_{j}\left(x,y_{<j}\right)\right\}
\]
for some analytic, subanalytic functions $a_{j}\left(x,y_{<j}\right)<b_{j}\left(x,y_{<j}\right)$,
where we also allow the possibility that $a_{j}\equiv-\infty$ and
the possibility that $b_{j}\equiv+\infty$.

If $m=0$, we will just say that $A$ is a subanalytic cell.
\end{defn}

\begin{defn}
\label{def:center} Let $A\subseteq\RR^{m+1}$ be a cell over $\RR^{m}$
that is open over $\RR^{m}$, and write $\left(x,y\right)=\left(x_{1},\ldots,x_{m},y\right)$
for coordinates on $\RR^{m+1}$. Call $\theta\colon \Pi_{m}\left(A\right)\to\RR$ a
\textbf{\emph{centre for $A$}} if the following hold.
\begin{enumerate}
\item $\theta$ is an analytic subanalytic function.
\item The graph of $\theta$ is disjoint from $A$ and is either contained
in, or is disjoint from, the closure of $A$ in $\Pi_{m}\left(A\right)\times\RR$.
\item The image of $\theta$ is contained in one of the sets $\left(-\infty,0\right)$,
$\{0\}$ or $\left(0,+\infty\right)$. Moreover, when $\theta\neq0$, the closure
of $\{\vert y/\theta\left(x\right)\vert :\ \left(x,y\right)\in A\}$ in $\RR$ is a compact subset of
$\left(0,+\infty\right)$.
\item The set $\{ y-\theta\left(x\right) :\ \left(x,y\right)\in A\}$ is contained in one of the
sets $\left(-\infty,-1\right)$, $\left(-1,0\right)$, $\left(0,1\right)$ or $\left(1,+\infty\right)$.
\end{enumerate}
Note that when $\theta$ is a centre for $A$, there exist unique
$\sigma,\tau\in\{-1,1\}$ such that
\[
\{\sigma\left(y-\theta\left(x\right)\right)^{\tau}:\ \left(x,y\right)\in A\}\subseteq\left(1,+\infty\right),
\]
and $A$ is of the form
\begin{equation}
A=\left\{\left(x,y\right):\ x\in\Pi_{m}\left(A\right),a\left(x\right)<\sigma\left(y-\theta\left(x\right)\right)^{\tau}<b\left(x\right)\right\}\label{eq:cellCenter}
\end{equation}
for some analytic subanalytic functions $1\leq a\left(x\right)<b\left(x\right)$, where
either $b<+\infty$ on $\Pi_{m}\left(A\right)$ or $b\equiv+\infty$ on $\Pi_{m}\left(A\right)$.

Define $P_{\theta}=\left(P_{\theta,1},\ldots,P_{\theta,m+1}\right)\colon\ \Pi_{m}\left(A\right)\times\left(1,+\infty\right)\to\Pi_{m}\left(A\right)\times\RR$
by
\begin{equation}
P_{\theta}\left(x,y\right)=\left(x,\sigma y^{\tau}+\theta\left(x\right)\right).\label{eq:Ptheta}
\end{equation}
Define $A_{\theta}=P_{\theta}^{-1}\left(A\right)$,
	and note that
	\begin{equation}
	A_{\theta}=\left\{ \left(x,y\right):x\in\Pi_{m}\left(A\right),a\left(x\right)<y<b\left(x\right)\right\}\label{eq:Atheta} 
	\end{equation}
	and that $P_{\theta}$ restricts to a bijection $P_{\theta}\colon A_{\theta}\rightarrow A$
	whose inverse is given by
	\[
	P_{\theta}^{-1}\left(x,y\right)=\left(x,\sigma\left(y-\theta\left(x\right)\right)^{\tau}\right).
	\]
	We shall henceforth restrict $P_{\theta}$ to $A_{\theta}$, considering
	it to be a bijection from $A_{\theta}$ to $A$.
\end{defn}
\begin{rem}
\label{rmk:center} When $b\equiv+\infty$ then necessarily $\tau=1$ and the second
sentence of Condition (3) of Definition \ref{def:center} implies
that $\theta=0$.
\end{rem}

For any polyradius $r=\left(r_{1},\ldots,r_{N}\right)\in\left(0,+\infty\right)^{N}$, define
\[
\begin{aligned}B_{r}\left(\mathbb{\mathbb{C}}\right)= & \left\{ z\in\mathbb{\mathbb{C}}^{N}:\ |z_{1}|\leq r_{1},\ldots,|z_{N}|\leq r_{N}\right\} \\
\text{and}\  & B_{r}\left(\mathbb{R}\right)=B_{r}\left(\mathbb{\mathbb{C}}\right)\cap\mathbb{R}^{N},
\end{aligned}
\]
where $z=\left(z_{1},\ldots,z_{N}\right)$.
\begin{defn}
\label{def:psiFunction} Let  $\psi\colon X\to\RR^{N}$ be  a 
subanalytic map such that $\psi\left(X\right)\subseteq B_{r}\left(\mathbb{R}\right)$,  for some $r\in\left(0,+\infty\right)^{N}$.
We call $f\colon X\to\RR$ a \textbf{\emph{$\psi$-function}}
if there exists a real analytic function $F$ such that $f=F\circ\psi$ 
and $F$ is given by a single convergent
power series in $N$ variables, centred at $0$ and converging in
some open neighbourhood of $B_{r}\left(\mathbb{R}\right)$ in $\RR^{N}$.

Observe that $F$ extends uniquely to a complex analytic function
on a neighbourhood of $B_{r}\left(\mathbb{C}\right)$ in $\mathbb{C}^{N}$.

If we additionally have that
\[
|F\left(z\right)-1 |<1\ \ \text{for\ all\ }z\in B_{r}\left(\mathbb{C}\right),
\]
then we call $f$ a \textbf{\emph{$\psi$-unit}}. \end{defn}
\begin{rems}
Let $f=F\circ\psi$ be a $\psi$-unit, with $r$ as above.
\begin{enumerate}
\item There exist strictly positive constants $k<K$ such that $k<|F\left(x\right) |<K$
for every $x\in B_{r}\left(\mathbb{R}\right)$.
\item The set $F\left(B_{r}\left(\mathbb{C}\right)\right)$ is compact.
Therefore there exists $\varepsilon\in\left(0,1\right)$ such that
$|F\left(z\right)-1|<{1-\varepsilon}$ for all $z\in B_{r}\left(\mathbb{C}\right)$.
\item The previous remark shows that the natural logarithm extends to a
holomorphic function on a neighbourhood of $F\left(B_{r}\left(\mathbb{C}\right)\right)$
in $\mathbb{C}^{N}$, so $\log F$ is given by a single convergent
power series on $B_{r}\left(\mathbb{C}\right)$ centred at $0$.
Therefore $\log f:X\rightarrow\mathbb{R}$ is a $\psi$-function.
\end{enumerate}
\end{rems}
\begin{defn}
\label{def:prepared} Consider the cell $A$ in (\ref{eq:cellCenter})
and a bounded, analytic, subanalytic map $\psi$, defined on
$A$, of the form
\[
\begin{aligned}\psi\left(x,y\right)=\left(c_{1}\left(x\right),\ldots,c_{N}\left(x\right),\left(\frac{a\left(x\right)}{\sigma\left(y-\theta\left(x\right)\right)^{\tau}}\right)^{1/d},\left(\frac{\sigma\left(y-\theta\left(x\right)\right)^{\tau}}{b\left(x\right)}\right)^{1/d}\right), & \ \text{if}\ b<+\infty,\\
\psi\left(x,y\right)=\left(c_{1}\left(x\right),\ldots,c_{N}\left(x\right),\left(\frac{a\left(x\right)}{\sigma\left(y-\theta\left(x\right)\right)^{\tau}}\right)^{1/d}\right), & \ \text{if}\ b\equiv+\infty,
\end{aligned}
\]
for some positive integer $d$ and some analytic functions $c_{1},\ldots,c_{N}$.

We say that a subanalytic function $f\colon A\to\RR$ is \textbf{\emph{$\psi$-prepared}}
if
\[
f\left(x,y\right)=f_{0}\left(x\right)|y-\theta\left(x\right)|^{\nu}u\left(x,y\right)
\]
on $A$ for some analytic $f_{0}\in\S\left(\Pi_{m}\left(A\right)\right)$, $\nu\in\QQ$
and $u$ a $\psi$-unit. \end{defn}
\begin{rem}
\label{psi when the center in zero}We shall frequently apply this
concept to the situation when $A=A_{\theta}$ (namely, $\theta=0$
and $\sigma=\tau=1$), in which case
\begin{equation}
\begin{aligned}\psi\left(x,y\right)=\left(c_{1}\left(x\right),\ldots,c_{N}\left(x\right),\left(\frac{a\left(x\right)}{y}\right)^{1/d},\left(\frac{y}{b\left(x\right)}\right)^{1/d}\right),\  & \text{if}\ b<+\infty,\\
\psi\left(x,y\right)=\left(c_{1}\left(x\right),\ldots,c_{N}\left(x\right),\left(\frac{a\left(x\right)}{y}\right)^{1/d}\right), & \ \text{if}\ b\equiv+\infty,
\end{aligned}
\label{eq:psiTheta}
\end{equation}
and
\begin{equation}
f\left(x,y\right)=f_{0}\left(x\right)y^{\nu}u\left(x,y\right),\label{eq:psi-prepared}
\end{equation}
on $A_{\theta}$ for some analytic $f_{0}\in\S\left(\Pi_{m}\left(A\right)\right)$, $\nu\in\QQ$
and $u$ a $\psi$-unit. 
\end{rem}
\begin{prop}
[Preparation of Constructible Functions] \label{prop:prepSubConstr}
Let $D\subseteq\RR^{m+1}$ be subanalytic and $\F\subseteq\C\left(D\right)$
be a finite set of constructible functions. Then there exists a finite
partition $\A$ of $D$ into cells over $\RR^{m}$ such that for each
$A\in\A$ that is open over $\RR^{m}$ there exists a centre $\theta$
for $A$ such that, for each $f\in\F$, we can write $f\circ P_{\theta}$
as a finite sum
\begin{equation}
f\circ P_{\theta}\left(x,y\right)=\sum_{j\in J}g_{j}\left(x\right)y^{r_{j}}\left(\log y\right)^{s_{j}}h_{j}\left(x,y\right)\label{eq:prepConstr}
\end{equation}
on $A_{\theta}$, where:
\begin{enumerate}
\item  $A_{\theta}$ is as in Equation (\ref{eq:Atheta});
\item $P_{\theta}$ is as in Equation (\ref{eq:Ptheta});
\item the functions $h_{j}$ are $\psi$-functions (see Definition \ref{def:psiFunction}),
where $\psi$ is as in Equation (\ref{eq:psiTheta}) for some analytic functions
$c_{1},\ldots,c_{N}$ and some integer $d>0$;
\item $s_{j}\in\NN$ and the $r_{j}$ are integer multiples of $1/d$;
\item the functions $g_{j}$ are analytic and in $\C\left(\Pi_{m}\left(A\right)\right)$.
\end{enumerate}
\end{prop}
\begin{proof}
We apply \cite[Corollary 3.5]{cluckers_miller:loci_integrability}
and we obtain a cell decomposition $\mathcal{A}$ such that Equation
(\ref{eq:prepConstr}) holds, with Conditions (1) and (2) satisfied.
Up to refining $\mathcal{A}$, we may assume that (5) also holds. 
We must now show that, up to some refinement of $\mathcal{A}$, we may
assume that Conditions (3) and (4) hold as well. 
By \cite[Corollary 3.5]{cluckers_miller:loci_integrability},
we know that a weaker version of Condition (3) holds, namely the
$h_{j}$ are of the form $\widetilde{F_{j}}\circ\widetilde{\psi}$, where
$\widetilde{F_{j}}$ is a power series converging on some open set $O_{j}$
containing the closure of the image of $\widetilde{\psi}$ and $\widetilde{\psi}$
is a bounded map whose components are
\[
c_{1}\left(x\right),\ldots,c_{M}\left(x\right),\left(e_{1}\left(x\right)/y\right)^{1/d},\left(e_{2}\left(x\right)y\right)^{1/d}
\]
for some $M\geq0$, some $d>0$, and some analytic subanalytic functions
$c_{1},\ldots,c_{M},e_{1},e_{2}$. We now explain, in the case that
$b\left(x\right)<+\infty$, how we can obtain the quotients $a\left(x\right)/y$ and ~$y/b\left(x\right)$
as arguments instead of $e_{1}\left(x\right)/y$ and $e_{2}\left(x\right)y$ (the case $b\left(x\right)=+\infty$
is similar and even easier). Since $e_{1}\left(x\right)/y$ and $e_{2}\left(x\right)y$
are bounded, and since $y$ runs from $a\left(x\right)$ to $b\left(x\right)$, one has
that $e_{1}\left(x\right)/a\left(x\right)$ and $e_{2}\left(x\right)b\left(x\right)$ are also bounded. Let $\psi\left(x,y\right)$
be
\[
\left(c_{1}\left(x\right),\ldots,c_{M}\left(x\right),\left(e_{1}\left(x\right)/a\left(x\right)\right)^{1/d},\left(e_{2}\left(x\right)b\left(x\right)\right)^{1/d},\left(\frac{a\left(x\right)}{y}\right)^{1/d},\left(\frac{y}{b\left(x\right)}\right)^{1/d}\right)
\]
and
\[
\pi:\mathbb{R}^{M+4}\ni\left(z_{1},\ldots,z_{M},Z_{1},Z_{2},Z_{3},Z_{4}\right)\mapsto\left(z_{1},\ldots,z_{M},Z_{1}Z_{3},Z_{2}Z_{4}\right)\in\mathbb{R}^{M+2}.
\]
Then $\psi$ is bounded, $\widetilde{\psi}=\pi\circ\psi$ and the closure
of the image of $\psi$ is contained in the open sets $\pi^{-1}\left(O_{j}\right)$.
We rename $c_{M+1}=\left(e_{1}\left(x\right)/a\left(x\right)\right)^{1/d}$ and $c_{M+2}=\left(e_{2}\left(x\right)b\left(x\right)\right)^{1/d}$,
and set $N=M+2$. It is clear that the power series $F_{j}=\widetilde{F_{j}}\circ\pi$
converge on $\pi^{-1}\left(O_{j}\right)$ and that $F_{j}\left(\psi\right)=\widetilde{F_{j}}\left(\widetilde{\psi}\right)$
on $A_{\theta}$. Hence Condition (3) holds.

Finally, by replacing $d$ by an integer multiple if necessary, we
can assume that condition (4) also holds.\end{proof}
\begin{rem}
\label{rem: prep subana with center}We have stated Proposition \ref{prop:prepSubConstr}
in the transformed coordinates (via $P_{\theta}$) out of convenience.
In the original coordinates, (\ref{eq:prepConstr}) becomes
\[
f\left(x,y\right)=\sum_{j\in J}\tld{g}_{j}\left(x\right)\left|y-\theta\left(x\right)\right|^{\tld{r}_{j}}\left(\log|y-\theta\left(x\right)|\right)^{s_{j}}\tld{h}_{j}\left(x,y\right)
\]
on $A$, where $\tld{g}_{j}\left(x\right)=\tau^{s_{j}}g_{j}\left(x\right)$, $\tld{r}_{j}=\tau r_{j}$
and $\tld{h}\left(x,y\right)=h\circ P_{\theta}^{-1}\left(x,y\right)$.
\end{rem}
\begin{rem}
\label{rem: prep of subana} If $\mathcal{F}\subseteq\mathcal{S}\left(D\right)$
is a finite collection of subanalytic functions, then the proof of
Proposition \ref{prop:prepSubConstr} (where we replace the use of \cite[Corollary 3.5]{cluckers_miller:loci_integrability}
by the use of  \cite[Theorem 3.4]{cluckers_miller:loci_integrability}) 
shows that for each $f\in\F$,
on $A_{\theta}$ we can write $f\circ P_{\theta}$ in the $\psi$-prepared
form in Equation (\ref{eq:psi-prepared}).
In addition, it follows from the proof of the subanalytic preparation theorem in \cite{lr:gabriel} that if $\varepsilon\in(0,1)$ is given beforehand, then the preparation can be constructed so that each $\psi$-unit $u$, as given in Equation  (\ref{eq:psi-prepared}), is within $\varepsilon$ of $1$, by which we mean that $u = U\circ\psi$ for some $r\in(0,\infty)^M$ (where $M = N+2$ when $b < +\infty$, and $M = N+1$ when $b \equiv +\infty)$ such that $\psi(A_\theta) \subseteq B_r(\RR)$ and some real analytic function $U$ on $B_r(\RR)$ that extends to a complex analytic function on a neighbourhood of $B_r(\CC)$ in $\CC^M$ such that $|U(z) - 1| < \varepsilon$ for all $z\in B_r(\CC)$.
\end{rem}
\begin{rem}
\label{rem: distinct tuples etc}In the situation described in Proposition
\ref{prop:prepSubConstr}, we may also assume that the following two
properties hold when $b\equiv+\infty$. Let $J_{1}=\{j\in J:\ h_{j}=1\}$.
Then,
\begin{enumerate}
\item \label{enu:property 1}for each $j\in J\setminus J_{1}$, $r_{j}<-1$;
\item \label{enu:property 2}$\left(\left(r_{j},s_{j}\right)\right)_{j\in J_{1}}$ is a family
of distinct pairs in $\QQ\times\NN$.
\end{enumerate}
To see this, notice that because $b\equiv+\infty$, we may write $h_{j}$
as a convergent power series
\[
h_{j}\left(x,y\right)=\sum_{k=0}^{+\infty}h_{j,k}\left(x\right)\left(\frac{a\left(x\right)}{y}\right)^{k/d}
\]
for $\left(c_{1},\ldots,c_{N}\right)$-functions $h_{j,k}$. To obtain Property
(\ref{enu:property 1}), for each $j\in J$, fix $n_{j}\in\NN$ such
that
\[
r_{j}-\frac{n_{j}}{d}<-1,
\]
and write the $j$-th term of (\ref{eq:prepConstr}) as
\[
\sum_{k=0}^{n_{j}-1}g_{j}\left(x\right)h_{j,k}\left(x\right)a\left(x\right)^{k/d}y^{r_{j}-k/d}\left(\log y\right)^{s_{j}}+R_{j}\left(x,y\right),
\]
where
\[
R_{j}\left(x,y\right)=g_{j}\left(x\right)a\left(x\right)^{n_{j}/d}y^{r_{j}-n_{j}/d}\left(\log y\right)^{s_{j}}\left(\sum_{k=n_{j}}^{+\infty}h_{j,k}\left(x\right)\left(\frac{a\left(x\right)}{y}\right)^{\left(k-n_{j}\right)/d}\right).
\]
To obtain Property (\ref{enu:property 2}), simply sum up terms in (\ref{eq:prepConstr})
for $j\in J_{1}$ with equal powers $r_{j}$ and $s_{j}$.
\end{rem}
We now study the integrability properties of the prepared form given
in (\ref{eq:prepConstr}). The following remarks will be useful in
Sections \ref{s:integGen} and \ref{s: proof of main results}.
\begin{rems}
\label{rems: integrability of constructible fncts}Consider the situation
described in Proposition \ref{prop:prepSubConstr} for some $A\in\A$.
In the notation of Proposition \ref{prop:prepSubConstr}, for each
$j\in J$, write
\[
G_{j}\left(x,y\right):=g_{j}\left(x\right)y^{r_{j}}\left(\log y\right)^{s_{j}}h_{j}\left(x,y\right).
\]

\begin{enumerate}
\item \label{enu:jac of p theta}Note that we have
\[
\partial_{y}P_{\theta}\left(y\right):=\PD{}{P_{\theta,m+1}}{y}\left(x,y\right)=\sigma\tau y^{\tau-1},
\]
that $\tau-1$ equals either $0$ or $-2$, and that
\[
\INT\left(f\restriction{A},\Pi_{m}\left(A\right)\right)=\INT\left(\left(f\circ P_{\theta}\right)\partial_{y}P_{\theta},\Pi_{m}\left(A\right)\right).
\]

\item \label{enu: continuous}For each $j\in J$ and $x\in\Pi_{m}\left(A\right)$,
$y\mapsto G_{j}\left(x,y\right)$ extends to a continuous (in fact, analytic)
function on the closure in $\RR$ of the fibre $\left(A_{\theta}\right)_{x}$,
and likewise for $\partial_{y}P_{\theta}$.\\
In particular, when $b<+\infty$, $\INT\left(G_{j}\partial_{y}P_{\theta},\Pi_{m}\left(A\right)\right)=\Pi_{m}\left(A\right)$
for each $j\in J$.
\item \label{enu:little o}Let $b\equiv+\infty$ and recall Property (\ref{enu:property 1})
of Remark \ref{rem: distinct tuples etc}.\\
For each $j\in J\setminus J_{1}$ and $x\in\Pi_{m}\left(A\right)$, the function
$y\mapsto G_{j}\left(x,y\right)\partial_{y}P_{\theta}\left(y\right)$ is $o\left(y^{\tau-2}\right)$
as $y\to+\infty$, and is therefore integrable. \\
For each $j\in J_{1}$,
\[
G_{j}\left(x,y\right)\partial_{y}P_{\theta}\left(y\right)=\sigma\tau g_{j}\left(x\right)y^{r_{j}+\tau-1}\left(\log y\right)^{s_{j}},
\]
which is integrable in $y$ if and only if $g_{j}\left(x\right)=0$ or $r_{j}+\tau<0$.
\\
Therefore by defining
\[
J^{\INT}=\left(J\setminus J_{1}\right)\cup\{j\in J_{1}:r_{j}+\tau<0\},
\]
we see that for each $j\in J$,
\[
\INT\left(G_{j}\partial_{y}P_{\theta},\Pi_{m}\left(A\right)\right)=\begin{cases}
\Pi_{m}\left(A\right), & \text{if \ensuremath{j\in J^{\INT}},}\\
\left\{x\in\Pi_{m}\left(A\right):g_{j}\left(x\right)=0\right\}, & \text{if \ensuremath{j\in J\setminus J^{\INT}}.}
\end{cases}
\]

\item \label{enu:zero-set}In the situation of the previous remark, define
the constructible functions
\begin{eqnarray*}
g\left(x,y\right) & = & \sum_{j\in J^{\INT}}G_{j}\left(x,y\right),\quad\text{for all \ensuremath{\left(x,y\right)\in A_{\theta}},}\\
h\left(x\right) & = & \sum_{j\in J\setminus J^{\INT}}g_{j}^2\left(x\right),\quad\text{for all \ensuremath{x\in\Pi_{m}\left(A\right)}.}
\end{eqnarray*}
Then
\[
\INT\left(f\restriction{A},\Pi_{m}\left(A\right)\right)=\{x\in\Pi_{m}\left(A\right):h\left(x\right)=0\},
\]
\[
\INT\left(g\partial_{y}P_{\theta},\Pi_{m}\left(A\right)\right)=\Pi_{m}\left(A\right),
\]
and
\[
\text{\ensuremath{f\circ P_{\theta}\left(x,y\right)=g\left(x,y\right)}\ for all \ensuremath{\left(x,y\right)\in A_{\theta}} with\ \ensuremath{x\in\INT\left(f\restriction{A},\Pi_{m}\left(A\right)\right)}.}
\]
\\
To see this, note that the previous remark shows that $\INT\left(g\partial_{y}P_{\theta},\Pi_{m}\left(A\right)\right)=\Pi_{m}\left(A\right)$,
and clearly
\[
\text{\ensuremath{f\circ P_{\theta}=g\ }on the set \ensuremath{\{\left(x,y\right)\in A_{\theta}:h\left(x\right)=0\}},}
\]
so $\{x\in\Pi_{m}\left(A\right):h\left(x\right)=0\}\subseteq\INT\left(f\restriction{A},\Pi_{m}\left(A\right)\right)$.
To show that $\INT\left(f\restriction{A},\Pi_{m}\left(A\right)\right)\subseteq\{x\in\Pi_{m}\left(A\right):h\left(x\right)=0\}$,
note that if $x\in\Pi_{m}\left(A\right)$ is such that $h\left(x\right)\neq0$, then by
choosing $j_{0}$ in the set $\{j\in J\setminus J^{\INT}:g_{j}\left(x\right)\neq0\}$
with $\left(r_{j_{0}},s_{j_{0}}\right)$ greatest with respect to the lexicographical
order on $\QQ\times\RR$, it follows from Property (\ref{enu:property 2})
in Remark \ref{rem: distinct tuples etc} that
\[
\lim_{y\to+\infty}\frac{f\left(x,y\right)}{G_{j_{0}}\left(x,y\right)}=1,
\]
so $f\left(x,\cdot\right)\not\in L^{1}\left(A_{x}\right)$ by the previous remark.
\item \label{enu: integrability locus when constructible}In particular,
if $\INT\left(f,X\right)=X$, then Remarks (\ref{enu: continuous}) and (\ref{enu:zero-set})
above show that for each $j\in J$, we have $\INT\left(G_{j}\partial_{y}P_{\theta},\Pi_{m}\left(A\right)\right)=\Pi_{m}\left(A\right)$.
\end{enumerate}
\end{rems}

\section{Integrating superintegrable generators}

\label{s:integGen}

This section is dedicated to the proof of Proposition \ref{prop:integGen-1},
of which we recall the statement.
\begin{prop*}
Let $f$ be a generator for $\C^{\exp}\left(X\times\RR^{n}\right)$ that is superintegrable
over $X$, and define $F\colon X\to\CC$ by
\[
F\left(x\right)=\int_{\RR^{n}}f\left(x,y\right)\,\mathrm{d}y\quad\text{.}
\]
Then $F\in\C^{\exp}\left(X\right)$. \end{prop*}
\begin{proof}
Assume that $X\subseteq\RR^{m}$, and write
\[
f\left(x,y\right)=g\left(x,y\right)\mathrm{e}^{\mathrm{i}\phi\left(x,y\right)}\gamma\left(x,y\right),\quad\text{for \ensuremath{\left(x,y\right)\in X\times\RR^{n}},}
\]
where $g\in\C\left(X\times\RR^{n}\right)$, $\phi\in\S\left(X\times\RR^{n}\right)$ and
$\gamma=\gamma_{h,\ell}$ for some $\ell\in\NN$ and $h\in\S\left(X\times\RR^{n}\times\RR\right)$
with $\INT\left(h,X\times\RR^{n}\right)=X\times\RR^{n}$.

Because $|f\left(x,y\right)|\leq f^{\abs}\left(x,y\right)$ for all $\left(x,y\right)\in X\times\RR^{n}$
(see Definition \ref{abs}), it follows that $f\left(x,\cdot\right)\in L^{1}\left(\RR^{n}\right)$
for all $x\in X$ . Moreover, the Fubini-Tonelli theorem shows that
for each $x\in X$,
\[
\left(y,t\right)\mapsto g\left(x,y\right)h\left(x,y,t\right)\left(\log|t|\right)^{\ell}
\]
is in $L^{1}\left(\RR^{n}\times\RR\right)$, and the iterated integral
\[
\int_{\RR^{n}}f\left(x,y\right)\,\mathrm{d}y=\int_{\RR^{n}}\left(\int_{\RR}g\left(x,y\right)\mathrm{e}^{\mathrm{i}\phi\left(x,y\right)}h\left(x,y,t\right)\left(\log|t|\right)^{\ell}\mathrm{e}^{\mathrm{i}t}\,\mathrm{d}t\right)\,\mathrm{d}y
\]
can be computed as a product integral
\[
\int_{\RR^{n}\times\RR}g\left(x,y\right)\mathrm{e}^{\mathrm{i}\phi\left(x,y\right)}h\left(x,y,t\right)\left(\log|t|\right)^{\ell}\mathrm{e}^{\mathrm{i}t}\,\mathrm{d}y\wedge\mathrm{d}t.
\]
Therefore up to replacing $n$ by $n+1$, we may simply assume that
$\gamma=1$.

Now construct a finite partition $\A$ of $X\times\RR^{n}$ into cells
over $\RR^{m}$ such that for each $A\in\A$ that is open over $\RR^{m}$,
either $\phi\left(x,y\right)=\phi_{0}\left(x\right)$ on $A$ for some $\phi_{0}\in\S\left(\Pi_{m}\left(A\right)\right)$,
or else the function $y\mapsto\phi\left(x,y\right)$ is $C^{1}$ on $A_{x}$
with $\sigma\PD{}{\phi}{y_{j}}>0$ on $A_{x}$ for some $\sigma\in\{-1,1\}$
and $j\in\{1,\ldots,n\}$.

When $\phi=\phi_{0}$,
\begin{equation}
\int_{A_{x}}f\left(x,y\right)\,\mathrm{d}y=\mathrm{e}^{\mathrm{i}\phi_{0}\left(x\right)}\int_{A_{x}}g\left(x,y\right)\,\mathrm{d}y,\quad.\label{eq:constantPhase}
\end{equation}
The fact that $\C$ is stable under integration \cite{cluckers-miller:stability-integration-sums-products,cluckers_miller:loci_integrability}
shows that the integral of $g$ with respect to $y$ is in $\C\left(\Pi_{m}\left(A\right)\right)$.
Hence, (\ref{eq:constantPhase}) is in $\C_{\naive}^{\exp}\left(\Pi_{m}\left(A\right)\right)$.

In the other case, by pulling back by the inverse of the map $\left(x,y\right)\mapsto\left(x,\phi\left(x,y\right),y_{<j},y_{>j}\right)$
and multiplying by the Jacobian of this map, we may simply assume
that $\phi\left(x,y\right)=y_{1}$. Write $\tld{A}=\Pi_{m+1}\left(A\right)$, and note that
the function $\tld{g}\colon \tld{A}\to\RR$ defined by
\[
\tld{g}\left(x,y_{1}\right)=\int_{A_{\left(x,y_{1}\right)}}g\left(x,y_{1},y_{>1}\right)\,\mathrm{d}y_{>1},\quad\text{for each \ensuremath{\left(x,y_{1}\right)\in\tld{A}},}
\]
is constructible and that
\[
\int_{A_{x}}g\left(x,y\right)\mathrm{e}^{\mathrm{i}\phi\left(x,y\right)}\,\mathrm{d}y=\int_{\tld{A}_{x}}\tld{g}\left(x,y_{1}\right)\mathrm{e}^{\mathrm{i}y_{1}}\,\mathrm{d}y_{1},\quad\text{for each \ensuremath{x\in\Pi_{m}\left(A\right)}.}
\]

We apply Proposition \ref{prop:prepSubConstr} to $\tld{g}\left(x,y_{1}\right)$
and then work piecewise, thereby focusing on one open cell $\tld{B}\subseteq\tld{A}$
given by the preparation which is open over $\RR^{m}$.

By applying Remark \ref{rems: integrability of constructible fncts}(\ref{enu: integrability locus when constructible}),
we may write
\[
\int_{\tld{B}_{x}}\tld{g}\left(x,y_{1}\right)\mathrm{e}^{\mathrm{i}y_{1}}\,\mathrm{d}y_{1}
\]
as a finite sum of terms of the form
\begin{equation}
g_{0}\left(x\right)\int_{\tld{B}_{x}}|y_{1}-\theta\left(x\right)|^{r}\,\tld{u}\left(x,y_{1}\right)\left(\log|y_{1}-\theta\left(x\right)|\right)^{s}\mathrm{e}^{\mathrm{i}y_{1}}\,\mathrm{d}y_{1},\label{eq: star}
\end{equation}
where $g_{0}\in\C\left(\Pi_{m}\left(\tld{B}\right)\right)$, $r\in\QQ$, $s\in\NN$, $\tld{u}$
is a $\psi$-function (for some $\psi$) and $\theta$ is the centre
given by the preparation on $\tld{B}$. Thus for some $\sigma\in\{-1,1\}$, by applying
the coordinate change $\left(x,y_{1}\right)\mapsto\left(x,\sigma y_{1}+\theta\left(x\right)\right)$
we may write (\ref{eq: star}) as
\begin{equation}
\sigma g_{0}\left(x\right)\mathrm{e}^{\mathrm{i}\theta\left(x\right)}\int_{B_{x}}y_{1}^{r}\left(\log y_{1}\right)^{s}u\left(x,y_{1}\right)\mathrm{e}^{\mathrm{i}\sigma y_{1}}\,\mathrm{d}y_{1},\label{eq: double star}
\end{equation}
where $B_{x}\subseteq\left(0,+\infty\right)$ and $u$ are the pullbacks of $\tld{B}_{x}$
and $\tld{u}$ by this coordinate change. Note that up to performing
the coordinate transformation $y_{1}\mapsto\sigma y_{1}$, the one-variable
integral in (\ref{eq: double star}) is of the form $\gamma_{h,l}$
with $l=s$ and $h\left(x,y_{1}\right)=y_1^ru\left(x,y_{1}\right)\chi_{B_{x}}\left(y_{1}\right)$,
where $\chi_{B_{x}}$ is the characteristic function of the subanalytic
set $B_{x}$. This concludes the proof of Proposition \ref{prop:integGen-1}.
\end{proof}

\section{Preparation of functions in $\mathcal{C}^{\exp}$\label{s:prepExpConstr}}

Throughout this section $X$ denotes a subanalytic subset of $\RR^{m}$,
and we write $\left(x,y\right)$ for coordinates on $\RR^{m}\times\RR$. This
section states and proves our main preparation theorem for functions
in $\mathcal{C}^{\exp}$. The purpose of the preparation theorem is
to express a given $f\in\C^{\exp}\left(X\times\RR\right)$ as a finite sum of
generators for $\C^{\exp}\left(X\times\RR\right)$ that are either superintegrable
over $X$ or are ``naive in $y$'' (in the sense that the $\gamma$-functions
in these terms depend only on $x$ and not on $y$, see Definition
\ref{def: naive in y}).
\begin{defn}
\label{def: naive in y} Let $A\subseteq\RR^{m+1}$ be a subanalytic
set and $T\left(x,y\right)\in\C^{\exp}\left(A\right)$ be a generator. We say
that $T$ is \textbf{\emph{naive in $y$ }}if $T$ is of the form
\[
T\left(x,y\right)=f\left(x\right)y^{r}\left(\log y\right)^{s}\mathrm{e}^{\mathrm{i}\phi\left(x,y\right)},
\]
where $f\in\C^{\exp}\left(\Pi_{m}\left(A\right)\right)$, $r\in\QQ$, $s\in\NN$, and $\phi\in\mathcal{S}\left(A\right)$.

Notice that, if $T$ is naive in $y$, then the function $\gamma$
appearing in the expression (\ref{eq:generator}) for $T$ does not
depend on $y$.
\end{defn}
We use the notation from Definition \ref{def:center} in the following
theorem.
\begin{thm}
\label{thm:prepExpConstr}\textbf{\emph{ }}Let $X\subseteq\RR^{m}$
be a subanalytic set and $f\in\C^{\exp}\left(X\times\RR\right)$. Then there
exists a finite partition $\A$ of $X\times\RR$ into cells over $\RR^{m}$
such that for each $A\in\A$ that is open over $\RR^{m}$, there exists
a centre $\theta$ for $A$ for which we may express $f\circ P_{\theta}$
as a finite sum
\[
f\circ P_{\theta}\left(x,y\right)=\sum_{j\in J}T_{j}\left(x,y\right)
\]
on $A_{\theta}=\{\left(x,y\right):\ x\in\Pi_{m}\left(A\right),a\left(x\right)<y<b\left(x\right)\}$, where each
$T_{j}$ is a generator for $\C^{\exp}\left(A_{\theta}\right)$, such that:
\begin{enumerate}
\item if $b<+\infty$, then for each $j$, $T_{j}$ is superintegrable over
$\Pi_{m}\left(A\right)$;
\item if $b\equiv+\infty$, then there exists a positive integer $d$ and
a partition $J=J^{\mathrm{int}}\cup J^{\mathrm{naive}}$ such that:

\begin{enumerate}
\item for each $j\in J^{\mathrm{int}}$, $T_{j}$ is superintegrable over
$\Pi_{m}\left(A\right)$;
\item \label{enu:2b}for each $j\in J^{\mathrm{naive}}$, $T_{j}$ is naive
in $y$, is not superintegrable over $\Pi_{m}\left(A\right)$ and is of the form
\begin{equation}
T_{j}\left(x,y\right)=f_{j}\left(x\right)y^{r_{j}}\left(\log y\right)^{s_{j}}\mathrm{e}^{\mathrm{i}\phi_{j}\left(x,y\right)},\label{eq:naive}
\end{equation}
where $f_{j}\in\C^{\exp}\left(\Pi_{m}\left(A\right)\right)$, $r_{j}\in\QQ\cap[-1,+\infty)$,
$s_{j}\in\NN$, and $\phi_{j}$ is a polynomial in $y^{1/d}$ (for
some $d\in\mathbb{N}$) with coefficients in $\S\left(\Pi_{m}\left(A\right)\right)$ such
that $\phi_{j}\left(x,0\right)=0$ for all $x\in\Pi_{m}\left(A\right)$; moreover,
\[
\left(\left(r_{j},s_{j},\phi_{j}\left(x,y\right)\right)\right)_{j\in J^{\mathrm{naive}}}
\]
is a family of distinct tuples in $\QQ\times\NN\times\RR[y^{1/d}]$.
\end{enumerate}
\end{enumerate}
\end{thm}
\begin{rem}
\label{rem: superint tend to 0}Let us restrict our attention to a
cell of the form
\begin{equation}
A=\{\left(x,y\right):x\in\Pi_{m}\left(A\right),\ y>a\left(x\right)\}\label{eq: unbounded cell}
\end{equation}
(by Remark \ref{rmk:center}, we have $A=A_{\theta}$). The proof
of Theorem \ref{thm:prepExpConstr} will actually show that, for every
$j\in J$, there are $r_{j}\in\mathbb{Q},\ s_{j}\in\mathbb{N}$ and
a function $g_{j}\left(x,y\right)\in\mathcal{C}^{\exp}\left(X\times\mathbb{R}\right)$
which is bounded in $y$ (more precisely, there is a subanalytic function
$\eta:\Pi_{m}\left(A\right)\rightarrow[0,+\infty)$ such that $\forall y>a\left(x\right),\ |g_{j}\left(x,y\right)|<\eta\left(x\right)$),
such that
\begin{equation}
T_{j}\left(x,y\right)=y^{r_{j}}\left(\log y\right)^{s_{j}}g_{j}\left(x,y\right).\label{eq: general form of gen}
\end{equation}
Moreover, if $j\in J^{\text{Int}}$, then $r_{j}<-1$, and if $j\in J^{\text{naive}}$,
then, in the notation of Equation (\ref{eq:naive}), we have $g_{j}\left(x,y\right)=f_{j}\left(x\right)\mathrm{e}^{\mathrm{i}\phi_{j}\left(x,y\right)}$.
\end{rem}
The proof of the above theorem will be broken down into several propositions
and lemmas.
\begin{defn}
\label{notation for prep} Let $X\subseteq\RR^{m}$ be a subanalytic
set and $A\subseteq X\times\RR$ be a cell over $\RR^{m}$ which is
open over $\RR^{m}$. Let $\theta$ be a centre for $A$, so that
we can write
\[
A_{\theta}=\{\left(x,y\right):\ x\in\Pi_{m}\left(A\right),a\left(x\right)<y<b\left(x\right)\},
\]
for some analytic subanalytic functions $1\leq a\left(x\right)<b\left(x\right)$, where
we also allow the case when $b\equiv+\infty$ on $\Pi_{m}\left(A\right)$, as
in Definition \ref{def:center}.

Fix $d\in\mathbb{N}\setminus \{ 0\} $ and a bounded, analytic,
subanalytic map $\psi$ on $A_{\theta}$, of the form
\begin{equation}
\begin{aligned}\psi\left(x,y\right)=\left(c_{1}\left(x\right),\ldots,c_{N}\left(x\right),\left(\frac{a\left(x\right)}{y}\right)^{1/d},\left(\frac{y}{b\left(x\right)}\right)^{1/d}\right), & \text{\ if\ }b<+\infty,\\
\text{and}\ \psi\left(x,y\right)=\left(c_{1}\left(x\right),\ldots,c_{N}\left(x\right),\left(\frac{a\left(x\right)}{y}\right)^{1/d}\right), & \text{\ if\ }b\equiv+\infty.
\end{aligned}
\label{eq: psi}
\end{equation}
Let $J$ be an index set and, for all $j\in J$, let
\[
A_{j}=\{\left(x,y,t\right):\ \left(x,y\right)\in A_{\theta},a_{j}\left(x,y\right)<t<b_{j}\left(x,y\right)\},
\]
for some analytic, subanalytic functions $1\leq a_{j}<b_{j}$, where
we also allow the case when $b_{j}\equiv+\infty$ on $A_{\theta}$.

Suppose also that $a_{j}$, $b_{j}$ and $b_{j}-a_{j}$ are $\psi$-prepared
on $A_{\theta}$ as follows:
\begin{eqnarray*}
a_{j}\left(x,y\right) & = & a_{j,0}\left(x\right)y^{\alpha_{j}}u_{a_{j}}\left(x,y\right),\\
b_{j}\left(x,y\right) & = & b_{j,0}\left(x\right)y^{\beta_{j}}u_{b_{j}}\left(x,y\right),\\
b_{j}\left(x,y\right)-a_{j}\left(x,y\right) & = & c_{j,0}\left(x\right)y^{\delta_{j}}u_{c_{j}}\left(x,y\right)
\end{eqnarray*}
for some analytic, subanalytic $a_{j,0},b_{j,0},c_{j,0}$, some $\alpha_{j},\beta_{j},\delta_{j}\in\mathbb{Q}$
and some $\psi$-units $u_{a_{j}},u_{b_{j}},u_{c_{j}}$ (when $b_{j}=+\infty$
we stipulate that $b_{j,0}=c_{j,0}=+\infty$, $\beta_{j}=\delta_{j}=0$
and $u_{b_{j}}=u_{c_{j}}=1$).

In this situation, given $d_{j}\in\mathbb{N}\setminus\{ 0\} $, we define
the bounded, analytic, subanalytic map $\psi_{j}$ on $A_{j}$
as
\begin{equation}
\begin{aligned}\psi_{j}\left(x,y,t\right)=\left(\psi\left(x,y\right),\left(\frac{a_{j,0}\left(x\right)y^{\alpha_{j}}}{t}\right)^{1/d_{j}},\left(\frac{t}{b_{j,0}\left(x\right)y^{\beta_{j}}}\right)^{1/d_{j}}\right), & \text{\ if\ }b_{j}<+\infty,\\
\text{and}\ \psi_{j}\left(x,y,t\right)=\left(\psi\left(x,y\right),\left(\frac{a_{j,0}\left(x\right)y^{\alpha_{j}}}{t}\right)^{1/d_{j}}\right), & \text{\ if\ }b_{j}\equiv+\infty.
\end{aligned}
\label{eq: psi_j}
\end{equation}

\end{defn}
The next proposition establishes that, after writing $f$ as a sum
of generators and after preparing suitably all the subanalytic and
constructible functions appearing in the generators, we obtain a decomposition
of $X\times\mathbb{R}$ into cells over which each of the generators
has a well organized form. In particular, the generators are superintegrable
over every cell in the partition whose fibres over $\mathbb{R}^{m}$
are bounded (see Remark \ref{rem: bounded cells}(\ref{enu: supercontin})
below).
\begin{prop}
\label{lem: supercontinuous} Let $f\in\C^{\exp}\left(X\times\RR\right)$, for
some subanalytic set $X\subseteq\RR^{m}$. Then there exists a finite
partition $\A$ of $X\times\RR$ into cells over $\RR^{m}$ such that
for each $A\in\A$ that is open over $\RR^{m}$, there exists a centre
$\theta$ for $A$ for which we may express $f\circ P_{\theta}$ as
a finite sum
\begin{equation}
f\circ P_{\theta}\left(x,y\right)=\sum_{j\in J}T_{j}\left(x,y\right)\label{eq:prepExpConstr1}
\end{equation}
on $A_{\theta}=\{\left(x,y\right):\ x\in\Pi_{m}\left(A\right),a\left(x\right)<y<b\left(x\right)\}$, where each
$T_{j}$ is a generator for $\C^{\exp}\left(A_{\theta}\right)$ of the form
\begin{equation}
T_{j}\left(x,y\right)=f_{j}\left(x\right)y^{p_{j}}\left(\log y\right)^{q_{j}}\mathrm{e}^{\mathrm{i}\phi_{j}\left(x,y\right)}\gamma_{j}\left(x,y\right)\label{eq:Tform-1}
\end{equation}
for some $f_{j}\in\C\left(\Pi_{m}\left(A\right)\right)$, $p_{j}\in\QQ$, $q_{j}\in\NN$,
$\phi_{j}\in\S\left(A_{\theta}\right)$ and function $\gamma_{j}$, where
\begin{equation}
\gamma_{j}\left(x,y\right)=\int_{a_{j}\left(x,y\right)}^{b_{j}\left(x,y\right)}\Gamma_{j}\left(x,y,t\right)\,\mathrm{d}t\label{eq: form of gamma}
\end{equation}
with
\[
\Gamma_{j}\left(x,y,t\right)=t^{r_{j}}h_{j}\left(x,y,t\right)\left(\log t\right)^{s_{j}}\mathrm{e}^{\mathrm{i}\sigma_{j}t}
\]
for some $r_{j}\in\QQ$, $s_{j}\in\NN$, $\sigma_{j}\in\{-1,1\}$
and analytic subanalytic functions $a_{j},b_{j}$ as in Definition
\ref{notation for prep}, and for some $\psi_{j}$-function $h_{j}$
(where $\psi_{j}$ is as in Equation (\ref{eq: psi_j}), for some
$d,\psi,d_{j}$).

We may furthermore assume that the rational numbers $\alpha_{j}$,
$\beta_{j}$ and $\delta_{j}$ (see Definition \ref{notation for prep})
are integer multiples of $1/d$.\end{prop}
\begin{proof}
Write $f$ as a finite sum of generators for $\C^{\exp}\left(X\times\mathbb{R}\right)$,
say
\[
f\left(x,y\right)=\sum_{j\in J}T_{j}\left(x,y\right),
\]
where
\[
T_{j}\left(x,y\right)=g_{j}\left(x,y\right)\mathrm{e}^{\mathrm{i}\phi_{j}\left(x,y\right)}\gamma_{j}\left(x,y\right),
\]
with
\[
\gamma_{j}\left(x,y\right)=\int_{\RR}H_{j}\left(x,y,t\right)\left(\log|t|\right)^{\ell_{j}}\mathrm{e}^{\mathrm{i}t}\,\mathrm{d}t.
\]
Apply Proposition \ref{prop:prepSubConstr} (in the form in Remark
\ref{rem: prep subana with center}) to the collection
\begin{equation}
\{H_{j}\left(x,y,t\right)\left(\log|t|\right)^{\ell_{j}}\}_{j\in J}\subseteq\C\left(X\times\RR\times\mathbb{R}\right).\label{eq:hlogPrep}
\end{equation}
This gives a finite partition $\B$ of $\left(X\times\RR\right)\times\mathbb{R}$
into cells over $\RR^{m+1}$. By further partitioning in $\left(x,y\right)$,
we may assume that $\A:=\left\{\Pi_{m+1}\left(B\right):\ B\in\B\right\}$ is a partition
of $X\times\mathbb{R}$. By working piecewise, we may focus on one
$A\in\A$. There are finitely many disjoint cells $B\in\B$ such that
$\Pi_{m+1}\left(B\right)=A$. Pick one such $B$ which is open over $\RR^{m+1}$.
Write
\[
B=\{\left(x,y,t\right):\ \left(x,y\right)\in A,\widetilde{a}\left(x,y\right)<t-\Theta\left(x,y\right)<\widetilde{b}\left(x,y\right)\},
\]
where $\Theta$ is the centre given by the preparation of the collection
in Equation (\ref{eq:hlogPrep}). We fix an element of this collection
and we focus on one summand of the preparation of such an element. This
will have the form
\[
f_{0}\left(x,y\right)|t-\Theta\left(x,y\right)|^{r}\left(\log|y-\Theta\left(x,y\right)|\right)^{s}h\left(x,y,t\right),
\]
where $f_{0}\in\mathcal{C}\left(A\right)$ and $h$ is a $\Psi$-function
(for a suitable bounded subanalytic $\Psi$).

We write $\mathrm{e}^{\mathrm{i}t}=\mathrm{e}^{\mathrm{i}\left(t-\Theta\left(x,y\right)\right)}\mathrm{e}^{\mathrm{i}\Theta\left(x,y\right)}$.
By factoring out of the integral the term $f_{0}\left(x,y\right)\mathrm{e}^{\mathrm{i}\Theta\left(x,y\right)}$,
and by absorbing $f_{0}$ in the constructible coefficient $g$ and
$\mathrm{e}^{\mathrm{i}\Theta\left(x,y\right)}$ in the exponential term
$\mathrm{e}^{\mathrm{i}\phi\left(x,y\right)}$, we can reduce to studying
generators of the form
\begin{equation}
g\left(x,y\right)\mathrm{e}^{\mathrm{i}\phi\left(x,y\right)}\int_{\widetilde{a}\left(x,y\right)}^{\widetilde{b}\left(x,y\right)}\left|t-\Theta\left(x,y\right)\right|^{r}h\left(x,y,t\right)\left(\log|t-\Theta\left(x,y\right)|\right)^{s}\mathrm{e}^{\mathrm{i}\left(t-\Theta\left(x,y\right)\right)}\,\mathrm{d}t\label{eq:prepGamma}
\end{equation}
on $A$. Now, the set
\begin{equation}
\{t-\Theta\left(x,y\right):\ \left(x,y,t\right)\in B\}\label{eq:t-theta(x)}
\end{equation}
is contained in one of the sets $\left(-\infty,-1\right)$, $\left(-1,0\right)$, $\left(0,1\right)$
or $\left(1,+\infty\right)$.

Suppose first that (\ref{eq:t-theta(x)}) is contained in either $\left(-1,0\right)$
or $\left(0,1\right)$. Then $\mathrm{e}^{\mathrm{i}\left(t-\Theta\left(x,y\right)\right)}$ is a complex-valued
subanalytic function on $A$ (see Definition \ref{def complex valued}),
so the integral in (\ref{eq:prepGamma}) is a complex-valued constructible
function on $A$. This implies that (\ref{eq:prepGamma}) is in $\C_{\naive}^{\exp}\left(A\right)$
(because $\mathcal{C}$ is stable under integration), hence we can
apply Proposition \ref{prop:prepSubConstr} to the constructible part
of (\ref{eq:prepGamma}), preparing it with respect to the variable
$y$. Now, we can view the $\psi$-function obtained in this preparation
as a $\gamma$-function of the form (\ref{eq: form of gamma}) (see
Remark \ref{rem: gamma subanalytic}), and we are done.

Now suppose that (\ref{eq:t-theta(x)}) is contained in $\left(-\infty,-1\right)$
or $\left(1,+\infty\right)$. Then by applying the change of coordinates $t\mapsto\sigma t+\Theta\left(x,y\right)$
for an appropriate choice of $\sigma\in\{-1,1\}$, and adjusting the
definitions of $\widetilde{a}$, $\widetilde{b}$ and $h$ accordingly, we
may assume that $1\leq\widetilde{a}\left(x,y\right)<\widetilde{b}\left(x,y\right)$ and that (\ref{eq:prepGamma})
is of the form
\[
g\left(x,y\right)\mathrm{e}^{\mathrm{i}\phi\left(x,y\right)}\int_{\widetilde{a}\left(x,y\right)}^{\widetilde{b}\left(x,y\right)}t^{r}h\left(x,y,t\right)\left(\log t\right)^{s}\mathrm{e}^{\mathrm{i}\sigma t}\,\mathrm{d}t.
\]

Summing up, we have constructed a finite partition $\A$ of $X\times\RR$ into subanalytic sets such that for each $A\in\A$ we may write $f$ as a finite sum
\begin{equation}\label{eq:f:Halfway}
f(x,y) = \sum_{j\in J} g_j(x,y) \textrm{e}^{\textrm{i} \phi_j(x,y)} \gamma_j(x,y)
\end{equation}
on $A$, where $g_j\in\C(A)$, $\phi_j\in\S(A)$, and
\begin{equation}\label{eq:gamma:Halfway}
\gamma_j(x,y) = \int_{\tld{a}_j(x,y)}^{\tld{b}_j(x,y)} \tld{\Gamma}_j(x,y,t)\,\textrm{d}t
\end{equation}
with
\begin{equation}\label{eq:Gamma:Halfway}
\tld{\Gamma}_j(x,y,t) = t^{r_j} \tld{h}_j(x,y,t) (\log t)^{s_j} \textrm{e}^{\textrm{i}\sigma_j t},
\end{equation}
where $1\leq \tld{a}_j < \tld{b}_j$ (with either $\tld{b}_j < +\infty$ or $\tld{b}_j \equiv +\infty$), $r_j\in\QQ$, $s_j\in\NN$, $\sigma_j\in\{-1,1\}$, and $\tld{h}_j$ is a $\tld{\psi}_j$-function, with
\begin{align*}
&&
\tld{\psi}_j(x,y,t) = \left(\tld{c}_{j,1}(x,y),\ldots,\tld{c}_{j,N_j}(x,y), \left(\frac{\tld{a}_j(x,y)}{t}\right)^{1/d_j}, \left(\frac{t}{\tld{b}_j(x,y)}\right)^{1/d_j}\right),
\quad\text{if $\tld{b}_j < \infty$,}
\\
&&
\text{and}\,\,\,
\tld{\psi}_j(x,y,t) = \left(\tld{c}_{j,1}(x,y),\ldots,\tld{c}_{j,N_j}(x,y), \left(\frac{\tld{a}_j(x,y)}{t}\right)^{1/d_j}\right),
\quad\text{if $\tld{b}_j \equiv +\infty$,}
\end{align*}
defined on
\[
\tld{A}_j = \{(x,y,t) : (x,y)\in A, \tld{a}_j(x,y) < t < \tld{b}_j(x,y)\}.
\]
We may additionally assume that the positive integer $d_j$ has been chosen so that $r_j$ is an integer multiple of $1/d_j$.

%
%

In order to have a more uniform notation, we will assume that $\tld{\psi}_j$ maps into $\RR^{N_j+2}$ for each $j\in J$.  (This is the case when $\tld{b}_j < +\infty$, and the argument adapts to the case that $\tld{b}_j \equiv +\infty$ by simply ignoring the last component of $\tld{\psi}_j$ involving $\left(\frac{t}{\tld{b}_j(x,y)}\right)^{1/d_j}$.)  For each $j\in J$, fix $p_j = (p_{j,1},\ldots,p_{j,N_j+2})$ and $\eta_j = (\eta_{j,1},\ldots,\eta_{j,N_j+2})$ in $(0,\infty)^{N_j+2}$ and also a real analytic function $\tld{H}_j$ on $B_p(\RR)$ such that $\tld{\psi}_j(\tld{A}_j)\subseteq B_p(\RR)$, $\tld{h}_j = \tld{H}_j\circ\tld{\psi}_j$, and $\tld{H}_j$ extends to a complex analytic function on a neighbourhood of $B_{p+\eta}(\CC)$.  We may assume that $p_{j,N_j+1} = p_{j,N_j+2} = 1$.  Fix $\varepsilon \in (0,1)$ sufficiently small so that for all $j\in J$ and $k\in\{1,\ldots,N_j+2\}$,
\begin{equation}\label{eq:epsilon:conditions}
\begin{cases}
\frac{1+\varepsilon}{1-\varepsilon}p_{j,k} < p_{j,k}+\eta_{j,k}
    & \text{if $k\in\{1,\ldots,N_j\}$,} \\
\left(\frac{1+\varepsilon}{1-\varepsilon}\right)^{1/d_j} < 1 +\eta_{j,k}
    & \text{if $k = N_j+1$ or $k = N_j+2$.}
\end{cases}
\end{equation}
For each set $A\in\A$, apply Proposition \ref{prop:prepSubConstr} (with respect to the variable $y$) to
\[
\{g_j\}_{j\in J} \subseteq \C(A)
\quad\text{and}\quad
\{\tld{c}_{j,1},\ldots,\tld{c}_{j,N_j}, \tld{a}_j, \tld{b}_j, \tld{b}_j-\tld{a}_j\}_{j\in J} \subseteq \S(A)
\]
so that the units occurring in the preparation of $\{\tld{c}_{j,1},\ldots,\tld{c}_{j,N_j}, \tld{a}_j, \tld{b}_j, \tld{b}_j-\tld{a}_j\}_{j\in J}$ are within $\varepsilon$ of $1$ (see Remark \ref{rem: prep of subana}), and then redefine $\A$ to be the finer partition of $X\times\RR$ into cells over $\RR^m$ thus created.

Focus on one cell $A\in\A$ that is open over $\RR^m$, and let $\theta$ be the center of $A$ given by the preparation.  We now use the notation set up in Definition \ref{notation for prep}, where $a_j = \tld{a}_j\circ P_\theta$ and $b_j = \tld{b}_j\circ P_\theta$, and where the positive integer $d$ in Definition \ref{notation for prep} has been chosen to be a common denominator of the set of rational exponents $\{\alpha_j, \beta_j, \delta_j : j\in J\}$ and also of the rational exponents of $y$ in the $\psi$-prepared forms of each of the functions $c_{j,k} := \tld{c}_{j,k}\circ P_\theta$ with $j\in J$ and $k\in\{1,\ldots,N_j\}$.  Since each constructible function $g_j\circ P_\theta$ is prepared on $A_\theta$, it is apparent from equations  \eqref{eq:f:Halfway}, \eqref{eq:gamma:Halfway}, and \eqref{eq:Gamma:Halfway} that $f\circ P_\theta$ is of the form  asserted in the conclusion of the proposition except for one detail: although each function $h_j(x,y,t) := \tld{h}_j(P_\theta(x,y),t)$ is clearly a $\tld{\psi}_j(P_\theta(x,y),t)$-function, the conclusion of the proposition asserts that $h_j$ is a $\psi_j$-function for the map $\psi_j$ defined in Definition \ref{notation for prep}.  To finish the proof, we will show that $h_j$ is a $\psi_j$-function after $\psi_j$ is modified by extending its list of component functions $c_1(x),\ldots,c_N(x)$ in $x$ alone by some additional functions in $x$ obtained from the $\psi$-prepared forms of the functions in $\{c_{j,k}\}_{j,k}$.

In order to have a more uniform notation when showing this, we shall assume that $\psi$ maps into $\RR^{N+2}$ (as would be the case when $b < +\infty$).  For each $j\in J$, define $K_{j}^{+}$ to be the set of all $k\in\{1,\ldots,N_j\}$ such that the exponent of $y$ in the $\psi$-prepared form of $c_{j,k}$ is greater than or equal to $0$, and define $K_{j}^{-} = \{1,\ldots,N_j\}\setminus K_{j}^{+}$.  For each $j\in J$ and $k\in\{1,\ldots,N_j\}$, we may write
\[
c_{j,k}(x,y) =
\begin{cases}
c_{j,k,0}(x)\left(\frac{y}{b(x)}\right)^{\nu_{j,k}/d} v_{j,k}(x,y),
    & \text{if $k\in K_{j}^{+}$,} \\
c_{j,k,0}(x)\left(\frac{a(x)}{y}\right)^{\nu_{j,k}/d} v_{j,k}(x,y),
    & \text{if $k\in K_{j}^{-}$,} \\
\end{cases}
\]
for some $c_{j,k,0}\in\S(\Pi_m(A))$, $\nu_{j,k}\in\NN$, and $\psi$-unit $v_{j,k}$.  Fix $q = (q_1,\ldots,q_{N+2})$ in $(0,\infty)^{N+2}$ such that $\psi(A_\theta) \subseteq B_q(\RR)$ and such that for all $j\in J$ and $k\in\{1,\ldots,N_j\}$ we have $u_{a_j} = U_{a_j}\circ\psi$, $u_{b_j} = U_{b_j}\circ\psi$, and $v_{j,k} = V_{j,k}\circ\psi$ for some real analytic functions $U_{a_j}$, $U_{b_j}$, and $V_{j,k}$ on $B_q(\RR)$ which extend to complex analytic functions on a neighbourhood of $B_q(\CC)$ such that for each $U\in\{U_{a_j},U_{b_j},V_{j,k}\}_{j,k}$,
\[
|U(z) - 1| < \varepsilon \quad\text{for all $z\in B_q(\CC)$}.
\]
We may assume that $q_{N+1} = q_{N+2} = 1$.

Focus on one choice of $j\in J$.  Writing out the equation $h_j(x,y,t) = \tld{H}_j\circ\tld{\psi}_j(P_\theta(x,y),t)$ in full detail with the $\psi$-prepared forms of its components gives
\begin{align}\label{eq:hj:LongForm}
h_j(x,y,t)
    &= \tld{H}_j\left(
    \left( c_{j,k,0}(x)\left(\frac{a(x)}{y}\right)^{\nu_{j,k}/d} V_{j,k}\circ\psi(x,y)\right)_{k\in K_{j}^{-}},
    \right.
    \\
    &\qquad\qquad
    \left( c_{j,k,0}(x)\left(\frac{y}{b(x)}\right)^{\nu_{j,k}/d} V_{j,k}\circ\psi(x,y)\right)_{k\in K_{j}^{+}},
    \nonumber\\
    &\qquad\qquad
    \left(\frac{a_{j,0}(x) y^{\alpha_j}}{t}\right)^{1/d_j} \left(U_{a_j}\circ\psi(x,y)\right)^{1/d_j},
    \nonumber\\
    &\qquad\qquad
    \left.
    \left(\frac{t}{b_{j,0}(x) y^{\beta_j}}\right)^{1/d_j} \left(U_{b_j}\circ\psi(x,y)\right)^{-1/d_j}
    \right).
    \nonumber
\end{align}
Consider $k\in\{1,\ldots,N_j\}$, and observe that on $A_\theta$ we have that $|c_{j,k}(x,y)|\leq p_{j,k}$, that $|v_{j,k}(x,y)|\geq 1-\varepsilon$, and that $\frac{a(x)}{y}$ and $\frac{y}{b(x)}$ can take values arbitrarily close to $1$ (for each fixed $x\in\Pi_m(A)$).  It follows that
\begin{equation}\label{eq:cjk:Bnd}
|c_{j,k,0}(x)| \leq \frac{p_{j,k}}{1-\varepsilon}
\end{equation}
on $A_\theta$.  Similar reasoning shows that
\begin{equation}\label{eq:t:Bnd}
\left|\frac{a_{j,0}(x)y^{\alpha_j}}{t}\right|^{1/d_j} \leq \left(\frac{1}{1-\varepsilon}\right)^{1/d_j}
\quad\text{and}\quad
\left|\frac{t}{b_{j,0}(x)y^{\beta_j}}\right|^{1/d_j} \leq \left(1+\varepsilon\right)^{1/d_j}
\end{equation}
hold on $A_j$ as well.  Clearly
\begin{equation}\label{eq:y:Bnd}
\left|\frac{a(x)}{y}\right|^{1/d} \leq 1
\quad\text{and}\quad
\left|\frac{y}{b(x)}\right|^{1/d} \leq 1
\end{equation}
on $A_\theta$, and also for all $k\in\{1,\ldots,N_j\}$ we have
\begin{equation}\label{eq:units:Bnd}
|V_{j,k}| \leq 1+\varepsilon,
\quad
|U_{a_j}|^{1/d_j} \leq (1+\varepsilon)^{1/d_j},
\quad\text{and}\quad
|U_{b_j}|^{-1/d_j} \leq \left(\frac{1}{1-\varepsilon}\right)^{1/d_j}
\end{equation}
on $B_q(\CC)$.  Using the variables $(W,X,Y,Z) = ((W_k)_{k=1}^{N}, (X_{j,k})_{k=1}^{N_j}, Y_1,Y_2,Z_1,Z_2)$,
define
\begin{align*}
H_j(W,X,Y,Z)
    &:=
    \tld{H}_j\left(
    \left(X_kY_{1}^{\nu_{j,k}}V_{j,k}(W)\right)_{k\in K_{j}^{-}}, \left(X_kY_{2}^{\nu_{j,k}}V_{j,k}(W)\right)_{k\in K_{j}^{+}},
    \right.
    \\
    &\qquad\qquad \left.
    Z_1\left(U_{a_j}(W)\right)^{1/d_j}, Z_2\left(U_{b_j}(W)\right)^{-1/d_j}
    \right).
\end{align*}
Define
\[
\rho = \left(q_1,\ldots,q_N,\frac{p_1}{1-\varepsilon},\ldots,\frac{p_{N_j}}{1-\varepsilon}, 1, 1, 1, 1\right).
\]
Observe from the inequalities \eqref{eq:cjk:Bnd}-\eqref{eq:units:Bnd}, from the conditions \eqref{eq:epsilon:conditions} imposed upon our choice of $\varepsilon$, and from \eqref{eq:hj:LongForm} that the range of the map on $A_j$ given by \begin{align}\label{eq:innerMap}
(x,y,t)
    &\mapsto
    \left(
    \left(c_k(x)\right)_{k=1}^{N}, \left(c_{j,k}(x)\right)_{k=1}^{N_j}, \left(\frac{a(x)}{y}\right)^{1/d}, \left(\frac{y}{b(x)}\right)^{1/d},
    \right.
    \\
    &\phantom{mapsto}\quad\left.
    \left(\frac{a_{j,0}(x)y^{\alpha_j}}{t}\right)^{1/d_j}, \left(\frac{t}{b_{j,0}(x)y^{\beta_j}}\right)^{1/d_j}\right)
    \nonumber
\end{align}
is contained in $B_\rho(\RR)$, that $H_j$ is defined as a complex analytic function on a neighbourhood of $B_\rho(\CC)$, and that $h_j$ is the composition of $H_j$ with the map \eqref{eq:innerMap}.  This completes the proof. \end{proof}
%
%

%
%
%
%
%
%
\begin{defn}
\label{def prepared generator}We call a generator for $\C^{\exp}\left(A_{\theta}\right)$
of the form (\ref{eq:Tform-1}) a \textbf{\emph{prepared generator}}.\end{defn}
\begin{rems}
\label{rem: bounded cells}Fix a prepared generator $T_{j}$ as in
Proposition \ref{lem: supercontinuous}.
\begin{enumerate}
\item \label{enu:r_j}If $b_{j}<+\infty$, then we may suppose $r_{j}=0$.
If $b_{j}\equiv+\infty$, then we may suppose that $r_{j}<-1$. \\
To see this, suppose first that $b_{j}<+\infty$. If $r_j\geq0$, then write
\begin{eqnarray*}
	g_{j}\left(x,y\right)t^{r_{j}}h_{j}\left(x,y,t\right) & = & \left(g_{j}\left(x,y\right)\left(b_{j,0}\left(x\right)y^{\beta_{j}}\right)^{r_{j}}\right)\left(\left(\frac{t}{b_{j,0}\left(x\right)y^{\beta_{j}}}\right)^{r_{j}}h_{j}\left(x,y,t\right)\right)\\
	& = & \widetilde{g_{j}}\left(x,y\right)\widetilde{h}_{j}\left(x,y,t\right).
\end{eqnarray*}
If $r_j<0$, then write
\begin{eqnarray*}
	g_{j}\left(x,y\right)t^{r_{j}}h_{j}\left(x,y,t\right) & = & \left(g_{j}\left(x,y\right)\left(a_{j,0}\left(x\right)y^{\alpha_{j}}\right)^{r_{j}}\right)\left(\left(\frac{a_{j,0}\left(x\right)y^{\alpha_{j}}}{t}\right)^{-r_{j}}h_{j}\left(x,y,t\right)\right)\\
	& = & \widetilde{g_{j}}\left(x,y\right)\widetilde{h}_{j}\left(x,y,t\right).
\end{eqnarray*}
Note that, in both cases, $\widetilde{h}_{j}$ is a $\psi_{j}$-function (but not necessarily a
$\psi_{j}$-unit),
because $r_{j}$ is an integral multiple of $1/d_{j}$. We have hence reduced to the case $r_j=0$.

Suppose now that $b_{j}\equiv +\infty$. Let $n_0$  be the smallest exponent appearing in the series expansion of $h_{j}\left(x,y,t\right)$
with respect to the variable $\left(\frac{a_{j,0}\left(x\right)y^{\alpha_{j}}}{t}\right)^{1/d_{j}}$. Then we can factor out the power $\left(\frac{a_{j,0}\left(x\right)y^{\alpha_{j}}}{t}\right)^{n_0/d_{j}}$ from the expansion of $h_j$ and write
\begin{eqnarray*}
	g_{j}\left(x,y\right)t^{r_{j}}h_{j}\left(x,y,t\right) & = & \left(g_{j}\left(x,y\right)\left(a_{j,0}\left(x\right)y^{\alpha_{j}}\right)^{n_0/d_j}\right)t^{r_j-n_0/d_j}\widetilde{h}_{j}\left(x,y,t\right),\\
\end{eqnarray*}
where $\widetilde{h}_j$ is a $\psi_j$-unit. Note that $\widetilde{r}_j:=r_j-n_0/d_j$ is necessarily strictly
smaller than $-1$ (otherwise $\gamma_{j}$ would not be defined).
\item \label{enu: supercontin}Whenever $b<+\infty$, $T_{j}$ is superintegrable
over $\Pi_{m}\left(A\right)$. This is clear, since for all $x\in\Pi_{m}\left(A\right)$,
$y\mapsto T_{j}^{\abs}\left(x,y\right)$ extends to a continuous function on
the closure of $A_{x}$ in $\RR$.
\end{enumerate}
\end{rems}
\begin{rem}
\label{rem:piecewise analytic}
For all $m\in\NN$, subanalytic $X\subseteq\RR^m$, and $g\in\C^{\exp}(X)$, there exists a finite partition $\A$ of $X$ into subanalytic cells (see Definition \ref{def:cell}) such that $g\restriction{A}$ is analytic for each open set $A\in\A$.
\end{rem}
\begin{proof}
Apply Proposition \ref{lem: supercontinuous} to $f$ (except we now omit the variable $y$ since we are working on $X$ rather than on $X\times\RR$), and let $\A$ be the partition of $X$ so obtained.  Consider an open set $A\in\A$.  For each $j\in J$, $\gamma_j\restriction{A}$ is analytic because it is the integral of an analytic function with analytic limits of integration.  (Namely, basic facts about power series show that the antiderivative in $t$ of integrand $\Gamma_j(x,t)$ is analytic, and evaluating this antiderivative at analytic limits of integration in $x$ gives in an analytic function in $x$.)  It therefore follows from equations (\ref{eq:prepExpConstr1}) and (\ref{eq:Tform-1}) that $f\restriction{A}$ is analytic.
\end{proof}
%
%
In view of Remark \ref{rem: bounded cells}(\ref{enu: supercontin}),
we can focus our attention on cells which are unbounded above. For
such cells, our next goal is to reduce to the case where each of the
generators in Equation (\ref{eq:prepExpConstr1}) is either superintegrable,
or in $\C_{\naive}^{\exp}\left(A_{\theta}\right)$, or is such that the variable
$y$ does not appear in the integration limits $a_{j}$ and $b_{j}$
of the $\gamma$-function.
\begin{prop}
\label{lem no y in the bounds}Proposition \ref{lem: supercontinuous}
holds with the additional property that, whenever $b\equiv+\infty$,
every $T_{j}$ is either superintegrable, or in $\C_{\naive}^{\exp}\left(A_{\theta}\right)$,
 or there exist analytic and subanalytic functions $a_{j,0}, b_{j,0}$  on $\Pi_m\left(A\right)$ such that
 $a_{j}\left(x,y\right)=a_{j,0}\left(x\right)$
and either $b_{j}\equiv+\infty$ or $b_{j}\left(x,y\right)=b_{j,0}\left(x\right)$.
\end{prop}
In order to prove the above proposition, we first need to establish
two technical lemmas (Lemmas \ref{lem alpha =00003D0} and \ref{lem: b_j finite, alpha >0}
below). Their aim is to reduce to the case of prepared generators
such that the variable $y$ only appears in the units in the prepared
form of the integration limits of the $\gamma$-function, that is $\alpha_j=\beta_j=0$. 
To achieve
this, our main tool will be to compute $\gamma$ by integrating by
parts. This will lead to rewriting the prepared generator as a finite
sum of generators which are either superintegrable, or in $\C_{\naive}^{\exp}\left(A_{\theta}\right)$,
or are in a \textquotedblleft{}better form\textquotedblright{} (for
example, the variable $y$ now only appears in one of the two integration
limits). We will also have to refine the partition into cells along
the way. This is harmless when we refine the partition with respect
to the variables $x$. When further partitioning with respect to the
variable $y$, we will possibly create new bounded cells, which can
be handled as in Remark \ref{rem: bounded cells}(\ref{enu: supercontin}).
\begin{defn}
\label{def: splitting}In the notation of Definition \ref{notation for prep},
suppose that $b_{j}<+\infty$. 
We let $\psi_{j,-}$ and $\psi_{j,+}$ be the maps
obtained from $\psi_{j}$ by omitting the last and the second-to-last component
of  $\psi_{j}$, respectively.
%
%
We extend this definition to the case $b_{j}\equiv+\infty$ by stipulating
that $\psi_{j,-}=\psi_{j}$ and $\psi_{j,+}=0$.\end{defn}
\begin{rem}
Notice that when $b_{j}\equiv+\infty$, $\alpha_{j}\text{\ and\ }\beta_{j}$
are necessarily non-negative, since $a_{j},b_{j}\geq1$.\end{rem}
\begin{lem}
[Splitting]\label{lem alpha =00003D0}
Let $f\in \mathcal{C}^{\exp}\left(X\right) $ for some subanalytic set $X\subseteq \RR^m$
 and let $A \in {\mathcal A}$ be one of the cells obtained from Proposition 
\ref{lem: supercontinuous}
satisfying $b =+\infty$. Let $T_j$ be one of the generators corresponding to
this $A$ satisfying $b_j < +\infty$, $\alpha_j=0$ and $\beta_j>0$
then we may write $T_{j}$ as a finite sum $\sum T_{k}$ of prepared
generators, where each $T_{k}$ is either superintegrable, or in $\C_{\naive}^{\exp}\left(A_{\theta}\right)$,
or is such that $b_{k}\equiv+\infty$ (and hence $h_{k}$ is a $\psi_{k,-}$-function).\end{lem}
\begin{proof}
We consider a generator $T_{j}$ as in the statement of the lemma.
In the notation of Proposition \ref{lem: supercontinuous}, let $n_{j}\in\NN$
be such that
\begin{equation}
p_{j}-n_{j}\beta_{j}+\delta_{j}<-1.\label{eq:nj}
\end{equation}
 Our next aim is to write $h_{j}$ as a sum of three terms, depending
on the choice of $n_{j}$, as follows:
\begin{equation}
h_{j}\left(x,y,t\right)=\left(\frac{t}{a_{j,0}\left(x\right)}\right)^{r_{j,-}}h_{j,-}\left(x,y,t\right)+h_{j,0}\left(x,y,t\right)+\left(\frac{t}{b_{j,0}\left(x\right)y^{\beta_{j}}}\right)^{n_{j}}h_{j,+}\left(x,y,t\right),\label{eq:h}
\end{equation}
where $r_{j,-}<-1$ is rational, $h_{j,-}$ is a $\psi_{j,-}$-function,
$h_{j,+}$ is a $\psi_{j,+}$-function, and $h_{j,0}$ is a finite
sum of terms of the form $g\left(x,y\right)z^{k}$, where $g$ is
a $\psi$-function, $k\in\mathbb{N}$ and $z$ is either $\left(\frac{t}{b_{j}}\right)^{1/d_{j}}$ or  $\left(\frac{a_{j}}{t}\right)^{1/d_{j}}$.

In order to do this, we expand $h_{j}$ as a series in the variables
$\left(\frac{t}{b_{j}}\right)^{1/d_{j}},\left(\frac{a_{j}}{t}\right)^{1/d_{j}}$,
with $\psi$-functions as coefficients. Now, remembering that $b\equiv+\infty$
and $\alpha_{j}=0$, for each $k,l\in\mathbb{N}$ write
\[
\left(\frac{t}{b_{j,0}\left(x\right)y^{\beta_{j}}}\right)^{k/d_{j}}\left(\frac{a_{j,0}\left(x\right)}{t}\right)^{l/d_{j}}=\begin{cases}
{ \left(\frac{a_{j,0}\left(x\right)}{b_{j,0}\left(x\right)y^{\beta_{j}}}\right)^{l/d_{j}}\left(\frac{t}{b_{j,0}\left(x\right)y^{\beta_{j}}}\right)^{\left(k-l\right)/d_{j}},} & \text{if \ensuremath{k\geq l},}\vspace*{10pt}\\
{ \left(\frac{a_{j,0}\left(x\right)}{b_{j,0}\left(x\right)y^{\beta_{j}}}\right)^{k/d_{j}}\left(\frac{a_{j,0}\left(x\right)}{t}\right)^{\left(l-k\right)/d_{j}},} & \text{if \ensuremath{k<l},}
\end{cases}
\]
and
\begin{equation}
\frac{a_{j,0}\left(x\right)}{b_{j,0}\left(x\right)y^{\beta_{j}}}=\left[\frac{a_{j,0}\left(x\right)}{b_{j,0}\left(x\right)a\left(x\right)^{\beta_{j}}}\right]\left(\frac{a\left(x\right)}{y}\right)^{\beta_{j}}.\label{eq:presplit}
\end{equation}
The quotient on left side of (\ref{eq:presplit}) is bounded, since
it is equal to the bounded quotient $\frac{a_{j}}{b_{j}}$ multiplied
by a unit. Moreover, for each $x$ we may take $y$ to be arbitrarily
close to $a\left(x\right)$, thereby making $\frac{a\left(x\right)}{y}$ arbitrarily close
to $1$. It follows that the function in square brackets on the right
side of (\ref{eq:presplit}), which does not depend on $y$, is also
bounded and therefore can be included in the list of functions $c_{1}\left(x\right),\ldots,c_{N}\left(x\right)$
in $\psi$.

Therefore we can write $h_{j}$ as the sum of a $\psi_{j,-}$-function
of the form $\sum_{k\geq1}g_{j,k}\left(x,y\right)\left(\frac{a_{j,0}\left(x\right)}{t}\right)^{k/d_{j}}$
plus a $\psi_{j,+}$-function of the form $\sum_{k\geq0}\widetilde{g}_{j,k}\left(x,y\right)\left(\frac{t}{b_{j,0}\left(x\right)y^{\beta_{j}}}\right)^{k/d_{j}}$,
where $g_{j,k},\widetilde{g}_{j,k}$ are $\psi$-functions. If we set
\[
h_{j,-}=\sum_{k\geq d_{j}+1}g_{j,k}\left(x,y\right)\left(\frac{a_{j,0}\left(x\right)}{t}\right)^{\frac{k}{d_{j}}-1-\frac{1}{d_{j}}},\ h_{j,+}=\sum_{k\geq d_{j}n_{j}}\widetilde{g}_{j,k}\left(x,y\right)\left(\frac{t}{b_{j,0}\left(x\right)y^{\beta_{j}}}\right)^{\frac{k}{d_{j}}-n_{j}},
\]
we obtain Equation (\ref{eq:h}) with $r_{j,-}=-1-\frac{1}{d_{j}}$
and
\[
h_{j,0}=\sum_{k=1}^{d_{j}}g_{j,k}\left(x,y\right)\left(\frac{a_{j,0}\left(x\right)}{t}\right)^{k/d_{j}}+\sum_{k=0}^{d_{j}n_{j}-1}\widetilde{g}_{j,k}\left(x,y\right)\left(\frac{t}{b_{j,0}\left(x\right)y^{\beta_{j}}}\right)^{k/d_{j}}.
\]
Hence we can write
\begin{eqnarray*}
\Gamma_{j,-}\left(x,y,t\right) & = & \left(\frac{t}{a_{j,0}\left(x\right)}\right)^{r_{j,-}}h_{j,-}\left(x,y,t\right)\left(\log t\right)^{s_{j}}\mathrm{e}^{\mathrm{i}\sigma_{j}t},\\
\Gamma_{j,0}\left(x,y,t\right) & = & h_{j,0}\left(x,y,t\right)\left(\log t\right)^{s_{j}}\mathrm{e}^{\mathrm{i}\sigma_{j}t},\\
\Gamma_{j,+}\left(x,y,t\right) & = & \left(\frac{t}{b_{j,0}\left(x\right)y^{\beta_{j}}}\right)^{n_{j}}h_{j,+}\left(x,y,t\right)\left(\log t\right)^{s_{j}}\mathrm{e}^{\mathrm{i}\sigma_{j}t}
\end{eqnarray*}
and
\[
T_{j}\left(x,y\right)=T_{j,-}\left(x,y\right)+T_{j,0}\left(x,y\right)+T_{j,+}\left(x,y\right),
\]
where one obtains $T_{j,-}$, $T_{j,0}$ and $T_{j,+}$ from $T_{j}$
by replacing $\Gamma_{j}$ with $\Gamma_{j,-}$, $\Gamma_{j,0}$ and
$\Gamma_{j,+}$, respectively.

To handle $T_{j,-}$, note that, since $r_{j,-}<-1$, we can use the
additivity relation
\[
\int_{a_{j}\left(x,y\right)}^{b_{j}\left(x,y\right)}\Gamma_{j,-}\left(x,y,t\right)\,\mathrm{d}t=\int_{a_{j}\left(x,y\right)}^{+\infty}\Gamma_{j,-}\left(x,y,t\right)\,\mathrm{d}t-\int_{b_{j}\left(x,y\right)}^{+\infty}\Gamma_{j,-}\left(x,y,t\right)\,\mathrm{d}t.
\]
Therefore we can replace $T_{j,-}$ by a sum of two prepared generators
for $\C^{\exp}\left(A_{\theta}\right)$ whose $\gamma$-functions are defined
by integrals with $+\infty$ as the upper limit of integration.

To handle $T_{j,0}$, compute $\int_{a_{j}\left(x,y\right)}^{b_{j}\left(x,y\right)}\Gamma_{j,0}\left(x,y,t\right)\,\mathrm{d}t$
by integrating by parts, where one differentiates $h_{j,0}\left(x,y,t\right)\left(\log t\right)^{s_{j}}$
and integrates $\mathrm{e}^{\mathrm{i}\sigma_{j}t}$. This has the effect
of replacing $T_{j,0}\left(x,y\right)$ with a sum of terms that are either prepared
generators for $\C_{\naive}^{\exp}\left(A_{\theta}\right)$ or are of the same
form as $T_{j,0}$ but with the powers of $t$ in $h_{j,0}$ reduced
by $1$. By repeating this strategy finitely many times, we reduce
to the case that all powers of $t$ in $h_{j,0}$ are less than $-1$,
which can then be handled as we did for $T_{j,-}$.

It remains to handle $T_{j,+}$. Recall that
\begin{equation}
\left(\frac{t}{b_{j,0}\left(x\right)y^{\beta_{j}}}\right)^{n_{j}}h_{j,+}\left(x,y,t\right)\left(\log t\right)^{s_{j}}=\left(\sum_{k=n_{j}d_{j}}^{+\infty}\widetilde{g}_{j,k}\left(x,y\right)\left(\frac{t}{b_{j,0}\left(x\right)y^{\beta_{j}}}\right)^{k/d_{j}}\right)\left(\log t\right)^{s_{j}}.\label{eq:h+}
\end{equation}
Differentiating the right side of (\ref{eq:h+}) with respect to $t$
gives
\[
\frac{1}{b_{j,0}\left(x\right)y^{\beta_{j}}} \left( \begin{array}{l}
\left( \sum_{k=n_{j}d_{j}}^{+\infty}\frac{k}{d_{j}}\widetilde{g}_{j,k}\left(x,y\right)\left(\frac{t}{b_{j,0}\left(x\right)y^{\beta_{j}}}\right)^{k/d_{j}-1}\right)\left(\log t\right)^{s_{j}}\\
+s_{j}\left(\sum_{k=n_{j}d_{j}}^{+\infty}\widetilde{g}_{j,k}\left(x,y\right)\left(\frac{t}{b_{j,0}\left(x\right)
y^{\beta_{j}}}\right)^{k/d_{j}-1}\right)\left(\log t \right)^{s_{j}-1}
\end{array} \right).
\]
Therefore, if we compute $\int_{a_{j}\left(x,y\right)}^{b_{j}\left(x,y\right)}\Gamma_{j,+}\left(x,y,t\right)\,\mathrm{d}t$
by integrating by parts $n_{j}$ times, where one begins by differentiating
the left side of (\ref{eq:h+}) and integrating $\mathrm{e}^{\mathrm{i}\sigma_{j}t}$
as before, one reduces to studying prepared generators for $\C^{\exp}\left(A_{\theta}\right)$
of the form
\[
T\left(x,y\right)=f_{j}\left(x\right)y^{p_{j}-n_{j}\beta_{j}}\left(\log y\right)^{q_{j}}\mathrm{e}^{\mathrm{i}\phi_{j}\left(x,y\right)}\int_{a_{j}\left(x,y\right)}^{b_{j}\left(x,y\right)}h\left(x,y,t\right)\left(\log t\right)^{s}\mathrm{e}^{\mathrm{i}\sigma_{j}t}\,\mathrm{d}t,
\]
where $h$ is a $\psi_{j,+}$-function and $s$ is a rational number. Since $h$ is bounded and
the length of the interval $\left(a_{j}\left(x,y\right),b_{j}\left(x,y\right)\right)$ is of order
$y^{\delta_{j}}$ as $y\to+\infty$ (see Definition \ref{notation for prep}
and Proposition \ref{lem: supercontinuous}), it follows that for
each $x\in\Pi_{m}\left(A\right)$ there is a constant $C\left(x\right)>0$
such that
\[
T^{\abs}\left(x,y\right)\leq C\left(x\right)y^{p_{j}-n_{j}\beta_{j}+\delta_{j}}\left(\log y\right)^{q_{j}+s}.
\]
Hence, by (\ref{eq:nj}) we can conclude that $T$ is superintegrable.
\end{proof}

\begin{lem}
\label{lem: b_j finite, alpha >0}In the notation of Proposition \ref{lem: supercontinuous},
suppose $b\equiv+\infty$; if $T_{j}$ is a prepared generator with
the property that $\alpha_{j}>0$, then we may write $T_{j}$ as a
finite sum of prepared generators which are either superintegrable
or in $\C_{\naive}^{\exp}\left(A_{\theta}\right)$. \end{lem}
\begin{proof}
We consider a generator $T_{j}$ as in the statement of the lemma.

In the notation of Proposition \ref{lem: supercontinuous}, suppose
first that $b_{j}<+\infty$. By Remark \ref{rem: bounded cells}(\ref{enu:r_j}),
we have $r_{j}=0$.

First assume that
\[
p_{j}+\delta_{j}<-1.
\]
Since $h_{j}$ is bounded by a constant and the length of the interval
$\left(a_{j}\left(x,y\right),b_{j}\left(x,y\right)\right)$ is of order $y^{\delta_{j}}$ as $y\to+\infty$,
it follows that for each $x\in\Pi_{m}\left(A\right)$ there is a constant $C\left(x\right)>0$
such that
\begin{equation}
T_{j}^{\abs}\left(x,y\right)\leq C\left(x\right)y^{p_{j}+\delta_{j}}\left(\log y\right)^{q_{j}+s_{j}}.\label{eq:TsuperO}
\end{equation}
So $T_{j}$ is superintegrable, and we are done.

So now assume that $p_{j}+\delta_{j}\geq-1$. Note that
\[
\PD{}{}{t}\left(\left(\frac{a_{j,0}\left(x\right)y^{\alpha_{j}}}{t}\right)^{1/d_{j}}\right)=-\frac{1}{d_{j}}\left(\frac{a_{j,0}\left(x\right)y^{\alpha_{j}}}{t}\right)^{1/d_{j}}\frac{1}{t},
\]
and that
\[
\PD{}{}{t}\left(\left(\frac{t}{b_{j,0}\left(x\right)y^{\beta_{j}}}\right)^{1/d_{j}}\right)=\frac{1}{d_{j}}\left(\frac{t}{b_{j,0}\left(x\right)y^{\beta_{j}}}\right)^{1/d_{j}}\frac{1}{t}.
\]
Write $h_{j}=H_{j}\circ\psi_{j}$, where $H_{j}\left(X_{1},\ldots,X_{N},Y,T_{1},T_{2}\right)$
is a power series converging in a neighbourhood of the closure of the
image of $\psi_{j}$. Thus we can factor out $1/t$ every time we
differentiate the expression $h_{j}\left(x,y,t\right)\left(\log t\right)^{s_{j}}$ with
respect to $t$. Moreover, the factor $1/t$ may be written as
\[
\frac{1}{t}=\frac{1}{a_{j,0}\left(x\right)y^{\alpha_{j}}}\left(\frac{a_{j,0}\left(x\right)y^{\alpha_{j}}}{t}\right).
\]
Therefore if we compute $\gamma_{j}\left(x,y,t\right)$ by integrating by parts,
where one integrates
$\mathrm{e}^{\mathrm{i}\sigma_{j}t}$ and
differentiates $h_{j}\left(x,y,t\right)\left(\log t\right)^{s_{j}}$, we can express $T_{j}$ as a finite
sum of terms, each of which is either in $\C_{\text{naive}}^{\exp}\left(A_{\theta}\right)$
or is of the same form as $T_{j}$, but with $p_{j}$ replaced by
$p_{j}-\alpha_{j}$. Therefore by repeating this strategy finitely
many times, we sufficiently decrease the value of $p_{j}$ in order
to reduce to the case that $p_{j}+\delta_{j}<-1$, and we are done
for the case $b_{j}<+\infty$.

It remains to consider the case $b_{j}\equiv+\infty$. By Remark \ref{rem: bounded cells}(\ref{enu:r_j}),
$r_{j}<-1$ and $h_{j}$ is a $\psi_{j,-}$-function. This case can
be handled very similarly to the previous case: one decreases the
value of $p_{j}$ by repeatedly integrating by parts in order to additionally
assume that
\[
p_{j}<-1.
\]
Since
\[
\gamma_{j}^{\abs}\left(x,y\right)\leq M\int_{1}^{+\infty}t^{r_{j}}\left(\log t\right)^{s_{j}}\,\mathrm{d}t<+\infty,
\]
where $|h_{j}|\leq M$, this shows that the analogue of (\ref{eq:TsuperO})
is now
\[
T_{j}^{\abs}\left(x,y\right)\leq C\left(x\right)y^{p_{j}}\left(\log y\right)^{q_{j}},
\]
hence $T_{j}$ is superintegrable.
\end{proof}

We now complete the proof of Proposition \ref{lem no y in the bounds}.
\begin{proof}
[Proof of Proposition \ref{lem no y in the bounds}]Let $b\equiv+\infty$
and $T_{j}$ be as in Proposition \ref{lem: supercontinuous}. If
$T_{j}$ is either superintegrable or in $\C_{\naive}^{\exp}\left(A_{\theta}\right)$,
then we are done. Otherwise, thanks to the lemmas above we may assume
that $\alpha_{j}=\beta_{j}=0$ (recall that if $b_{j}\equiv+\infty$
we have set $b_{j,0}\equiv+\infty$ and $\beta_{j}=0$). To see this,
if $\alpha_{j}>0$, then apply Lemma \ref{lem: b_j finite, alpha >0}.
Suppose now that $\alpha_{j}=0$. If $b_{j}\equiv+\infty$ or $\beta_{j}=0$,
then we are done. Otherwise, apply Lemma \ref{lem alpha =00003D0}
and again Lemma \ref{lem: b_j finite, alpha >0}.

We first establish the following claim: up to replacing $a\left(x\right)$
with some analytic subanalytic $\widetilde{a}\left(x\right)\geq a\left(x\right)$
and up to further partitioning with respect to the variables $x$,
we may assume that for all $j\in J$,
\begin{enumerate}
\item $|a_{j}\left(x,y\right)-a_{j,0}\left(x\right)|\leq1$ and $|b_{j}\left(x,y\right)-b_{j,0}\left(x\right)|\leq1$
on $A_{\theta}$, and
\item the function $h_{j}$ extends to a $\psi_{j}$-function with $\psi_{j}$
now defined on the set
\[
\widetilde{A_{j}}=\{\left(x,y,t\right):\ \left(x,y\right)\in A_{\theta},\min\{a_{j,0}\left(x\right),a_{j}\left(x,y\right)\}<t<\max\{b_{j,0}\left(x\right),b_{j}\left(x,y\right)\}\}.
\]

\end{enumerate}
To establish the claim, for each $j\in J$, fix a subanalytic neighbourhood
$U_{j}$ of the closure of $\psi_{j}\left(A_{j}\right)$ such that $h_{j}=H_{j}\circ\psi_{j}$
for some power series $H_{j}$ centred at the origin and converging
on $U_{j}$. Recall that every $\psi$ is bounded and subanalytic.
Hence, for each $\psi$-unit $u\in\{u_{a_{j}},u_{b_{j}},u_{c_{j}}\}$,
$\lim_{y\to+\infty}u\left(x,y\right)$ is a well-defined subanalytic function
of $x$ (which may be supposed to be analytic, up to refining the
partition), and therefore may be considered as a part of the 
corresponding coefficient
function in $\{a_{j,0},b_{j,0},c_{j,0}\}$. We may therefore assume
that for each $u\in\{u_{a_{j}},u_{b_{j}},u_{c_{j}}\}$,
\[
\lim_{y\to+\infty}u\left(x,y\right)=1.
\]
In particular, $\lim_{y\to+\infty}a_{j}\left(x,y\right)=a_{j,0}\left(x\right)$ and $\lim_{y\to+\infty}b_{j}\left(x,y\right)=b_{j,0}\left(x\right)$.
Hence, for each $x\in\Pi_{m}\left(A\right)$ there exists a real number $\widetilde{a}\left(x\right)\geq a\left(x\right)$
such that for all $y>\widetilde{a}\left(x\right)$ and all $j\in J$,
we have that $|a_{j}\left(x,y\right)-a_{j,0}\left(x\right)|\leq1$, that $|b_{j}\left(x,y\right)-b_{j,0}\left(x\right)|\leq1$
and that $\psi_{j}\left(\tld{A}_{j}\right)\subseteq U_{j}$. By definable choice
(see for example \cite[Chapter 6]{vdd:tame}), we may take $\widetilde{a}$
to be a subanalytic function of $x$ (and we may be suppose $\widetilde{a}$
to be analytic, up to refining the partition). This establishes the
claim.

We may therefore partition $A_{\theta}$ according to the conditions
$a\left(x\right)<y<\widetilde{a}\left(x\right)$ and $\widetilde{a}\left(x\right)<y$. We are done on the subset
of $A_{\theta}$ defined $a\left(x\right)<y<\widetilde{a}\left(x\right)$ (as in the case of
$b<+\infty$ treated in Proposition \ref{lem: supercontinuous}),
so it suffices to consider the subset of $A_{\theta}$ defined by
$y>\widetilde{a}\left(x\right)$. Therefore up to changing notation, we may simply
assume that $\widetilde{a}\left(x\right)=a\left(x\right)$.

Now, write
\[
\int_{a_{j}\left(x,y\right)}^{b_{j}\left(x,y\right)}\Gamma_{j}\left(x,y,t\right)\,\mathrm{d}t=\int_{a_{j,0}\left(x\right)}^{b_{j,0}\left(x\right)}\Gamma_{j}\left(x,y,t\right)\,\mathrm{d}t-\int_{a_{j,0}\left(x\right)}^{a_{j}\left(x,y\right)}\Gamma_{j}\left(x,y,t\right)\,\mathrm{d}t+\int_{b_{j,0}\left(x\right)}^{b_{j}\left(x,y\right)}\Gamma_{j}\left(x,y,t\right)\,\mathrm{d}t
\]
(note that when $b_{j}\equiv+\infty$ the last term of the sum does
not appear). Remark that
\[
\int_{a_{j,0}\left(x\right)}^{a_{j}\left(x,y\right)}\Gamma_{j}\left(x,y,t\right)\,\mathrm{d}t=\mathrm{e}^{\mathrm{i}\sigma_{j}a_{j,0}\left(x\right)}\int_{a_{j,0}\left(x\right)}^{a_{j}\left(x,y\right)}t^{r_{j}}h_{j}\left(x,y,t\right)\left(\log t\right)^{s_{j}}\mathrm{e}^{\mathrm{i}\sigma_{j}\left(t-a_{j,0}\left(x\right)\right)}\,\mathrm{d}t,
\]
and, thanks to the claim, $|a_{j}\left(x,y\right)-a_{j,0}\left(x\right)|\leq1$ on $A_{\theta}$.
Hence, $\mathrm{e}^{\mathrm{i}\sigma_{j}\left(t-a_{j,0}\left(x\right)\right)}$ is a complex-valued
subanalytic function (see Definition \ref{def complex valued}) over
its domain of integration. So the integral on the right side of the
above equation is a complex-valued constructible function on $A_{\theta}$,
because $\C$ is stable under integration (see \cite{cluckers-miller:stability-integration-sums-products,cluckers_miller:loci_integrability}).
This shows that $\left(x,y\right)\mapsto\int_{a_{j,0}\left(x\right)}^{a_{j}\left(x,y\right)}\Gamma_{j}\left(x,y,t\right)\,\mathrm{d}t$
is in $\C_{\naive}^{\exp}\left(A_{\theta}\right)$. For similar reasons, $\left(x,y\right)\mapsto\int_{b_{j,0}\left(x\right)}^{b_{j}\left(x,y\right)}\Gamma_{j}\left(x,y,t\right)\,\mathrm{d}t$
is also in $\C_{\naive}^{\exp}\left(A_{\theta}\right)$.
\end{proof}

We are now ready to finish the proof of the Preparation Theorem. In
view of Proposition \ref{lem no y in the bounds}, it only remains
to show that those generators for which the variable $y$ does not
appear in the integration limits of the $\gamma$-function can be
expressed as finite sums of generators which are either superintegrable
or naive in $y$. Moreover, we need ensure that Property (\ref{enu:2b})
in the statement of Theorem \ref{thm:prepExpConstr} is also satisfied.
\begin{proof}
[Proof of Theorem \ref{thm:prepExpConstr}] Let $b\equiv+\infty$
and consider a generator $T_{j}$ as in the statement of Proposition
\ref{lem no y in the bounds}, which is neither superintegrable nor
in $\C_{\naive}^{\exp}\left(A_{\theta}\right)$. Thus $T_{j}$ is such that $\gamma_{j}\left(x,y\right)=\int_{a_{j,0}\left(x\right)}^{b_{j,0}\left(x\right)}\Gamma_{j}\left(x,y,t\right)\,\mathrm{d}t$
(where we also allow the possibility $b_{j,0}\equiv+\infty$) and,
since $\alpha_{j}=\beta_{j}=0$, the variable $y$ now only appears
in the component $\left(\frac{a\left(x\right)}{y}\right)^{1/d}$ of
$\psi_{j}$ (see Equations (\ref{eq: psi}) and (\ref{eq: psi_j})).
Hence, we can now expand $h_{j}\left(x,y,t\right)$ as a power series
in the variable $\left(a\left(x\right)/y\right)^{1/d}$ with coefficients in the variables
$\left(x,t\right)$.
The powers of $y$ which appear in $T_j$ are thus of the form 
$p_j-\frac{n}{d}$, where $n$ is the summation index in the power series expansion 
of $h_j$. Since finitely many of such powers are greater than or equal to $-1$, we can 
write $T_{j}$ as a sum of finitely many terms
that are naive in $y$ plus a final term of the form
\[
f_{j}\left(x\right)y^{p}\left(\log y\right)^{q_{j}}\mathrm{e}^{\mathrm{i}\phi_{j}\left(x,y\right)}\int_{a_{j,0}\left(x\right)}^{b_{j,0}\left(x\right)}t^{r_{j}}h\left(x,y,t\right)\left(\log t\right)^{s_{j}}\mathrm{e}^{\mathrm{i}\sigma_{j}t}\,\mathrm{d}t
\]
for some rational $p<-1$ and $\psi_{j}$-function $h$.
This final
term is clearly superintegrable, since $h$ is bounded.

Summing up, we have written $f\circ P_{\theta}\left(x,y\right)$ as a finite
sum of generators which are either superintegrable or of the form
\begin{equation}
T_{j}\left(x,y\right)=f_{j}\left(x\right)y^{r_{j}}\left(\log y\right)^{s_{j}}\mathrm{e}^{\mathrm{i}\phi_{j}\left(x,y\right)},\label{eq:naive gen}
\end{equation}
where $f_{j}\in\C^{\exp}\left(\Pi_{m}\left(A\right)\right)$, $r_{j}\in\QQ\cap[-1,+\infty)$,
$s_{j}\in\NN$, and $\phi_{j}\in\mathcal{S}\left(A_{\theta}\right)$.

It remains to prove Property (\ref{enu:2b}) in the statement if Theorem
\ref{thm:prepExpConstr}.

Let $J'=\{ j:\ T_{j}\ \text{is\ as\ in\ }\eqref{eq:naive gen}\} $
and apply Proposition \ref{prop:prepSubConstr} to the collection
$\{ \phi_{j}:\ j\in J'\} $. Focus on a cell $A'=\{\left(x,y\right):\ x\in\Pi_{m}\left(A'\right),a'\left(x\right)<\sigma'\left(y-\theta'\left(x\right)\right)^{\tau'}<b'\left(x\right)\}\subseteq A_{\theta}$
that this constructs, along with its associated centre $\theta'$
and map $\psi'$ given by
\[
\begin{aligned}\psi'\left(x,y\right)=\left(c_{1}'\left(x\right),\ldots,c'_{N'}\left(x\right),\left(\frac{a'\left(x\right)}{y}\right)^{1/d'},\left(\frac{y}{b'\left(x\right)}\right)^{1/d'}\right) & \ \text{if}\ b'<+\infty,\\
\text{and}\ \psi'\left(x,y\right)=\left(c_{1}'\left(x\right),\ldots,c'_{N'}\left(x\right),\left(\frac{a'\left(x\right)}{y}\right)^{1/d'}\right) & \ \text{if}\ b'\equiv+\infty
\end{aligned}
\]
on
\[
A'_{\theta'}=\{\left(x,y\right):\ x\in\Pi_{m}\left(A'\right),a'\left(x\right)<y<b'\left(x\right)\}.
\]

First suppose that $\theta'\neq0$. Then the closure of $\{y/\theta'\left(x\right):\ \left(x,y\right)\in A'\}$
is a compact subset of $\left(0,+\infty\right)$, so each of the fibres $A'_{x}$
is bounded above. We are then done on $A'$ by Remark \ref{rem: bounded cells}(\ref{enu: supercontin}).

Now suppose that $\theta'=0$. Because $A'_{x}\subseteq\left(1,+\infty\right)$
for each $x$, it follows that $\sigma'=\tau'=1$. Thus $A'=\{\left(x,y\right):\ x\in\Pi_{m}\left(A'\right),a'\left(x\right)<y<b'\left(x\right)\}$
with $1\leq a\leq a'<b'$. When $b'<+\infty$, we are again done on
$A'$, so assume that $b'\equiv+\infty$. We may assume that the list
of functions $c'_{1},\ldots,c'_{N'}$ contains $c_{1},\ldots,c_{N}$
and also $\left(a\left(x\right)/a'\left(x\right)\right)^{1/d}$. We may also assume that $d'$ was
chosen so that $1/d$ is an integer multiple of $1/d'$. So because
\[
\left(\frac{a\left(x\right)}{y}\right)^{1/d}=\left(\frac{a\left(x\right)}{a'\left(x\right)}\right)^{1/d}\left(\frac{a'\left(x\right)}{y}\right)^{1/d},
\]
it follows that each component of $\psi$ is a $\psi'$-function.
Therefore to simplify notation, we may simply assume that $A'=A_{\theta}$
and that $\psi'=\psi$.

Hence, on $A_{\theta}$ we can write, for all $j\in J'$,
\[
\phi_{j}\left(x,y\right)=\phi_{j,0}\left(x\right)y^{l_{j}}u_{j}\left(x,y\right),
\]
where $\phi_{j,0}\in\mathcal{S}\left(\Pi_{m}\left(A\right)\right)$
is analytic, $l_{j}\in\mathbb{Q}$ (an integer multiple of $1/d$)
and $u_{j}$ is a $\psi$-unit. We expand the unit $u_{j}$ with respect
to the variable $\left(\frac{a\left(x\right)}{y}\right)^{1/d}$ and
multiply by $y^{l_{j}}$, so that we can rewrite the above equation
as
\[
\phi_{j}\left(x,y\right)=\phi_{j,+}\left(x,y\right)+\phi_{j,-}\left(x,y\right),
\]
 where $\phi_{j,+}\in\mathcal{S}\left(\Pi_{m}\left(A\right)\right)\left[y^{1/d}\right]$ and $\phi_{j,-}\left(x,y\right)=\phi_{j,0}\left(x\right)\left(\frac{a\left(x\right)}{y}\right)^{1/d}\widetilde{u}_{j}\left(x,y\right)$,
for some $\psi$-unit $\widetilde{u}_{j}$. Up to refining the partition
with respect to the variables $x$, we may assume that $|\phi_{j,0}|$
is either bounded from above or bounded away from zero. In either
of the two cases, $\phi_{j,-}$ is a $\psi$-function. This is clear
in the first case. In the second case, up to further partitioning
the cell (as we have done for example in the proof of Proposition
\ref{lem no y in the bounds}), we may suppose that $y>a\left(x\right)\left(\phi_{j,0}\left(x\right)\right)^{d}$.
We then modify $\psi$ accordingly, by adding the bounded function
$\left(\frac{1}{\phi_{j,0}}\right)^{d}$ to $c_{1}\left(x\right),\ldots,c_{N}\left(x\right)$
and considering the function $\left(\frac{a\left(x\right)\left(\phi_{j,0}\left(x\right)\right)^{d}}{y}\right)^{1/d}$
as the last component of $\psi$.

Therefore, $\text{exp}\left(\mathrm{i}\phi_{j,-}\right)$ is a complex-valued
subanalytic function (see Definition \ref{def complex valued}) which
can be expanded as a power series $F$ in the variable $y^{-1/d}$
with analytic functions of $x$ as coefficients. Let $K_{j}\in\mathbb{N}$
be such that, in the notation of Equation (\ref{eq:naive gen}), $r_{j}-\frac{K_{j}}{d}<-1$.
We split the power series $F$ into a polynomial part, by summing
up to $K_{j}$, and a rest $G$. Therefore, we can replace $T_{j}$
by a finite sum of terms of the form appearing in Equation (\ref{eq:naive gen}),
but with the further property that $\phi_{j}\in\mathcal{S}\left(\Pi_{m}\left(A\right)\right)[y^{1/d}]$,
plus a final superintegrable term (corresponding to the rest $G$
of the series).

Summing up, we have partitioned the index set $J$ as $J^{\text{int}}\cup J^{\text{naive}}$,
where $T_{j}$ is superintegrable for every $j\in J^{\text{int}}$
and for all $j\in J^{\text{naive}}$, $T_{j}$ is of the form in Equation
(\ref{eq:naive gen}) with $\phi_{j}\in\mathcal{S}\left(\Pi_{m}\left(A\right)\right)[y^{1/d}]$.
Now, by writing
\[
\mathrm{e}^{\mathrm{i}\phi_{j}\left(x,y\right)}=\mathrm{e}^{\mathrm{i}\left(\phi_{j}\left(x,y\right)-\phi_{j}\left(x,0\right)\right)}\mathrm{e}^{\mathrm{i}\phi_{j}\left(x,0\right)}
\]
 and absorbing $\mathrm{e}^{\mathrm{i}\phi_{j}\left(x,0\right)}$ into $f_{j}\left(x\right)$,
we may assume that $\phi_{j}\left(x,0\right)=0$ for all $j\in J^{\text{naive}}$.

By further partitioning in $x$, we may also assume that for all $j,k\in J^{\text{naive}}$,
$y\mapsto\phi_{j}\left(x,y\right)$ and $y\mapsto\phi_{k}\left(x,y\right)$ either define
the same polynomial function for all $x\in\Pi_{m}\left(A\right)$ or define different
polynomial functions for all $x\in\Pi_{m}\left(A\right)$. Therefore by summing
over terms for $j\in J^{\text{naive}}$ with equal tuples $\left(r_{j},s_{j},\phi_{j}\left(x,y\right)\right)$,
we may assume that these tuples are distinct. We have thus completed
the proof of Theorem \ref{thm:prepExpConstr}.
\end{proof}

\section{Proof of Theorem \ref{thm:interp-1}\label{s: proof of main results}}
In this section we complete the proof of Theorem  
\ref{thm:interp-1} using Proposition 
\ref{prop:goal1}(\ref{V big}) below, which states that, denoting by $f$ the function 
$\sum_{j\in J}c_{j}\mathrm{e}^{\mathrm{i}p_{j}\left(t\right)}$, then
there exists a real number $\varepsilon>0$
such that the set $V_{\varepsilon}=\{ t\in[0,+\infty): \left|f\left(t\right)\right|\ge\varepsilon\} $
is not too sparse. To prove Proposition \ref{prop:goal1}(\ref{V big}) let us first introduce a definition and notation. 
In what follows the notation $\vol_\ell$ stands for the Lebesgue measure
in the corresponding space $\RR^\ell$, $\ell\ge 1$. All sets and maps involved with this notation are tacitly assumed to be Lebesgue measurable. 
%

\begin{defn} \label{def CUD map}
Let $\{ x\}:=x-\lfloor x \rfloor$ denote the fractional part of the real number $x$ and let $p=\left(p_1, \ldots, p_\ell\right)\colon  [0,+\infty)\to \mathbb{R}^\ell$, be a map. 
If $I_1, \ldots, I_\ell\subseteq \RR$ are bounded intervals with nonempty interior, we denote by 
$I$ the box $\prod_{j=1}^\ell I_j$. For $T\ge 0$ we let
$$ W_{p,I,T} : = \left\{ t\in [0,T] : \{ p\left(t\right) \}\in I \right\}, $$
where $\{p(t)\}$ denotes the tuple $\left(\{ p_1(t)\},\ldots, \{p_\ell\left(t\right)\}\right)$.

The map $p$ is
said to be \textbf{\emph{continuously uniformly distributed modulo $1$}}, in short \textbf{\emph{ c.u.d. mod $1$}}, if for every box $I\subseteq [0,1)^\ell$,  
$$ \lim_{T\to +\infty} \frac{\vol_1\left( W_{p,I,T}\right)}{ T } =
\vol_\ell\left(I\right).  $$    
\end{defn}
%
%
\begin{rem}\label{rem Polynomial maps are CUD}
By \cite{Wey} and \cite[Corollary 9.1]{Kui}, a polynomial 
map $p=\left(p_1, \ldots, p_\ell\right)$ is c.u.d. mod $1$, provided that
no nontrivial linear combination over $\mathbb{Z}$ of the polynomials $p_j$ is constant. 
\end{rem}
\begin{lem}
\label{lem:main-lemma-appendix} 
Let $p\colon  [0,+\infty)\to \mathbb{R}^\ell$ be a c.u.d. mod $1$ map, 
$I\subseteq [0,1)^\ell$ be a box
and 
$$
W_{p,I}:= \{ t\in \mathbb{R} : \{p\left(t\right)\}\in I \}. $$
Then for all sufficiently large $k\in \NN$,
\[
 \vol_1\left(W_{p,I} \cap [2^k,2^{k+1}]\right) \ge 2^{k-1}\vol_\ell\left(I\right)
\ \ \mathrm{ and } \ \
\int_{W_{p,I}}\frac{\mathrm{d}t}{t}=+\infty.
\]
\end{lem}
\begin{proof}
Let us denote $\vol_\ell\left(I\right)$ by $s$. By definition there exists $T_0\ge 0$ such that for every $T\ge T_0$, 
$$ \vol_1\left(W_{p,I,T}\right)\in [T\frac{5s}{6}, T\frac{7s}{6}].  $$
It follows that for any $k\in \mathbb{N}$ such that 
$2^k\ge T_0$,  
$$ \vol_1\left(W_{p,I}\right) \cap [2^k,2^{k+1}]= 
\vol_1\left( W_{p,I,2^{k+1}} \setminus W_{p,I,2^k} \right) \ge 
s\frac{5}{6}2^{k+1} - s \frac{7}{6}2^{k}=s2^{k-1}. $$
Therefore, denoting by $k_0$ the smallest integer $k$ such that $2^{k}\ge T_0$,
we have
$$
\int_{W_{p,I}} \frac{\mathrm{d}t}{t} 
 \ge\int_{W_{p,I} \cap[2^{k_0},+\infty)}\frac{{ \mathrm{d}t}}{t} 
=\sum_{k=k_0}^{+\infty}
 \int_{ W_{p,I} \cap [2^k,2^{k+1}]}\frac{\mathrm{d}t}{t} 
\ge s \sum_{k=k_0}^{+\infty} \frac{2^{k-1}}{2^{k+1}}=+\infty,
$$
which concludes the proof.
\end{proof}

\begin{rem}\label{rem Z-free polynomials}
Let  $c_1, \ldots, c_n\in \CC\setminus \{0\}$ and let $p_1\left(t\right), \ldots, p_n\left(t\right)\in \mathbb{R}[t]$ be distinct polynomials
such that $p_1\left(0\right)=\ldots= p_n\left(0\right)=0$ and such at least one of them is not 
constantly zero. Consider the function 
$f\left(t\right)=\sum_{j=1}^n c_j\mathrm{e}^{\mathrm{i} p_j\left(t\right)}$. Extract from the family of the polynomials $p_j$ a basis of the $\mathbb{Q}$-vector space spanned by this family. Without loss of generality, we may suppose that such a basis is given by $\left(p_1, \ldots, p_\ell\right)$. Write 
$$p_k= r_{k,1}p_1+\ldots +r_{k,\ell}p_\ell, \
k=\ell+1, \ldots, n,$$ 
with $r_{k,j}\in \mathbb{Q}$ for $k=\ell+1, \ldots, n, j=1, \ldots, \ell$. If, for $j=1, \ldots, \ell$, we denote by $\rho_j$
the least common multiple of the denominators of the nonzero rational numbers among $r_{\ell+1,j}, \ldots, r_{n,j}$
and we let $\widetilde{p}_j=p_j/2\pi\rho_j$, then we have, for $\ k=\ell+1, \ldots, n$,
$$p_k=2\pi s_{k,1} \widetilde{p}_1+\ldots + 2\pi s_{k,\ell}\widetilde{p}_\ell,$$
where $ s_{k,1}, \ldots,  s_{k,\ell} \in \mathbb{Z}$ and the family of polynomials $\left(\widetilde{p}_1, \ldots, \widetilde{p}_\ell\right)$ is independent over $\mathbb{Z}$. To sum up, one can  write 
$$f\left(t\right)=P\left(\mathrm{e}^{2\pi\mathrm{i} \widetilde{p}_1\left(t\right)}, \ldots, \mathrm{e}^{2\pi\mathrm{i} \widetilde{p}_\ell\left(t\right)}\right),$$ 
where $P$ is a Laurent polynomial in
$ \mathbb{C}[X_1, \ldots, X_\ell, \frac{1}{X_1}, \ldots, \frac{1}{X_\ell}]$ which contains at least $\ell\ge1$ monomials of the form $c_{j}X_{j}^{\rho_{j}}$, with
$c_{j}\ne0$ and $\rho_{j}\in\mathbb{N}$. Therefore $P$ is not a constant (note that we have not assumed that the function $f$ is not constant).

Now since the family $\left(\widetilde{p}_1, \ldots, \widetilde{p}_\ell\right)$ is independent over $\mathbb{Z}$ and since  
$p_1\left(0\right)=\ldots= p_\ell\left(0\right)=0$, no nontrivial $\mathbb{Z}$-linear combination of $\widetilde{p}_1,\ldots, \widetilde{p}_\ell$ is 
constant, thus by Remark \ref{rem Polynomial maps are CUD}, 
the map $p=(\widetilde{p}_1,\ldots, \widetilde{p}_\ell)$ is c.u.d. mod $1$.

\end{rem}

\begin{prop}
\label{prop:goal1} 
Let $f\colon \RR\to\CC$ be given by a finite sum
\[
f\left(t\right)=\sum_{j\in J}c_{j}\mathrm{e}^{\mathrm{i}p_{j}\left(t\right)},
\]
where the $c_{j}\in \CC\setminus\{0\}$ and the $p_{j}\left(t\right)$
are  distinct polynomials in $\mathbb{R}[t]$, vanishing at $0$. 
Then one can find $\varepsilon>0$ such that
\begin{enumerate}
%
\item \label{enu: existence of limit}
There exist two sequences $\left(t_{0,n}\right)_{n\in \mathbb{N}}$  and $\left(t_{1,n}\right)_{n\in \mathbb{N}}$, which both tend to $+\infty$, 
such that $\forall n\in \NN$, $\vert f\left(t_{0,n}\right)-f\left({t_{1,n}}\right) \vert \ge \varepsilon$.
 In particular  $\lim_{t\to+\infty}f\left(t\right)$ 
exists if and only if $p_{j}=0$ for all $j\in J$
(in other words, 
if and only if $f$ is a constant function).
%
\item \label{enu:density pointwise}
 There exists
a sequence $\left(t_n\right)_{n\in \mathbb{N}}$ which tends to $+\infty$ such that 
for all $n\ge 0$, $\vert f\left(t_n\right) \vert \ge \varepsilon$. 
%
\item \label{V big} 
$
\int_{V_{\varepsilon}}\frac{1}{t}\,\mathrm{d}t=+\infty,
$
where
$ V_{\varepsilon}=\{t\in[1,+\infty):\ \left|f\left(t\right)\right|\geq\varepsilon\}. $
\end{enumerate}
\end{prop}

\begin{proof}
We may assume without loss of generality that
$J$ is $\{ 1, \ldots, n\}$.
By Remark \ref{rem Z-free polynomials}, one can write 
$f\left(t\right)=P\left(\mathrm{e}^{2\pi\mathrm{i} \widetilde{p}_1\left(t\right)}, \ldots, 
\mathrm{e}^{2\pi\mathrm{i} \widetilde{p}_\ell\left(t\right)}\right),$ 
where $P$ is a nonconstant Laurent polynomial in
$ \mathbb{R}[X_1, \ldots, X_\ell,\frac{1}{X_1}, \ldots, 
\frac{1}{X_\ell} ]$ and  $\left(\widetilde{p}_1, \ldots, \widetilde{p}_\ell\right)$ is a c.u.d. mod $1$ map. 
Set $a_1=\mathrm{e}^{2\pi\mathrm{i} \alpha_1},  b_1=\mathrm{e}^{2\pi\mathrm{i} \beta_1}, \ldots, a_\ell
=\mathrm{e}^{2\pi\mathrm{i} \alpha_\ell}, b_\ell=
\mathrm{e}^{2\pi\mathrm{i} \beta_\ell} $, where $\alpha_1,\beta_1, \ldots, \alpha_\ell, \beta_\ell\in [0,1)$ are  complex numbers such that 
$$\vert P\left(a_1,\ldots, a_\ell\right)- P\left(b_1, \ldots, b_\ell\right)\vert \ge 3\varepsilon,$$ for some $\varepsilon>0$,
and let us then consider the following two sets 
$$ A=\{ t\in \mathbb{R} : \left(\{\widetilde{p}_1\left(t\right) \}, \ldots, \{ \widetilde{p}_\ell\left(t\right) \}\right)\in \prod_{j=1}^\ell A_j\} , \ \
B = \{ t\in \mathbb{R} : \left(\{ \widetilde{p}_1\left(t\right) \}, \ldots,
 \{ \widetilde{p}_\ell\left(t\right)\}\right)\in \prod_{j=1}^n B_\ell\},$$
where $A_j\subseteq[0,1)$ is an interval centred at 
$\alpha_j$ and $B_j\subseteq[0,1)$  is an  interval centred at $\beta_j$.
If we denote by $h\left(t\right)$ the map $\left(\mathrm{e}^{\mathrm{i}\widetilde{p}_1\left(t\right)}, \ldots, \mathrm{e}^{\mathrm{i} \widetilde{p}_\ell\left(t\right)}\right)$,  
since  $\left(\widetilde{p}_1, \ldots, \widetilde{p}_\ell\right)$ is a c.u.d. mod $1$ map, 
by the continuity of $h$ and $P$, by taking our intervals $A_j$ and $B_j$ sufficiently small, one can find two sequences $\left(t_{0,n}\right)_{n\in \NN}\in A$, $\left(t_{1,n}\right)_{n\in \NN}\in B$ both tending to $+\infty$
such that 
$$\forall n\in \NN, \ \vert P\left(h\left(t_{0,n}\right)\right)-P\left(a_1,\ldots,a_\ell\right) \vert \le \varepsilon \ \
\mathrm{  and } \ \ 
\vert P\left(h\left(t_{1,n}\right)\right)-P\left(b_1,\ldots,b_\ell\right) \vert \le \varepsilon.$$ This gives that
$\forall n\in \NN$, $\vert  f\left(t_{0,n}\right)-f\left(t_{1,n}\right) \vert \ge \varepsilon$ and proves $\left(1\right)$. 
To prove $\left(2\right)$ we repeat the same argument as in $\left(1\right)$: we choose 
 complex numbers $a_1=\mathrm{e}^{2\pi\mathrm{i} \alpha_1},   \ldots, a_\ell
=\mathrm{e}^{2\pi\mathrm{i} \alpha_\ell} $, with $\alpha_1  \ldots, \alpha_\ell,  \in [0,1)$, such that 
$\vert P\left(a_1,\ldots, a_\ell\right) \vert \ge 2\varepsilon$ and we define 
as above the corresponding sets $A_1, \ldots, A_\ell$ and $A$ with the property that when $t\in A$,
$\vert f\left(t\right) - P\left(a_1,\ldots,a_\ell\right) \vert\le \varepsilon$. One thus  has that for every $t\in A$, 
$\vert f\left(t\right)\vert \ge  \varepsilon $. However, $A$ certainly contains a sequence 
$\left(t_n\right)_{n\in \NN}$ which tends to $+\infty$, since  $\left(\widetilde{p}_1, \ldots, \widetilde{p}_\ell\right)$ is a c.u.d. mod $1$ map. This proves $\left(2\right)$.

%
%

Now, since the set $A$ defined above is such that
$A \subseteq V_\varepsilon$, and since 
$\left(\widetilde{p}_1, \ldots, \widetilde{p}_\ell\right)$ is c.u.d. mod $1$,
 by Lemma \ref{lem:main-lemma-appendix} we have proved $\left(3\right)$. 
\end{proof}
%

We now complete the proof of Theorem \ref{thm:interp-1}.
\begin{proof}
[Proof of Theorem \ref{thm:interp-1}]Let $f\in\C^{\exp}\left(X\times\RR\right)$
and apply Theorem \ref{thm:prepExpConstr} to $f$. This produces
a finite partition $\A$ of $X\times\RR$ into cells over $\RR^{m}$.
Consider one such cell $A\in\A$ that is open over $\RR^{m}$, and
let $\theta$ be a centre for $A$. Write
\[
f\circ P_{\theta}\left(x,y\right)=\sum_{j\in J^{\text{int}}}T_{j}\left(x,y\right)+\sum_{j\in J^{\text{naive}}}T_{j}\left(x,y\right)\ \text{on\ }A_{\theta}.
\]
 Therefore,
\[
f=\sum_{j\in J^{\text{int}}}T_{j}\circ P_{\theta}^{-1}+\sum_{j\in J^{\text{naive}}}T_{j}\circ P_{\theta}^{-1}\ \text{on\ }A.
\]
If $J^{\text{naive}}=\emptyset$, then we are done. So suppose $J^{\text{naive}}\not=\emptyset$,
which implies that $A_{\theta}$ is unbounded above (i.e. $b\equiv+\infty$,
in the notation of Definition \ref{def:center}).

Recall from Remark \ref{rems: integrability of constructible fncts}(\ref{enu:jac of p theta})
that
\[
\partial_{y}P_{\theta}\left(y\right):=\PD{}{P_{\theta,m+1}}{y}\left(x,y\right)=\sigma\tau y^{\tau-1},
\]
and that $\tau-1$ equals either $0$ or $-2$. Notice that
\[
\INT\left(T_{j}\circ P_{\theta}^{-1},\Pi_{m}\left(A\right)\right)=\INT\left(T_{j}\partial_{y}P_{\theta},\Pi_{m}\left(A\right)\right)\ \forall j\in J.
\]
For every $j\in J^{\text{naive}}$, in the notation of Equation (\ref{eq:naive}),
we have:
\begin{equation}
T_{j}\left(x,y\right)\partial_{y}P_{\theta}\left(y\right)=\sigma\tau f_{j}\left(x\right)y^{r_{j}+\tau-1}\left(\log y\right)^{s_{j}}\mathrm{e}^{\mathrm{i}\phi_{j}\left(x,y\right)},\label{eq:t_j Jac}
\end{equation}
which is integrable in $y$ if and only if $f_{j}\left(x\right)=0$ or $r_{j}+\tau<0$.
Therefore by defining
\[
J^{\INT}:=J^{\text{int}}\cup\{j\in J^{\text{naive}}:\ r_{j}+\tau<0\},
\]
we see that for each $j\in J$,
\[
\INT\left(T_{j}\partial_{y}P_{\theta},\Pi_{m}\left(A\right)\right)=\begin{cases}
\Pi_{m}\left(A\right), & \text{if \ensuremath{j\in J^{\INT}},}\\
\{x\in\Pi_{m}\left(A\right):\ f_{j}\left(x\right)=0\}, & \text{if \ensuremath{j\in J^{\text{naive}}:=J\setminus J^{\INT}}.}
\end{cases}
\]
Let
\begin{eqnarray*}
g\left(x,y\right) & = & \sum_{j\in J^{\INT}}T_{j}\circ P_{\theta}^{-1}\left(x,y\right),\quad\text{for all \ensuremath{\left(x,y\right)\in A},}\\
H & = & \bigcap_{j\in J^{\text{naive}}}\{x\in\Pi_{m}\left(A\right):\ f_{j}\left(x\right)=0\}.
\end{eqnarray*}
Notice that, by taking the sum of the squares of the real and imaginary
parts of $f_{j}$ (see Remark \ref{rem: important}(\ref{enu: real and imag parts, fourier})),
we can write $H=\{x\in\Pi_{m}\left(A\right):\ h\left(x\right)=0\}$, for some $h\in\C^{\exp}\left(X\right)$.

It is clear that
\[
\INT\left(g,\Pi_{m}\left(A\right)\right)=\Pi_{m}\left(A\right).
\]
It remains to show that
\[
\INT\left(f\restriction{A},\Pi_{m}\left(A\right)\right)=H
\]
and
\[
\text{\ensuremath{f\left(x,y\right)=g\left(x,y\right)}\ for all \ensuremath{\left(x,y\right)\in A}\ with\ \ensuremath{x\in\INT\left(f\restriction{A},\Pi_{m}\left(A\right)\right)}.}
\]

Clearly,
\[
\text{\ensuremath{f\left(x,y\right)=g\left(x,y\right)}\ for all \ensuremath{\left(x,y\right)\in A}\ with\ \ensuremath{x\in H},}
\]
so $H\subseteq\INT\left(f\restriction{A},\Pi_{m}\left(A\right)\right)$.

To prove the other inclusion, we show that if $x\not\in H$, then
$x\not\in\INT\left(\sum_{j\in J^{\mathrm{naive}}}T_{j}\partial_{y}P_{\theta},\Pi_{m}\left(A\right)\right)=\INT\left(\sum_{j\in J^{\mathrm{naive}}}T_{j}\circ P_{\theta}^{-1},\Pi_{m}\left(A\right)\right)$,
and hence $x\not\in\INT\left(f\restriction{A},\Pi_{m}\left(A\right)\right)$.

Fix $x\in\Pi_{m}\left(A\right)\setminus H$. Then the set $J:=\{ j\in J^{\text{naive}}:\ f_{j}\left(x\right)\not=0\} \subseteq J^{\text{naive}}$
is nonempty (see Equation (\ref{eq:t_j Jac})). Recall that the tuples
$\{ \left(r_{j}-\tau-1,s_{j},\phi_{j}\left(x,y\right)\right)\} _{j\in J}$
are distinct and that $\rho_{j}:=r_{j}-\tau-1\geq-1$ for all $j\in J$.
Let
\[
\begin{aligned}E=\{ \left(\rho,s\right)\in\mathbb{Q}\times\mathbb{N}:\ \exists j\in J\ \left(\rho_{j},s_{j}\right)=\left(\rho,s\right)\} \ \text{and,}\\
\text{if\ }\left(\rho,s\right)\in E,\ \text{let}\ E_{\left(\rho,s\right)}=\{ j\in J:\ \left(\rho_{j},s_{j}\right)=\left(\rho,s\right)\} .
\end{aligned}
\]
Write
\begin{eqnarray*}
F\left(x,y\right) & = & \sum_{j\in J}T_{j}\partial_{y}P_{\theta}\left(x,y\right)=\\
 & = & \sum_{\left(\rho,s\right)\in E}y^{\rho}\left(\text{log\ \!}y\right)^{s}\left(\sum_{j\in E_{\left(\rho,s\right)}}\sigma\tau f_{j}\left(x\right)\mathrm{e}^{\mathrm{i}\phi_{j}\left(x,y\right)}\right).
\end{eqnarray*}
Up to summing like terms, we may suppose that all polynomials $\phi_{j}$
in the previous sum are distinct. Let $\left(\rho_{0},s_{0}\right)$
be the lexicographic maximum of $E$ and $$G\left(x,y\right)=\sum_{j\in E_{\left(\rho_{0},s_{0}\right)}}\sigma\tau f_{j}\left(x\right)\mathrm{e}^{\mathrm{i}\phi_{j}\left(x,y\right)}.$$

By applying Proposition \ref{prop:goal1}(\ref{V big}) to the function $y\mapsto G\left(x,y\right)$, we obtain
 $\varepsilon>0$ such that for all $y\in V_{\varepsilon}$,
 $F\left(x,y\right)$ can be written as
\[
y^{\rho_{0}}\left(\log\ \! y\right)^{s_{0}}G\left(x,y\right)\left[1+\sum_{\left(\rho,s\right)\in E\setminus\{ \left(\rho_{0},s_{0}\right)\} }y^{\rho-\rho_{0}}\left(\text{log\ \!}y\right)^{s-s_{0}}G\left(x,y\right)^{-1}\left(\sum_{j\in E_{\left(\rho,s\right)}}\sigma\tau f_{j}\left(x\right)\mathrm{e}^{\mathrm{i}\phi_{j}\left(x,y\right)}\right)\right].
\]
Notice that there exists $M>0$ such that for all $y\in V_{\varepsilon}\cap[M,+\infty)$
the square bracket in the previous equation is bounded from below
by some positive constant $K$. Therefore,
\[
\int_{M}^{+\infty}\left|F\left(x,y\right)\right|\,\mathrm{d}y\geq\int_{V_{\varepsilon}\cap[M,+\infty)}\left|F\left(x,y\right)\right|\,\mathrm{d}y\]
\[\geq\varepsilon K\int_{V_{\varepsilon}\cap[M,+\infty)}y^{\rho_{0}}\left(\log y\right)^{s_{0}}\,\mathrm{d}y\geq\varepsilon K\int_{V_{\varepsilon}\cap[M,+\infty)}\frac{1}{y}\,\mathrm{d}y.
\]
However, by Proposition \ref{prop:goal1}(\ref{V big}) the last integral on the right
diverges.
 Hence, $x\not\in\INT\left(F,\Pi_{m}\left(A\right)\right)$,
and we are done.
\end{proof}



\section{Asymptotic expansions and limits\label{sec:Asymptotic-expansions-of}}

In this section we prove a series of consequences of our main
results and their proofs.

In Subsection \ref{sub:Asymptotic-expansions-of} we prove that functions
in $\mathcal{C}_{\text{naive}}^{\exp}$ have convergent asymptotic
expansions of a certain form. We use this result to produce in Subsection
\ref{sub:Two-functions-which} two examples of functions that are
in $\mathcal{C}^{\exp}$ but not in $\mathcal{C}_{\text{naive}}^{\exp}$.

In Subsection \ref{sub:Pointwise-limits} we prove that $\mathcal{C}^{\exp}$
is stable under taking pointwise limits.


\subsection{Asymptotic expansions of naive functions\label{sub:Asymptotic-expansions-of}}
\begin{defn}
\label{def: asympt exp}A collection $\mathcal{G}=\left( g_{n} \right)_{n\in\mathbb{N}}$
of functions $g_{n}:\left(0,+\infty\right)\rightarrow\mathbb{R}$ with strictly positive germ at $+\infty$
is an \textbf{\emph{asymptotic scale at $+\infty$}} if for all $n\in\mathbb{N},\ \lim_{y\rightarrow+\infty}\frac{g_{n+1}\left(y\right)}{g_{n}\left(y\right)}=0$.

An $\mathbb{C}$-vector space $\mathcal{A}$ of functions $a:\mathbb{R}\rightarrow\mathbb{C}$
is a \textbf{\emph{space of coefficients}} if for every $a\in\mathcal{A}\setminus\{ 0\} $
there are $\varepsilon>0$ and a sequence $\left( y_n\right)_{n\in\mathbb{N}}$,
with $\lim_{n\rightarrow+\infty}y_{n}=+\infty$, such that $\forall n\in\mathbb{N},\ |a\left(y_{n}\right)|>\varepsilon$.

Given a function $f:\left(0,+\infty\right)\rightarrow\mathbb{C}$,
an asymptotic scale $\mathcal{G}$ and a space of coefficients $\mathcal{A}$,
we say that $f$ \textbf{\emph{has a $\left(\mathcal{G},\mathcal{A}\right)$-asymptotic
expansion at $+\infty$}} if there are $y_{0}>0$ and a sequence $\left(a_{n}\left(y\right)\right)_{n\in\mathbb{N}}\subseteq\mathcal{A}$
such that
\[
\forall N\in\mathbb{N}\ \exists C>0\ \text{s.\ t.\ }\forall y>y_{0},\ \left|f\left(y\right)-\sum_{n=0}^{N}a_{n}\left(y\right) g_{n}\left(y\right)\right|\leq Cg_{N+1}\left(y\right).
\]
\end{defn}
\begin{lem}
\label{lem:uniqueness asympt exps} If a function $f$ admits a $\left(\mathcal{G},\mathcal{A}\right)$-asymptotic
expansion, then such an expansion is unique, i.e. the sequence $\left( a_{n}\left(y\right) \right)_{n\in\mathbb{N}}$
is uniquely determined.\end{lem}
\begin{proof}
Suppose that $\left( \widetilde{a}_{n}\left(y\right)\right) _{n\in\mathbb{N}}$
is another sequence of coefficients. Supposing inductively that $a_{n}=\widetilde{a}_{n}$
for all $n<N$, we have
\begin{eqnarray*}
\left|\left(a_{N}\left(y\right)-\widetilde{a}_{N}\left(y\right)\right)g_{N}\left(y\right)\right| & = & \left|\sum_{n=0}^{N}a_{n}\left(y\right)g_{n}\left(y\right)-\sum_{n=0}^{N}\widetilde{a}_{n}\left(y\right)g_{n}\left(y\right)\right|\\
 & \leq & \left|f\left(y\right)-\sum_{n=0}^{N}a_{n}\left(y\right)g_{n}\left(y\right)\right|+\left|f\left(y\right)-\sum_{n=0}^{N}\widetilde{a}_{n}\left(y\right)g_{n}\left(y\right)\right|\\
 & \leq & C_{0}g_{N+1}\left(y\right),
\end{eqnarray*}
for some constant $C_{0}>0$ and for $y$ sufficiently large. Dividing
by $g_{N}\left(y\right)$, we obtain that ${ \lim_{y\rightarrow+\infty}}\left|a_{N}\left(y\right)-\widetilde{a}_{N}\left(y\right)\right|=0$.
Now, the function $a\left(y\right):=a_{N}\left(y\right)-\widetilde{a}_{N}\left(y\right)$
belongs to $\mathcal{A}$, and if $a\left(y\right)$ is not identically
zero, then $a\left(y\right)$ is bounded away from zero on some sequence
of points going to $+\infty$. Hence the only way for $a\left(y\right)$
to tend to zero is if $a\left(y\right)$ is identically zero.
\end{proof}
\begin{prop}
\label{prop: esympt expansion of naive}
Let $f\in\mathcal{C}_{\mathrm{naive}}^{\mathrm{exp}}\left(\mathbb{R}\right)$.
Then $f$ has a $\left(\mathcal{G},\mathcal{A}\right)$-asymptotic
expansion, where
\begin{itemize}
\item $\mathcal{G}= \left(y^{r_{n}}\left(\log y\right)^{s_{n}}\right)_{n\in\mathbb{N}}$
with $r_{n}\in\mathbb{Q}\ s_{n}\in\mathbb{N}$ and $\left(r_{n},s_{n}\right)_{n\in\mathbb{N}}$
a decreasing sequence of lexicographically ordered pairs;
\item $\mathcal{A}=\left\{ E\left(y\right)=\sum_{j\in J}c_{j}\mathrm{e}^{\mathrm{i}p_{j}\left(y^{1/d}\right)}:\ \left(c_{j}\right)_{j\in J}\in\mathbb{C}^{J}\right\} $,
for some $d\in\mathbb{N}$, some finite set $J\subseteq\mathbb{N}$
and distinct polynomials $p_{j}\left(y\right)\in\mathbb{R}[y]$
with $p_{j}\left(0\right)=0$.
\end{itemize}
Moreover, if $E_{n}\left(y\right)=\sum_{j\in J}c_{j,n}\mathrm{e}^{\mathrm{i}p_{j}\left(y^{1/d}\right)}$
are the coefficients of such an expansion, then for all sufficiently
large $y$ and for all $j\in J$, the series
\begin{equation}\label{eq: coeff converg}
f_{j}\left(y\right)=\sum_{n\in\mathbb{N}}c_{j,n}y^{r_{n}}\left(\log y\right)^{s_{n}} 
\end{equation}
 converge absolutely and
\[
f\left(y\right)=\sum_{n\in\mathbb{N}}E_{n}\left(y\right)y^{r_{n}}\left(\log y\right)^{s_{n}}.
\]
\end{prop}
\begin{proof}
Notice first that, if $\mathcal{G}$ and $\mathcal{A}$ are as in
the statement, then $\mathcal{G}$ is an asymptotic scale and, by
Proposition \ref{prop:goal1}(\ref{enu:density pointwise}),
$\mathcal{A}$ is indeed a space of coefficients.

By Remark \ref{rem: lojas}, if $g\in\mathcal{S}\left(\mathbb{R}\right)$,
then, in the notation of Equation (\ref{eq: poly + rest}), we have
\[
\mathrm{e}^{\mathrm{i}g\left(y\right)}=G\left(y\right)\mathrm{e}^{\mathrm{i}p\left(y^{\frac{1}{d}}\right)},
\]
where $G\left(y\right)=\mathrm{e}^{\mathrm{i}g_0\left(y\right)}$ is a complex-valued subanalytic
function (see Definition \ref{def complex valued}), since $g_{0}$
is bounded. Moreover, in the notation of Equation (\ref{eq:loja}),
\[
\log g\left(y\right)=\log c+r\log y+h\left(y\right),
\]
where $h\left(y\right)=\log\left(1+H\left(y^{-\frac{1}{d}}\right)\right)$
is in $\mathcal{S}\left([y_{0},+\infty)\right)$ for some sufficiently large $y_{0}$.

Hence, it is easy to see that, if $f\in\mathcal{C}_{\text{naive}}^{\text{exp}}\left(\mathbb{R}\right)$,
then we may assume that, for $y$ sufficiently large,
\begin{equation}
f\left(y\right)=\sum_{j\in J}f_{j}\left(y\right)\mathrm{e}^{\mathrm{i}p_{j}\left(y^{\frac{1}{d}}\right)},\label{eq:first}
\end{equation}
where $J$ is a finite set, $f_{j}$ is a complex-valued constructible
function$,\ d\in\mathbb{N}$ and $\{ p_{j}\left(y\right):\ j\in J\} \subseteq\mathbb{R}[y]$
is a collection of distinct polynomials such that $p_{j}\left(0\right)=0$.

Moreover, there exists a finite set $K$ such that each $f_{j}$ is
of the form
\[
f_{j}\left(y\right)=\sum_{k\in K}h_{j,k}\left(y\right)\left(\log y\right)^{s_{j,k}},
\]
where $s_{j,k}\in\mathbb{N}$ and $h_{j,k}$ is a complex-valued subanalytic
function.

Let us prove that $f_j$ is indeed an absolutely convergent series. 

It is easy to see that Remark \ref{rem: lojas} also holds for complex-valued
subanalytic functions (where now $c\in\mathbb{C}$ and $H$ a convergent
power series with complex coefficients). Applying again Remark \ref{rem: lojas}
to each $h_{j,k}$, for $y$ sufficiently large we can write
\begin{equation}\label{eq:second}
f_{j}\left(y\right)=\sum_{k\in K}b_{j,k}y^{r_{j,k}}\left(\log y\right)^{s_{j,k}}\left(1+H_{j,k}\left(y^{-\frac{1}{d}}\right)\right),
\end{equation}
where $b_{j,k}\in\mathbb{C},\ r_{j,k}\in\mathbb{Q}$ is an integer
multiple of $\frac{1}{d}$ and $H_{j,k}\left(y\right)={ \sum_{m\in\mathbb{N}}a_{j,k,m}y^{m}}$
is an absolutely convergent power series with complex coefficients
and such that $H_{j,k}\left(0\right)=0$.  Hence, up to reorganizing the sum in Equation 
\eqref{eq:second}, we have proved Equation \eqref{eq: coeff converg}.

Now, setting $r_{j,k,m}=r_{j,k}-\frac{m}{d}$, we can write
\[
f\left(y\right)=\sum_{{\left(j,k\right)\in J\times K\atop m\in\mathbb{N}}}b_{j,k}a_{j,k,m}y^{r_{j,k,m}}\left(\log y\right)^{s_{j,k}}\mathrm{e}^{\mathrm{i}p_{j}\left(y^{\frac{1}{d}}\right)}.
\]
Let
\[
\begin{aligned}I=\{ \left(r,s\right)\in\mathbb{Q}\times\mathbb{N}:\ \exists j\in J,\exists k\in K,\exists m\in\mathbb{N}\text{\ s.t.}\ \left(r_{j,k,m},s_{j,k}\right)=\left(r,s\right)\} \ \text{and,}\\
\text{if\ }\left(r,s\right)\in I,\ \text{let}\ I_{\left(r,s\right)}=\{ \left(j,k,m\right)\in J\times K\times\mathbb{N}:\ \left(r_{j,k,m},s_{j,k}\right)=\left(r,s\right)\} .
\end{aligned}
\]
We can write
\[
f\left(y\right)=\sum_{\left(r,s\right)\in I}y^{r}\left(\log y\right)^{s}E_{\left(r,s\right)}\left(y\right),
\]
where $E_{\left(r,s\right)}\left(y\right)=\sum_{\left(j,k,m\right)\in I_{\left(r,s\right)}}b_{j,k}a_{j,k,m}\mathrm{e}^{\mathrm{i}p_{j}\left(y^{\frac{1}{d}}\right)}$.
Notice that for every $\left(r,s\right)\in I$ the set $I_{\left(r,s\right)}$
is finite, so $E_{\left(r,s\right)}$ is a finite sum of exponentials.
Moreover, if $I_{0}$ is the set of all $r\in\mathbb{Q}$ such that
there exists $s\in\mathbb{N}$ with $\left(r,s\right)\in I$, we have
that $I_{0}$ is bounded from above (by ${ \max_{\left(j,k\right)\in J\times K}}r_{j,k}$)
and for every $r\in I_{0}$ there are finitely many $s\in\mathbb{N}$
such that $\left(r,s\right)\in I$ (in fact, the cardinality of the
set of all $s$ such that there exists $r\in I_{0}$ with $\left(r,s\right)\in I$
is uniformly bounded by the product of the cardinalities of $J$ and
$K$). Hence, with respect to the lexicographic order, $I$ has the
same order type as $\omega$ and we can fix a decreasing bijection
$\mathbb{N}\ni n\mapsto\left(r_{n},s_{n}\right)=\left(r,s\right)\in I$.

Let us thus rename $E_{n}\left(y\right)=E_{\left(r_{n},s_{n}\right)}\left(y\right)$.
We have proved that, for $y$ sufficiently large,
\[
f\left(y\right)=\sum_{n\in\mathbb{N}}E_{n}\left(y\right)y^{r_{n}}\left(\log y\right)^{s_{n}}.
\]
In particular, $f$ has indeed a $\left(\mathcal{G},\mathcal{A}\right)$-asymptotic
expansion.
\end{proof}

\subsection{Two functions which are in $\mathcal{C}^{\text{exp}}\left(\mathbb{R}\right)$
but not in $\mathcal{C}_{\text{naive}}^{\text{exp}}\left(\mathbb{R}\right)$\label{sub:Two-functions-which}}
\begin{example}
\label{exa: fourier transform of flat exp} Consider the function $f\left(y\right)=\mathrm{e}^{-|y|}$.
\end{example}
Consider the Fourier transform of $f$:
\[
\hat{f}\left(y\right)=\int_{\mathbb{R}}\mathrm{e}^{-2\pi\mathrm{i}xy}\mathrm{e}^{-|x|}\mathrm{d}x.
\]
It is well known that $\hat{f}$ is a semi-algebraic integrable function,
namely $\hat{f}\left(y\right)=\frac{2}{1+4\pi^{2}y^{2}}$ (see for
example \cite{gasquet-witomski:fourier-analysis})
Since we can compute $f$ as the inverse Fourier transform of $\hat{f}$,
and since $\hat{f}$ is semi-algebraic, we have that $f$ belongs
to the class $\mathcal{C}^{\text{exp}}\left(\mathbb{R}\right)$.

It follows from Remark \ref{rem: table under right composition with sa}
that if $g\in\mathcal{S}\left(\mathbb{R}\right)$, then $\mathrm{e}^{-|g\left(y\right)|}\in\mathcal{C}^{\text{exp}}\left(\mathbb{R}\right)$
(in particular, $\mathrm{e}^{-y^{2}}\in\mathcal{C}^{\text{exp}}\left(\mathbb{R}\right)$).
\begin{claim*}
\label{claim: flat exp not naive}The function $f\left(y\right)=\mathrm{e}^{-|y|}$
is not in $\mathcal{C}_{\text{naive}}^{\text{exp}}\left(\mathbb{R}\right)$.\end{claim*}
\begin{proof}
Suppose for a contradiction that $f\in\C_{\naive}^{\exp}\left(\RR\right)$. By
Proposition \ref{prop: esympt expansion of naive} we may write $f\left(y\right)$
as the sum of a convergent series

\[
f\left(y\right)=\sum_{n\in\NN}E_{n}\left(y\right)y^{r_{n}}\left(\log y\right)^{s_{n}}
\]
for all sufficiently large $y$. Since the germ of $f$ at $+\infty$
is nonzero, this series contains a nonzero term. Choose $n_{0}\in\NN$
least such that $E_{n_{0}}\left(y\right)$ is not identically $0$. Thus there
exists a constant $C>0$ such that

\[
\left|f\left(y\right)\lyxmathsym{\textendash}E_{n_{0}}\left(y\right)y^{r_{n_{0}}}\left(\log y\right)^{s_{n_{0}}}\right|\leq Cy^{r_{n_{0}+1}}\left(\log y\right)^{s_{n_{0}+1}}
\]
for all sufficiently large $y$. Since $\frac{f\left(y\right)}{y^{r_{n_{0}}}\left(\log y\right)^{s_{n_{0}}}}$
and $\frac{y^{r_{n_{0}+1}}\left(\log y\right)^{s_{n_{0}+1}}}{y^{r_{n_{0}}}\left(\log y\right)^{s_{n_{0}}}}$
both tend to $0$ as $y\to+\infty$, dividing both sides of this inequality
by $y^{r_{n_{0}}}\left(\log y\right)^{s_{n_{0}}}$ and letting $y$ tend to $+\infty$
gives $\lim_{y\to+\infty}E_{n_{0}}\left(y\right)=0$, which contradicts Proposition  \ref{prop:goal1}(\ref{enu:density pointwise}).\end{proof}
\begin{example}
\label{exa: Si}Consider the sine integral $\Si\colon [0,+\infty)\to\RR$,
which is defined by
\[
\Si\left(y\right)={\int_{0}^{y}}\frac{\sin\left(t\right)}{t}\mathrm{d}t={ \int_{0}^{y}}\frac{\mathrm{e}^{\mathrm{i}t}-\mathrm{e}^{-\mathrm{i}t}}{2\mathrm{i}t}\mathrm{d}t\quad.
\]
Clearly, $\text{Si}\in\mathcal{C}^{\exp}\left([0,+\infty)\right)$.\end{example}
\begin{claim*}
\label{claim: Si not naive}The function $\text{Si}\left(y\right)$
is not in $\mathcal{C}_{\text{naive}}^{\exp}\left([0,+\infty)\right)$.\end{claim*}
\begin{proof}
Recall the classical asymptotic formula (see \cite{abramowitz:handbook_mathematical_functions})
\[
\text{Si}\left(y\right)\sim\frac{\pi}{2}-\frac{\cos y}{y}\sum_{k\in\mathbb{N}}\left(-1\right)^{k}\frac{\left(2k\right)!}{y^{2k}}-\frac{\sin y}{y}\sum_{k\in\mathbb{N}}\left(-1\right)^{k}\frac{\left(2k+1\right)!}{y^{2k+1}}.
\]
Hence, $\text{Si}\left(y\right)$ has a $\left(\mathcal{G},\mathcal{A}\right)$-asymptotic
expansion, with $\mathcal{G}$ and $\mathcal{A}$ as in the statement
of Proposition \ref{prop: esympt expansion of naive}. However, in
the notation of Equation (\ref{eq: coeff converg}), the series $F_{1}\left(y\right)=\sum_{k\in\mathbb{N}}\left(-1\right)^{k}\frac{\left(2k\right)!}{y^{2k+1}}$
and $F_{2}\left(y\right)=\sum_{k\in\mathbb{N}}\left(-1\right)^{k}\frac{\left(2k+1\right)!}{y^{2k+2}}$
are divergent. Therefore, by Lemma \ref{lem:uniqueness asympt exps}
and Proposition \ref{prop: esympt expansion of naive} $\text{Si}\left(y\right)\notin\mathcal{C}_{\text{naive}}^{\exp}\left([0,+\infty)\right)$.
\end{proof}

\subsection{Pointwise limits\label{sub:Pointwise-limits}}
\begin{defn}
\label{def: limits}For any $X\subseteq\RR^{m}$ and $f\colon X\times\RR\to\CC$,
let
\[
\text{Lim}\left(f,X\right):=\{x\in X:\ \text{\ensuremath{\lim_{y\to+\infty}f\left(x,y\right)\ }exists}\}.
\]
\end{defn}
\begin{prop}
\label{prop:limits} Let $f\in\C^{\exp}\left(X\times\RR\right)$ for some subanalytic
set $X\subseteq\mathbb{R}^{m}$. There exist $g,h\in\C^{\exp}\left(X\right)$
such that
\[
\mathrm{Lim}\left(f,X\right)=\{x\in X:\ h\left(x\right)=0\}
\]
and such that for all $x\in\mathrm{Lim}\left(f,X\right)$,
\[
\lim_{y\to+\infty}f\left(x,y\right)=g\left(x\right).
\]
\end{prop}
\begin{proof}
Apply Theorem \ref{thm:prepExpConstr} to $f\left(x,y\right)$ with respect to
$y$. Focus on one cell of the form
\[
A=\{\left(x,y\right):\ x\in\Pi_{m}\left(A\right),\ y>a\left(x\right)\}.
\]

Let
\begin{equation}
\begin{aligned}E=\{ \left(r,s\right)\in\mathbb{Q}\times\mathbb{N}:\ \exists j\in J\ \left(r_{j},s_{j}\right)=\left(r,s\right)\} \ \text{and,}\\
\text{if\ }\left(r,s\right)\in E,\ \text{let}\ E_{\left(r,s\right)}=\{ j\in J:\ \left(r_{j},s_{j}\right)=\left(r,s\right)\} .
\end{aligned}
\label{eq:E and Ers}
\end{equation}
The terms in the preparation involving $y^{r}\left(\log y\right)^{s}$ with $r<0$
may be neglected since they affect neither the existence of $\lim_{y\to+\infty}f\left(x,y\right)$
nor its value when it exists. So we may assume that $f\left(x,y\right)$ is naive
in $y$ with nonnegative powers of $y$ in each term of the preparation.
Write $f$ as the finite sum
\begin{equation}
f\left(x,y\right)=\sum_{\left(r,s\right)\in E}y^{r}\left(\log y\right)^{s}\left(\sum_{j\in E_{\left(r,s\right)}}f_{j}\left(x\right)\mathrm{e}^{\mathrm{i}\phi_{j}\left(x,y\right)}\right),\label{eq:fPrep}
\end{equation}
where each $f_{j}$ is in $\C^{\exp}\left(\Pi_{m}\left(A\right)\right)$, and where we have
that for each $\left(r,s\right)$ and each $x\in\Pi_{m}\left(A\right)$,
\[
\phi_{j}\left(x,y\right)=\sum_{k=1}^{m}a_{j,k}\left(x\right)y^{\frac{k}{d}},
\quad\text{for \ensuremath{j\in E_{\left(r,s\right)}},}
\]
is a family of distinct polynomials in $y^{\frac{1}{d}}$ with subanalytic
coefficients $a_{j,k}$. By partitioning in $x$ we may also assume
that if there exist $\widetilde{j}\in E_{\left(r,s\right)}$ and $\widetilde{x}\in\Pi_{m}\left(A\right)$
such that $\phi_{\widetilde{j}}\left(\widetilde{x},y\right)=0$ for all $y$ such that
$\left(\widetilde{x},y\right)\in A$, then $\phi_{\widetilde{j}}\left(x,y\right)=0$
for all $\left(x,y\right)\in A$ (note that there is at most one such $\widetilde{j}\in E_{\left(r,s\right)}$
such that $\phi_{\widetilde{j}}\equiv0$ because for each $x\in\Pi_{m}\left(A\right)$,
$\left(\phi_{j}\left(x,y\right)\right)_{j\in J}$ is a family of distinct polynomials
in $y^{\frac{1}{d}}$). \\

\noindent \emph{Claim.} For each $x\in\Pi_{m}\left(A\right)$, $x\in\text{Lim}\left(f,\Pi_{m}\left(A\right)\right)$
if and only if the following two conditions hold:
\begin{enumerate}
\item For each $\left(r,s\right)\in E$ such that $r>0$ or $s>0$, we have that $f_{j}\left(x\right)=0$
for all $j\in E_{\left(r,s\right)}$.
\item For all $j\in E_{\left(0,0\right)}$ such that $\phi_{j}\not\equiv0$,
we have $f_{j}\left(x\right)=0$.
\end{enumerate}

To prove the claim, fix $x\in\Pi_{m}\left(A\right)$. Observe that if Conditions
1 and 2 hold, then either $f$ is identically $0$, or else there
exists $j_{0}\in E_{\left(0,0\right)}$ such that $f\left(x,y\right)=f_{j_{0}}\left(x\right)$
for all $y$. Either way, $\lim_{y\to+\infty}f\left(x,y\right)$ exists trivially.

To prove the converse, assume that $x\in\text{Lim}\left(f,\Pi_{m}\left(A\right)\right)$.
Conditions 1 and 2 clearly hold if $f_{j}\left(x\right)=0$ for all $j\in{ \bigcup_{\left(r,s\right)\in E}}E_{\left(r,s\right)}$,
so assume otherwise. Choose $\left(r_{0},s_{0}\right)$ maximal with respect
to the lexicographical ordering such that $f_{j}\left(x\right)\neq0$ for some
$j\in E_{\left(r_{0},s_{0}\right)}$. 
 By Proposition \ref{prop:goal1}(\ref{enu:density pointwise}), since $\lim_{y\to+\infty}f\left(x,y\right)$
exists, it follows that $r_{0}=s_{0}=0$.
Thus Condition 1 holds, and we have
\[
f\left(x,y\right)=\sum_{j\in E_{\left(0,0\right)}}f_{j}\left(x\right)\mathrm{e}^{\mathrm{i}\phi_{j}\left(x,y\right)}
\]
for all $y$. Proposition \ref{prop:goal1}(\ref{enu: existence of limit})
now shows that Condition 2 holds. This proves the claim.\medskip{}

The claim easily implies the proposition. Indeed, define
\[
h=\left(\sum_{{\text{\ensuremath{\left(r,s\right)\in E}\ s.t.}\atop \text{\ensuremath{r>0}\ or \ensuremath{s>0}}}}\sum_{j\in E_{\left(r,s\right)}}\left|f_{j}\right|^{2}\right)+\left(\sum_{{\text{\ensuremath{j\in E_{\left(0,0\right)}\ }s.t.}\atop \phi_{j}\not\equiv0}}\left|f_{j}\right|^{2}\right);
\]
define $g=f_{j_{0}}$ if there exists $j_{0}\in E_{\left(0,0\right)}$
such that $\phi_{j_{0}}\equiv0$, and define $g=0$ otherwise. Then
$g,h\in\C^{\exp}\left(\Pi_{m}\left(A\right)\right)$. The claim shows that
\[
\text{Lim}\left(f\restriction{A},\Pi_{m}\left(A\right)\right)=\{x\in A:\ h\left(x\right)=0\}
\]
and that
\[
f\left(x,y\right)=g\left(x\right)\quad\text{for all \ensuremath{\left(x,y\right)\in A}\ such that \ensuremath{h\left(x\right)=0}.}
\]

\end{proof}

\section{
 Parametric $L^{p}$-completeness and the 
Fourier-Plancherel transform \label{sub:-completeness-and-the}}
%

In this section we prove a parametric $L^p$-completeness theorem for $\C^{\exp}$ and use this to show that $\C^{\exp}$ is closed under the Fourier-Plancherel transform.

\begin{defn}
\label{def:Cauchy} Let $X\subseteq\RR^{m}$ and $f\colon X\times\RR\to\CC$
be Lebesgue measurable, and $p\in[1,+\infty]$. For each $y\in\RR$,
define $f_{y}:X\to\CC$ by $f_{y}\left(x\right)=f\left(x,y\right)$ for all $x\in X$. We
say that the family of functions $\left(f_{y}\right)_{y\in\RR}$ is 
\textbf{\emph{Cauchy in $L^{p}\left(X\right)$ as $y\to+\infty$}}
if $\left(f_{y}\right)_{y\in\RR}\subseteq L^{p}\left(X\right)$
and for all $\varepsilon>0$ there exists $y_{0}\in\RR$ such that
\[
\|f_{y}-f_{y'}\|_{p}<\varepsilon\quad\text{for all \ensuremath{y,y'\geq y_{0}}.}
\]
\end{defn}
%
\begin{prop}
\label{prop:complete}
 Let $p\in[1,+\infty]$ and $f\in\C^{\exp}\left(X\times\RR\right)$
for a subanalytic set $X\subseteq\RR^{m}$, and suppose that $\left(f_{y}\right)_{y\in\RR}$
is Cauchy in $L^{p}\left(X\right)$ as $y\to+\infty$. Then there exist $g\in\C^{\exp}\left(X\right)\cap L^{p}\left(X\right)$
and a subanalytic set $X_{0}\subseteq X$ such that $\vol_{m}\left(X\setminus X_{0}\right)=0$,
\[
\lim_{y\to+\infty}\|f_{y}-g\|_{p}=0,
\]
and
\[
\lim_{y\to+\infty}f\left(x,y\right)=g\left(x\right)\quad\text{for all \ensuremath{x\in X_{0}}.}
\]
\end{prop}

Before proving Proposition \ref{prop:complete}, we use it to show that $\C^{\exp}$ is closed under the Fourier-Plancherel transform.

\begin{thm}
\label{cor:Plancherel} Let $\widetilde{\hbox{\calli F}\ }$ be the
Fourier-Plancherel extension of the Fourier transform to $L^{2}\left(\mathbb{R}^{n}\right)$,
as in (\ref{eq: F tilda}). Then, the image of $\C^{\exp}\left(\RR^{n}\right)\cap L^{2}\left(\RR^{n}\right)$
under $\widetilde{\hbox{\calli F}\ }$ is $\C^{\exp}\left(\RR^{n}\right)\cap L^{2}\left(\RR^{n}\right)$. \end{thm}
\begin{proof}
Let $f\in\C^{\exp}\left(\RR^{n}\right)\cap L^{2}\left(\RR^{n}\right)$. We use coordinates
$x=\left(x_{1},\ldots,x_{n}\right)$ and $t=\left(t_{1},\ldots,t_{n}\right)$ on $\RR^{n}$.
For each $y\in\RR$, define
\[
B_{y}=\{t\in\RR^{n}:\ |t|\leq y\},
\]
and observe that $L^{2}\left(B_{y}\right)\subseteq L^{1}\left(B_{y}\right)$ for each $y$
(by the Cauchy-Schwartz inequality, since $\vol_{n}\left(B_{y}\right)<+\infty$).
So we may define $F\colon \RR^{n+1}\to\CC$ by
\begin{equation}
F\left(x,y\right):=\int_{B_{y}}f\left(t\right)\mathrm{e}^{-2\pi\mathrm{i}t\cdot x}\,\mathrm{d}t=\int_{\RR^{n}}\chi_{B_{y}}\left(t\right)f\left(t\right)\mathrm{e}^{-2\pi\mathrm{i}t\cdot x}\,\mathrm{d}t,\label{eq:Plancherel}
\end{equation}
and we have that $F\in\C^{\exp}\left(\RR^{n+1}\right)$ since $\C^{\exp}$ is
closed under integration. The extended Fourier transform $\widetilde{\hbox{\calli F}}\left(f\right)$
is the equivalence class of functions $\left[\widehat{f}\right]$ (with respect
to almost everywhere equivalence) that is defined by the condition
\[
\lim_{y\to+\infty}\|\widehat{f}-F_{y}\|_{2}=0.
\]
Thus $\left(F_{y}\right)_{y\in\RR}$ is Cauchy in $L^{2}\left(\RR^{n}\right)$ as $y\to+\infty$,
so by Proposition \ref{prop:complete} we may fix $g\in\C^{\exp}\left(\RR^{n}\right)$
such that
\[
\lim_{y\to+\infty}\|g-F_{y}\|_{2}=0,
\]
and hence $\left[\widehat{f}\right]=[g]$. This shows that the extended Fourier
transform $\widetilde{\hbox{\calli F}\ }$ maps $\C^{\exp}\left(\RR^{n}\right)\cap L^{2}\left(\RR^{n}\right)$
into $\C^{\exp}\left(\RR^{n}\right)\cap L^{2}\left(\RR^{n}\right)$. A completely symmetric
argument, where one simply replaces $\mathrm{i}$ with $-\mathrm{i}$
in (\ref{eq:Plancherel}), shows that the inverse extended Fourier
transform maps $\C^{\exp}\left(\RR^{n}\right)\cap L^{2}\left(\RR^{n}\right)$ into $\C^{\exp}\left(\RR^{n}\right)\cap L^{2}\left(\RR^{n}\right)$
as well, so $\C^{\exp}\left(\RR^{n}\right)\cap L^{2}\left(\RR^{n}\right)$ is in fact the
image of $\C^{\exp}\left(\RR^{n}\right)\cap L^{2}\left(\RR^{n}\right)$ under $\widetilde{\hbox{\calli F}\ }$.
\end{proof}

The remainder of the section is devoted to the proof of Proposition \ref{prop:complete}, which requires us to develop a bit of machinery.  This proof is somewhat similar to the proof of Proposition \ref{prop:limits}, except we cannot rely on the facts about c.u.d.\ mod $1$ maps quoted in Remark \ref{rem Polynomial maps are CUD}.  Instead, we need to adapt these facts to parametric families of maps that are c.u.d.\ mod $1$ in a certain uniform sense.  We are not aware of a reference in the literature on c.u.d. mod $1$ maps that considers this parametric case, so this section develops this material from scratch.  
We remark that the proofs of Lemma \ref{lemma:Weyl} and Proposition \ref{prop:WeylToCUD} below use ideas found in the proofs of the following content in \cite{Kui}: Example 9.2 and the closely interrelated Theorems 1.1, 2.1, 6.1, 6.2, 9.1, 9.2, and 9.9.

%
Let us first give the parametric version of Definition \ref{def CUD map}. For this, let $X$ be a nonempty set and  
  $\psi=\left(\psi_1, \ldots, \psi_n\right): X\times [0,+\infty)\to \mathbb{R}^n$ be a map. 
If $I_1, \ldots, I_n\subseteq \RR$ are bounded intervals with nonempty interior, we denote by 
$I$ the box $\prod_{j=1}^\ell I_j$  
and, for $T\ge 0$ and $x\in X$, we let
$$ W^x_{\psi,I,T} : = \{ t\in [0,T] : \{ \psi\left(x,t\right) \}\in I \}, $$
where $\{\psi\left(x,t\right)\}$ denotes the vector of fractional parts $\left(\{ \psi_1\left(x,t\right)\},\ldots, \{\psi_n\left(x,t\right)\}\right)$ of the components of $\psi$.

%
\begin{defn}\label{defn:CUD:param}
With this notation, we say that the map $\psi$ is   \textbf{\emph{continuously uniformly distributed modulo $1$ on $X$}} (abbreviated as \textbf{\emph{c.u.d.\ mod $1$ on $X$}}) if for every box $I \subseteq [0,1)^n$,
\[
\lim_{T\to+\infty} \sup_{x\in X}
\frac{\vol_1\left(W^x_{\psi,I,T}\right)}{T} = \vol_n\left(I\right).
\]
\end{defn}

The following  Remark is  the parametric analogue of Lemma \ref{lem:main-lemma-appendix}.
\begin{rem}\label{rem:Param:geometricIntervals}
Suppose that $\psi:X\times[0,+\infty)\to\RR^n$ is c.u.d.\ mod $1$ on $X$.  Then for each box $I \subseteq [0,1)^n$, there exists $k_0 \in \NN$ such that for all $k\geq k_0$, for all $x\in X$,
\[
\vol_1\left(  \{t\in[2^k, 2^{k+1}) :   \{ \psi\left(x,t\right)\}\in I\}  \right) 
\geq 2^{k-1}\vol_n\left(I\right).
\]
This bound is proven just as Lemma \ref{lem:main-lemma-appendix}, using the uniform limit 
in the parameter $x$ provided by Definition \ref{defn:CUD:param}.
\end{rem}

%
%
The following technical lemma will be used in the proof of the forthcoming Proposition 
\ref{prop:WeylToCUD}.

\begin{lem}\label{lemma:Weyl}
Define $\phi:X\times [0,+\infty)\to\RR$ by
\[
\phi\left(x,t\right) = \sum_{j=0}^{d} \phi_j\left(x\right) t^j,
\]
where $d$ is a positive integer, the functions $\phi_0,\ldots,\phi_d:X\to\RR$ are bounded, and there exists $\varepsilon > 0$ such that $|\phi_d\left(x\right)| > \varepsilon$ for all $x\in X$.  Then the function $\Phi:X\times[0,+\infty)\to\CC$ defined by
\[
\Phi\left(x,T\right) = \int_{0}^{T} \mathrm{e}^{\mathrm{i}\phi\left(x,t\right)}\,\mathrm{d}t
\]
is bounded.
\end{lem}
%
%
\begin{proof}
It suffices to show that for some suitable choice of $T_0\geq 0$ there exists a constant $C > 0$ such that 
for all $\left(x,T\right)\in X\times[T_0,+\infty)$
\[
\left|\int_{T_0}^{T} \mathrm{e}^{\mathrm{i} \phi\left(x,t\right)}\,\mathrm{d}t \right| \leq C.
\]
 Define
\[
f\left(x,t\right) := \frac{\phi\left(x,t\right)}{\phi_d\left(x\right)} = t^d + \sum_{j=0}^{d-1} \frac{\phi_j\left(x\right)}{\phi_d\left(x\right)} t^j,
\]
and observe that our assumed bounds on $\phi_0,\ldots,\phi_n$ show that the coefficient functions in $x$ of the polynomial $f\left(x,t\right)$ are bounded.  Therefore by computing $ \PD{}{f}{t}$ and $ \PD{2}{f}{t}$ and factoring out their leading terms, we may fix $T_0 > 0$ such that
\begin{equation}\label{eq:derivLT}
\PD{}{f}{t}\left(x,t\right) = dt^{d-1} u\left(x,t\right)
\quad\text{and}\quad
\PD{2}{f}{t}\left(x,t\right) = \begin{cases}
0,
    & \text{if $d=1$,} \\
d\left(d-1\right)t^{d-2}v\left(x,t\right),
    & \text{if $d > 1$,}
\end{cases}
\end{equation}
for some functions $u\left(x,t\right)$ and $v\left(x,t\right)$ (when $d>1$) that take values in $ [\frac{1}{2},\frac{3}{2}]$ for all $\left(x,t\right)\in X\times [T_0,+\infty)$.  Therefore $ \PD{}{f}{t} > 0$ and $ \PD{2}{f}{t} \geq 0$ on $X\times[T_0,+\infty)$, so for each $x\in X$, the functions $t\mapsto f\left(x,t\right)$ and $t\mapsto  \PD{}{f}{t}\left(x,t\right)$ are respectively strictly increasing and monotonically increasing on $[T_0,+\infty)$.  For each $x\in X$, let $t = g\left(x,s\right)$ be the inverse of $s = f\left(x,t\right)$, where $t\geq T_0$ and $s\geq f\left(x,T_0\right)$.  For each $T \geq T_0$, we can perform the integral substitution
\[
s = f\left(x,t\right), \quad \mathrm{d}s = \PD{}{f}{t}\left(x,t\right)\,\mathrm{d}t = \PD{}{f}{t}\left(x,g\left(x,s\right)\right)\,\mathrm{d}s
\]
to write
\begin{align}\label{eq:integralSub}
\int_{T_0}^{T} \mathrm{e}^{\mathrm{i}\phi\left(x,t\right)}\,\mathrm{d}t
    &= \int_{T_0}^{T} \mathrm{e}^{\mathrm{i}\phi_d\left(x\right)f\left(x,t\right)}\,\mathrm{d}t \nonumber \\
    &= \int_{f\left(x,T_0\right)}^{f\left(x,T\right)} \frac{\mathrm{e}^{\mathrm{i}\phi_d\left(x\right)s}}{\PD{}{f}{t}\left(x,g\left(x,s\right)\right)}\,\mathrm{d}s.
\end{align}
The function
\[
s\mapsto \frac{1}{\PD{}{f}{t}\left(x,g\left(x,s\right)\right)}
\]
is monotonically decreasing on $[f\left(x,T_0\right),+\infty)$, so we can apply the second mean value theorem for integrals to the real and complex parts of the integral \eqref{eq:integralSub}.  For the real part, this gives
\begin{equation}\label{eq:realPart}
\int_{f\left(x,T_0\right)}^{f\left(x,T\right)} \frac{\cos\left(\phi_d\left(x\right)s\right)}{\PD{}{f}{t}\left(x,g\left(x,s\right)\right)}\,\mathrm{d}s
=
\frac{1}{\PD{}{f}{t}\left(x,T_0\right)} \int_{f\left(x,T_0\right)}^{\xi\left(x,T\right)} \cos\left(\phi_d\left(x\right)s\right)\,\mathrm{d}s
\end{equation}
for some $\xi\left(x,T\right)\in\left(f\left(x,T_0\right),f\left(x,T\right)\right)$.  Since $s\mapsto \cos\left(\phi_d\left(x\right)s\right)$ has an antiderivative with period $\frac{2\pi}{\left|\phi_d\left(x\right)\right|}$, and since $\frac{2\pi}{\left|\phi_d\left(x\right)\right|}\leq \frac{2\pi}{\varepsilon}$, the integral in the right side of \eqref{eq:realPart} may be replaced with an integral over an interval of length at most $\frac{2\pi}{\varepsilon}$.  This, along with the form of $\PD{}{f}{t}$ given in \eqref{eq:derivLT}, shows that \eqref{eq:realPart} is bounded. A nearly identical calculation shows the same for the imaginary part of \eqref{eq:integralSub}, and the lemma follows.
\end{proof}
The following Proposition \ref{prop:WeylToCUD} is the parametric analogue of Remark \ref{rem Polynomial maps are CUD}, stating that polynomials maps  are c.u.d. mod $1$
when nontrivial $\ZZ$-linear combinations of their components are nonconstant. For technical reasons, in the parametric case it is more convenient to reduce to the situation of maps with monomial instead of polynomial components.
%
\begin{prop}\label{prop:WeylToCUD}
Consider a map $\psi = \left(\psi_1,\ldots,\psi_n\right):X\times[0,+\infty)\to\RR^n$, where $X$ is a compact topological space and where for each $j\in\{1,\ldots,n\}$,
\[
\psi_j\left(x,t\right) = g_j\left(x\right) t^{\gamma_j}
\]
for some continuous function $g_j:X\to\RR$ and positive integer $\gamma_j$.  Assume that for each $x\in X$, the functions $t\mapsto \psi_1\left(x,t\right),\ldots,t\mapsto\psi_n\left(x,t\right)$ are linearly independent over $\QQ$.  Then $\psi$ is c.u.d.\ mod $1$ on $X$.
\end{prop}
The following notation and observation will be used in the proof of Proposition \ref{prop:WeylToCUD}. 
%
\begin{rem}\label{rem:independence through J_k}
Let
$
\gamma = \max\{\gamma_1,\ldots,\gamma_n\},$
and for each $k\in\{1,\ldots,\gamma\}$, let
$J_k = \{j\in\{1,\ldots,n\} : \gamma_j = k\}.$
The assumption that $t\mapsto \psi_1\left(x,t\right),\ldots,t\mapsto\psi_n\left(x,t\right)$ are linearly independent over $\QQ$ for each $x\in X$ is equivalent to saying that for each $k\in\{1,\ldots,\gamma\}$ and $x\in X$, the family of real numbers $\left(g_j\left(x\right)\right)_{j\in J_k}$ is linearly independent over $\QQ$.
\end{rem}
%
%
\begin{proof}[Proof of Proposition  \ref{prop:WeylToCUD}]
We shall use the variables $t$, $y = \left(y_1,\ldots,y_n\right)$, and $z = \left(z_1,\ldots,z_n\right)$, and write $\mathrm{d}y $ for $ \mathrm{d}y_1\wedge\ldots\wedge \mathrm{d}y_n$.  Let $\varepsilon > 0$ and a box $I = \prod_{j=1}^{n}I_j \subseteq [0,1)^n$ be given.  For each $j\in\{1,\ldots,n\}$, let $\chi_{I_j}:\RR\to\{0,1\}$ be the $1$-periodic extension of the characteristic function of $I_j$ in $[0,1)$, and define $\chi_I :\RR^n\to\{0,1\}$ by $\chi_I\left(y\right) = \prod_{j=1}^{n}\chi_{I_j}\left(y_j\right)$.  Thus
$$
\vol_1\left(\{t\in[0,T] : \{\psi\left(x,t\right)\}\in I\}\right) = \int_{0}^{T}\chi_I\circ\psi\left(t\right) \,\mathrm{d}t.
$$
Let $\varepsilon>0$ and fix $\delta\in (0,1]^n$ sufficiently small so that $ 1 - \left(1-\delta\right)^n < \frac{\varepsilon}{4}$.  For each $j\in\{1,\ldots,n\}$, fix $1$-periodic continuous functions $p_j:\RR\to[0,1]$ and $q_j:\RR\to[0,1]$ such that $p_j\left(t\right) \leq \chi_{I_j}\left(t\right) \leq q_j\left(t\right)$ for all $t\in\RR$ and such that
\[
\vol_1\left(\{t\in[0,1] : p_j\left(t\right) \neq \chi_{I_j}\left(t\right)\}\right) \leq \delta
\quad\text{and}\quad
\vol_1\left(\{t\in[0,1] : q_j\left(t\right) \neq \chi_{I_j}\left(t\right)\}\right) \leq \delta.
\]
Define $p:\RR^n\to[0,1]$ and $q:\RR^n\to[0,1]$ by $p\left(y\right) = \prod_{j=1}^{n}p_j\left(y_j\right)$ and $q\left(y\right) = \prod_{j=1}^{n}q_j\left(y_j\right)$.  Since $p\left(y\right) \leq \chi_I\left(y\right)\leq q\left(y\right)$ for all $y\in\RR^n$, we have, for all $x\in X$,
\begin{equation}\label{eq:squeeze}
\frac{1}{T}\int_{0}^{T}p\circ\psi\left(x,t\right)\,\mathrm{d}t
\leq
\frac{1}{T}\int_{0}^{T}\chi_I\circ\psi\left(x,t\right)\,\mathrm{d}t
\leq
\frac{1}{T}\int_{0}^{T}q\circ\psi\left(x,t\right)\,\mathrm{d}t.
\end{equation}
It now suffices to show that there exists $T_0 > 0$ such that the lower and upper bounds in \eqref{eq:squeeze} are within $\varepsilon$ of $\vol_n\left(I\right)$ for all $x\in X$ and $T\geq T_0$.  The computations involving the lower bound and the upper bound are identical, so we only show the computation with the lower bound.

Fix $\eta\in(0,1]^n$ sufficiently small so that for all $y,z\in [-2,2]^n$, if $|y_j - z_j| < \eta$ for all $j\in\{1,\ldots,n\}$, then $\left|\prod_{j=1}^{n} y_j - \prod_{j=1}^{n} z_j\right| < \frac{\varepsilon}{4}$.  By a Weierstrass approximation theorem, for each $j\in\{1,\ldots,n\}$ we may fix a trigonometric polynomial
\[
P_j\left(t\right) = \sum_{\alpha=-N_j}^{N_j} c_{j,\alpha} \mathrm{e}^{2\pi \mathrm{i} \alpha t}
\]
(where $N_j\in\NN$ and $c_{j,\alpha}\in\CC$ for each $\alpha$) such that
\begin{equation}\label{eq:trigApprox}
\left|p_j\left(t\right) - P_j\left(t\right)\right| \leq \eta
\end{equation}
for all $t\in\RR$.  Define $P:\RR^n\to\CC$ by $P\left(y\right) = \prod_{j=1}^{n}P_j\left(y_j\right)$.  Since
$
\vol_n\left(I\right) = \int_{[0,1]^n}\chi_I\left(y\right)\,\mathrm{d}y,$
we have
\begin{align}\label{eq:integralTriIneq}
\left|\frac{1}{T}\int_0^{T} p\circ\psi (x,t)\,\mathrm{d}t - \vol_n(I)\right|
&\leq
\left|\frac{1}{T}\int_{0}^{T}   \left( p\circ\psi(x,t) - P\circ\psi(x,t) \right)  \,\mathrm{d}t\right|
& \text{(*)}\\
&\quad
+
\left|\frac{1}{T}\int_{0}^{T}P\circ\psi(x,t)\,\mathrm{d}t - \int_{[0,1]^n}\!\!\! P(y)\,\mathrm{d}y\right|
& \text{(**)} \nonumber\\
&\quad
+
\left|\int_{[0,1]^n}\!\!\! \left( P\left(y\right)-p\left(y\right) \right)\,\mathrm{d}y\right|
& \text{(***)} \nonumber \\
&\quad
+
\left|\int_{[0,1]^n}\!\!\! \left( p\left(y\right)-\chi_I\left(y\right) \right) \,\mathrm{d}y \right|.
& \text{(****)} \nonumber 
\end{align}
Note that $|p_j\left(t\right)| \leq 1$ and $|P_j\left(t\right)|\leq |p_j\left(t\right)| + |P_j\left(t\right)-p_j\left(t\right)| \leq 1+\eta \leq 2$ for all $j\in\{1,\ldots,n\}$ and $t\in\RR$, so by our choice of $\eta$, \eqref{eq:trigApprox} implies that $ |p\left(y\right)-P\left(y\right)| \leq \frac{\varepsilon}{4}$ for all $y\in\RR^n$.  Therefore the terms (*) and (***) in \eqref{eq:integralTriIneq} are both bounded above by $  \frac{\varepsilon}{4}$.  And since
\[
\bigcap_{j=1}^{n}\{y\in[0,1]^n : p_j\left(y_j\right) =\chi_{I_j}\left(y_j\right)\} \subseteq \{y\in[0,1]^n : p\left(y\right) = \chi_I\left(y\right)\}
\]
and
\[
\vol_n\left(\bigcap_{j=1}^{n}\{y\in[0,1]^n : p_j\left(y_j\right) = \chi_{I_j}\left(y_j\right)\}\right) \geq \left(1-\delta\right)^n,
\]
it follows that
\[
\vol_n\left(\{y\in[0,1]^n : p\left(y\right) \neq \chi_I\left(y\right)\}\right) \leq 1 - \left(1-\delta\right)^n \leq \frac{\varepsilon}{4},
\]
so the term (****) in \eqref{eq:integralTriIneq} is also bounded above by $\frac{\varepsilon}{4}$.  So to finish, we need to show that there exists $T_0 > 0$ such that the term (**) in \eqref{eq:integralTriIneq} is also  bounded above by $ \frac{\varepsilon}{4}$ for all $x\in X$ and all $T\geq T_0$.

We have
\[
P\left(y\right)
= \prod_{j=1}^{n}\left(\sum_{\alpha_j = -N_j}^{N_j} c_{j,\alpha_j} \mathrm{e}^{2\pi \mathrm{i} \alpha_jy_j}\right)
= \sum_{\alpha\in \ZZ^n\cap[-N,N]} c_\alpha \mathrm{e}^{2\pi \mathrm{i} \alpha\cdot y},
\]
where $N = \left(N_1,\ldots,N_n\right)$,
$[-N,N]=\prod_{j=1}^n [-N_j,N_j]$,
 $c_\alpha = \prod_{j=1}^{n} c_{j,\alpha_j}$ for $\alpha = \left(\alpha_1,\ldots,\alpha_n\right)$, and $\alpha\cdot y = \sum_{j=1}^{n} \alpha_jy_j$.  Thus
\[
P\circ\psi\left(x,t\right) = \sum_{\alpha\in\ZZ^n\cap [-N,N]} c_\alpha \mathrm{e}^{2\pi \mathrm{i} \alpha\cdot \psi\left(x,t\right)}.
\]
Observe that
\[
\frac{1}{T}\int_{0}^{T}c_0\,\mathrm{d}t - \int_{[0,1]^n} c_0 ds = c_0-c_0 = 0
\
\textrm{ and that }
\
\int_{[0,1]^n} c_\alpha \mathrm{e}^{2\pi \mathrm{i} \alpha \cdot s} \,\mathrm{d}s = 0
\]
for all nonzero $\alpha\in\ZZ^n\cap[-N,N]$.  Therefore the term (**) equals
\[
\left|\sum_{\alpha\in\left(\ZZ^n\setminus\{0\}\right)\cap [-N,N]} \frac{c_\alpha}{T} \int_{0}^{T}\mathrm{e}^{2\pi \mathrm{i} \alpha\cdot \psi\left(x,t\right)}\,\mathrm{d}t\right|.
\]
Using the notation $J_k$ from Remark \ref{rem:independence through J_k}, for each nonzero $\alpha\in\ZZ^n\cap[-N,N]$ let
\begin{align*}
\phi_{\alpha,k}\left(x\right)
    &= \sum_{j\in J_k} \alpha_j g_j\left(x\right) \quad\text{for each $k\in\{1,\ldots,\gamma\}$,}\\
d\left(\alpha\right)
    &= \max\{k \in \{1,\ldots,\gamma\} : \text{$\alpha_j \neq 0$ for some $j\in J_k$}\},
\end{align*}
and observe that
\[
\alpha\cdot \psi\left(x,t\right) = \sum_{k=1}^{d\left(\alpha\right)} \phi_{\alpha,k}\left(x\right)t^k.
\]
The set $X$ is compact, the functions $\phi_{\alpha,1},\ldots,\phi_{\alpha,d\left(\alpha\right)}$ are continuous on $X$, and by Remark 
\ref{rem:independence through J_k}, $\phi_{\alpha,d\left(\alpha\right)}$ has no zero in $X$ because $t\mapsto \psi_1\left(x,t\right),\ldots,t\mapsto \psi_n\left(x,t\right)$ are linearly independent over $\QQ$ for each $x\in X$.  Therefore $|\phi_{\alpha,1}|,\ldots,|\phi_{\alpha,d\left(\alpha\right)}|$ are bounded above, and $|\phi_{\alpha,d\left(\alpha\right)}|$ is bounded below by a positive constant.  We may apply Lemma \ref{lemma:Weyl} to fix $T_0 > 0$ such that for all $x\in X$ and $T\geq T_0$, 
the term (**) is bounded above by $ \frac{\varepsilon}{4}$.
\end{proof}

Let us fix the notation in view of Lemma \ref{lem:Param:OscPrep}.
For this consider a cell 
$$A = \{\left(x,t\right) : x\in\Pi_m\left(A\right), t > a\left(x\right)\},$$
 where $\Pi_m\left(A\right)$ is connected and open in $\RR^m$.  Define $f:A\to\CC$ by
\[
f\left(x,t\right) = \sum_{j\in J} f_j\left(x\right) \mathrm{e}^{\mathrm{i}\phi_j\left(x,t\right)},
\]
where $J$ is a nonempty finite index set,  $\left(f_j\right)_{j\in J}$ is a family of analytic functions in $\C^{\exp}\left(\Pi_m\left(A\right)\right)$, $\left(\phi_j\right)_{j\in J}$ is a family of distinct functions on $\Pi_m\left(A\right)\times\RR$ that are polynomials in $t$ with analytic coefficients in $\S\left(\Pi_m\left(A\right)\right)$, and $\phi_j\left(x,0\right) = 0$ for all $j\in J$ and $x\in\Pi_m\left(A\right)$. 


Lemma \ref{lem:Param:OscPrep} below is the parametric analogue of 
the presentation of the function $f\left(t\right)$ in Remark \ref{rem Z-free polynomials} as a nonconstant Laurent polynomial in 
$\mathrm{e}^{2\pi\mathrm{i} \widetilde{p}_1}, \ldots,\mathrm{e}^{2\pi\mathrm{i} \widetilde{p}_\ell} $
where the polynomial map $\left(\widetilde{p}_1\left(t\right), \ldots, \widetilde{p}_\ell\left(t\right)\right)$ is c.u.d. mod $1$. Here in the parametric case  it is technically more convenient 
to present $f\left(x,t\right)$ as a nonconstant Laurent polynomial (with coefficient functions in the parameter $x$) in $\mathrm{e}^{2\pi\mathrm{i}\psi_1\left(x,t\right)}, \ldots, \mathrm{e}^{2\pi\mathrm{i}\psi_n\left(x,t\right)}$, where the map
$\left(
\psi_1\left( x,t \right), \ldots, \psi_n\left(x,t\right)
\right)$ is a monomial map in $t$ that is c.u.d. mod 1 on certain compact sets of $\Pi_m\left(A\right)$.
%
%
\begin{lem}\label{lem:Param:OscPrep}
With the notation just fixed above, we may express $f$ as a composition
\[
f\left(x,t\right) = F\left(x,\psi\left(x,t\right)\right)
\]
on $A$, where for some $n\in\NN$, $\psi = \left(\psi_1,\ldots,\psi_n\right)$ is a monomial map in $t$ with analytic coefficient functions in $\S\left(\Pi_m\left(A\right)\right)$
and $F\left(x,z_1, \ldots, z_n\right)$ is 
a Laurent polynomial in the variables $\mathrm{e}^{2\pi \mathrm{i} z_1}, \ldots , \mathrm{e}^{2\pi\mathrm{i} z_n}$ 
with coefficients $f_j\left(x\right)$, $j\in J$. 
If $J$ is a singleton $\{j_0\}$ and if $\phi_{j_0} = 0$, then $n = 0$ and $F\left(x\right) = f_{j_0}\left(x\right)$.  
Otherwise we have
$n>0$ and 
\begin{enumerate}
\item\label{F par non nul} there exists a set $B\subseteq \Pi_m\left(A\right)$ such that 
$\vol_m\left(\Pi_m \left(A\right) \setminus B\right)=0$
and for any $x\in B$,  $z\mapsto F\left(x,z\right)$ is nonconstant, 
\item\label{Psi CUD}
for any open set $\Omega\subseteq \Pi_m\left(A\right)$ and any real number $\lambda < \vol_m\left(\Omega\right)$, there exists a real number $T_0$ and a compact set $K\subseteq \Omega\cap B$ 
such that $K\times [T_0,+\infty)\subseteq A$, $\vol_m\left(K\right)\ge \lambda$ and $\psi\restriction{ K\times [T_0,+\infty)}$ is c.u.d. mod 1 on $K$. 
\end{enumerate}
\end{lem}
%
%
\begin{proof}
Since the functions $\phi_j$, $j\in J$, are distinct, it is only possible to have $\phi_j=0$ for all $j\in J$ when $J$ is a singleton $\{j_0\}$, and in this case we have $f\left(x,t\right) = f_{j_0}\left(x\right)$.  We may now assume that $\phi_j\neq 0$ for some $j\in J$.  
In this case, since $\phi_j\left(x,0\right) = 0$ for all $j\in J$ and all $x\in \Pi_m\left(A\right)$,
\[
d := \max\{\deg \phi_j : j\in J\}
\]
is a positive integer.  For each $j\in J$, write
\[
\phi_j\left(x,t\right) = \sum_{k=1}^{d} \phi_{j,k}\left(x\right) t^k
\]
with $\phi_{j,k}\in \S\left(\Pi_m\left(A\right)\right)$.  
For each $k\in\{1,\ldots,d\}$, fix $\Gamma_k\subseteq J$ such that $\left(\phi_{\gamma,k}\right)_{\gamma\in\Gamma_k}$ is a basis over $\QQ$ of the span  over $\QQ$
of the family $\left(\phi_{j,k}\right)_{j\in J}$ (as functions of $x$), and let
\[
\Gamma = \{\left(\gamma,k\right) : k\in\{1,\ldots,d\}, \gamma\in\Gamma_k\}.
\]
We may fix a positive integer $\eta$ such that for each $\left(j,k\right)\in J\times\{1,\ldots,d\}$,
\[
\phi_{j,k} = \sum_{\gamma\in\Gamma_k} \frac{\alpha_{j;\gamma,k}}{\eta}\phi_{\gamma,k}
\]
for a unique tuple of integers $\left(\alpha_{j;\gamma,k}\right)_{\gamma\in\Gamma_k}$.  
With this notation we have
\begin{align*}
f\left(x,t\right)
    &= \sum_{j\in J} f_j\left(x\right) \mathrm{e}^{\mathrm{i} \sum_{k=1}^{d}\phi_{j,k}\left(x\right)t^k} = \sum_{j\in J} f_j\left(x\right) \mathrm{e}^{\mathrm{i}  \sum_{k=1}^{d}\sum_{\gamma\in\Gamma_k} \frac{\alpha_{j;\gamma,k}}{\eta}\phi_{\gamma,k}\left(x\right)t^k}
    \\
    &= \sum_{j\in J} f_j\left(x\right) \prod_{\left(\gamma,k\right)\in\Gamma}\left(\mathrm{e}^{2\pi \mathrm{i}  
    \psi_{\gamma,k}\left(x\right)}\right)^{\alpha_{j;\gamma,k}}= F\left(x, \left(\psi_{\gamma,k}\left(x\right)\right)_{\left(\gamma,k\right)\in \Gamma}\right).
\end{align*}
where for each $\left(\gamma,k\right)\in\Gamma$ 
$\psi_{\gamma,k}\left(x,t\right) = \frac{\phi_{\gamma,k}\left(x\right)t^k}{2\pi\eta}
$ and $$F\left(x,\left(z_{\gamma,k}\right)_{\left(\gamma,k\right)\in \Gamma}\right)= \sum_{j\in J} f_j\left(x\right) 
\prod_{\left(\gamma,k\right)\in\Gamma}\left(\mathrm{e}^{2\pi \mathrm{i}  
   z_{\gamma,k}}\right)^{\alpha_{j;\gamma,k}}.$$
%
  For each $j\in J$, $f_j$ is a nonzero analytic function on the connected and open set $\Pi_m\left(A\right)$, so the set
\[
U := \{x\in\Pi_m\left(A\right) : \text{$f_j(x)\neq 0$ for all $j\in J$}\}
\]
satisfies $\vol\left(\Pi_m\left(A\right)\setminus U\right)=0$. The fact that $\phi_j$, $j\in J$, are distinct functions implies that 
$\left(\left(\alpha_{j;\gamma,k}\right)_{\left(\gamma,k\right)\in\Gamma}\right)_{j\in J}$ 
is a family of distinct tuples in $\ZZ^\Gamma$.
As a consequence, for each $x\in U$ the trigonometric polynomial $z\mapsto F\left(x,z\right)$ is nonconstant.

Observe that since $\left(\phi_{\gamma,k}\right)_{\gamma\in\Gamma_k}$ is independent over $\QQ$ (as functions of $x$), 
for each $k\in\{1,\ldots,d\}$ and nonzero tuple $c = \left(c_{\gamma}\right)\in\ZZ^{\Gamma_k}$, $\sum_{\gamma\in\Gamma_k}c_{\gamma}\phi_{\gamma,k}$ is a nonzero analytic function on $\Pi_m\left(A\right)$, so the set
$
\left\{x\in U : \sum_{\gamma\in\Gamma_k} c_{\gamma} \psi_{\gamma,k}\left(x\right) = 0\right\}
$
cannot have a positive measure, and the set
\[
 B:= U\setminus\left(\bigcup_{k=1}^{d}\bigcup_{c\in\ZZ^{\Gamma_k}
 \setminus\{0\}}\left\{x\in U : \sum_{\gamma\in\Gamma_k} c_{\gamma} \phi_{\gamma,k}\left(x\right) = 0\right\} \right)
\]
satisfies $\vol_m\left(\Pi_m\left(A\right)\setminus B\right) = 0$ as well. This gives \eqref{F par non nul}, since $B\subset U$.

 On the other hand, from the definition of $B$ we see that for each $k\in \{1,\ldots, d\}$, for each $x\in B$, the family of numbers  $\left(\phi_{\gamma,k}\left(x\right)\right)_{\left(\gamma,k\right)\in\Gamma}$ is linearly independent over $\QQ$, 
and by Remark \ref{rem:independence through J_k}, for each $x\in B$
the family of functions $\left(t\mapsto \psi_{\gamma,k}\left(x,t\right)\right)_{\left(\gamma,k\right)\in\Gamma}$ is linearly independent over $\QQ$.
Given an open set $\Omega\subseteq\Pi_m\left(A\right)$ and any positive real number $\lambda$ with $\lambda < \vol_m\left(\Omega\right) = \vol_m\left(\Omega\cap B\right)$, the inner regularity of the Lebesgue measure shows that we may fix a compact set $K\subseteq \Omega \cap B$ with $\vol_m\left(K\right) \geq\lambda$.  Since $K$ is compact and $a\left(x\right)$ is continuous, we may fix $T_0$ sufficiently large so that $K \times[T_0,+\infty) \subseteq A$.  Proposition 
\ref{prop:WeylToCUD} then shows that the restriction of $\psi := \left(\psi_{\gamma,k}\right)_{\left(\gamma,k\right)\in\Gamma}$ to $K\times[T_0,+\infty)$ is c.u.d. mod $1$ on $K$, which completes the proof of \eqref{Psi CUD}.
\end{proof}

The following 
Lemma \ref{lem:Param:Dovetail}
is the parametric version of Proposition \ref{prop:goal1}(\ref{enu: existence of limit}).
It will be used in the proof 
of Proposition \ref{prop:complete}.
 
\begin{lem}\label{lem:Param:Dovetail}
Consider $f\left(x,t\right) = F\left(x,\psi\left(x,t\right)\right)$ as given in Lemma \ref{lem:Param:OscPrep}, with $z\mapsto F\left(x,z\right)$ nonconstant for some $x\in \Pi_m\left(A\right)$.  
Then there exist $\varepsilon>0$, 
$\delta > 0$,
a strictly increasing sequence $\left(t_j\right)_{j\in\NN}$ in $\RR$ diverging to $+\infty$, 
a compact set $K\subset X$ and a sequence $\left(X_j\right)_{j\in\NN}$ of Lebesgue measurable subsets of $K \subseteq \Pi_m\left(A\right)$,
with for any  $j\in \NN$, $\vol_m\left(X_j\right)\ge \delta$,  $ X_{2j+1}\subseteq X_{2j}$  and
such that
for all $j\in\NN$, for all $x_0\in X_{2j}$, $x_1\in X_{2j+1}$, 
$$ 
\vert f\left(x_0,t_{2j}\right)\vert \ge \varepsilon
 \ \mathrm{ and } \ \vert f\left(x_0,t_{2j}\right) - f\left(x_1,t_{2j+1}\right)\vert \ge \varepsilon .$$
\end{lem}
\begin{proof}
Since $z\mapsto F\left(x,z\right)$ is nonconstant and $1$-periodic in each of the components  of $z = \left(z_1,\ldots,z_n\right)$, 
one may find $v_0, v_1\in [0,1)^{n}$ such that $f\left(x,v_0\right)\not=f\left(x,v_1\right)$ and thus, assuming for instance that 
$f\left(x,v_0\right)\not=0$,
 one may fix $\varepsilon>0$, 
an open subset $U$ of $\Pi_m\left(A\right)$ containing $x$ 
and boxes $I_0, I_1 \subseteq[0,1)^n$ respectively containing $v_0,v_1$ 
such that $\dist\left(F\left(U\times I_0\right),0\right)\ge \varepsilon$ and 
$\dist\left(
F\left(U\times I_0\right), F\left(U\times I_1
\right)\right)\ge \varepsilon $.  

By Lemma \ref{lem:Param:OscPrep}\eqref{Psi CUD} we may fix a compact set $K\subseteq U$ and $T_0\in\RR$ such that $\vol_m\left(K\right) > 0$, $K\times[T_0,+\infty)\subseteq A$, and $\psi\restriction{K\times[T_0,+\infty)}$ is c.u.d.\ mod $1$ on $K$.  Define
$$
\delta = \frac{1}{2}\vol_m\left(K\right)\vol_n\left(I_0\right)  \min \left\{1, \frac{1}{2}\vol_n\left(I_1\right) \right\}.
$$
Remark \ref{rem:Param:geometricIntervals} shows that we may fix $k_0\in \NN$, 
$2^{k_0} > T_0$, such that for all $k\ge k_0$, all $i\in \{0,1\}$ and all $x\in K$,
$$
\vol_1\left(\{t\in[2^k,2^{k+1}) : \{ \psi\left(x,t\right)\} \in I_i\}\right) \geq 2^{k-1} \vol_n\left(I_i\right).
$$
Let us now construct $t_0, t_1$ and the corresponding sets $X_1\subset X_0 \subset K$. 
For this we consider 
\[
E_0 := \{\left(x,t\right) \in K \times [2^{k_0}, 2^{k_0+1}) : \{ \psi\left(x,t\right) \} \in I_0\}.
\]
Fubini's theorem gives, integrating first in the variable $t$ and then in the variable $x$,
$$ \vol_{m+1}\left(E_0\right) = \int_{x\in K} \vol_1\left(\{t : \left(x,t\right)\in E_0\}\right) \,\mathrm{d}x  
     \geq \vol_m\left(K\right) 2^{k_0-1}\vol_n\left(I_0\right).$$ 
But Fubini's theorem  also gives, integrating first in the variable $x$ and then in the variable $t$,
$$
\vol_{m+1}\left(E_0\right) = \int_{2^{k_0}}^{2^{k_0+1}} \vol_m\left(\{x : \left(x,t\right)\in E_0\}\right) \,\mathrm{d}t.
$$
It follows that we may certainly choose $t_0 \in [2^{k_0}, 2^{k_0+1})$ to  define
$$
X_0 = \left\{x\in K : \left(x,t_0\right)\in E_0\right\}
$$
so that
\begin{equation}\label{eq:X0bound}
\vol_m\left(X_0\right) \geq \frac{\vol_{m+1}\left(E_0\right)}{2^{k_0}} \geq 
\frac{1}{2}\vol_m\left(K\right)\vol_n\left(I_0\right)\ge \delta.
\end{equation}
Now denote $k_0+1$ by $k_1$. Then $2^{k_1} > t_0$. 
We apply the same construction as above but with $I_1$ instead of $I_0$, $[2^{k_1}, 2^{k_1+1})$ instead $[2^{k_0}, 2^{k_0+1})$
and $X_0$ instead of $K$. 
For this we define
$$
E_1 = \{\left(x,t\right) \in X_0 \times [2^{k_1}, 2^{k_1+1}) : \{ \psi\left(x,t\right) \} \in I_1\}
$$
and then we choose, with the same argument as above using Fubini's theorem on $E_1$, some $t_1\in [2^{k_1}, 2^{k_1+1})$ and define
$$
X_1 = \left\{x\in X_0 : \left(x,t_1\right)\in E_1\right\}
$$
so that, in conjunction with \eqref{eq:X0bound}
$$
\vol_m\left(X_1\right) \geq \frac{1}{2} \vol_m\left(X_0\right)\vol_n\left(I_1\right) \geq \frac{1}{4}\vol_m\left(K\right)\vol_n\left(I_0\right)\vol_n\left(I_1\right)\ge \delta. 
$$
 For $j\ge 1$, the pairs $\left(t_{2j}, t_{2j+1}\right)$, $\left(X_{2j},X_{2j+1}\right)$ are defined in the same way, $t_{2j}$ being constructed
 from $t_{2j-1}$, $X_{2j}$ from $K$, $t_{2j+1}$ from $t_{2j}$ and $X_{2j+1}$ from $X_{2j}$. 
\end{proof}

We can finally prove Proposition \ref{prop:complete}. 

\begin{proof}[Proof of Proposition \ref{prop:complete}]
Let $p\in[1,+\infty]$ and $f\in\C^{\exp}\left(X\times\RR\right)$ for a subanalytic set $X\subseteq\RR^{m}$, and suppose that $\left(f_y\right)_{y\in\RR}$ is Cauchy in $L^{p}\left(X\right)$ as $y\to+\infty$.  Since $L^p\left(X\right)$ is complete, there exists a function $h\in L^p\left(X\right)$ such that
\begin{equation}\label{eq:LpComplete}
\lim_{y\to+\infty}\|f_y-h\|_p = 0,
\end{equation}
and there exists a sequence $\left(y_j\right)_{j\in\NN}$ in $\RR$ tending to $+\infty$ such that
\begin{equation}\label{eq:SeqConverge}
\lim_{j\to+\infty} f\left(x,y_j\right) = h\left(x\right) \quad\text{for almost all $x\in X$.}
\end{equation}
(See for instance Theorems 3.11 and 3.12 in Rudin \cite{rudin:realandcomplex}.)

Apply Theorem \ref{thm:prepExpConstr} to $f\left(x,y\right)$ with respect to $y$. Let $\A$ be the collection of cells $A$ given by the preparation that are open in $\RR^{m+1}$ and of the form
\[
A=\{\left(x,y\right) : x\in\Pi_{m}\left(A\right), y > a\left(x\right)\},
\]
and put $X_0=\bigcup\{\Pi_{m}\left(A\right):A\in\A\}$.  Since $\vol_{m}\left(X\setminus X_0\right)=0$, it suffices to focus on one $A\in\A$ and prove that the conclusion of the theorem holds for $f\restriction{A}$.  Write $f$ as a finite sum
\[
f\left(x,y\right) = \sum_{j\in J} T_j\left(x,y\right)
\]
on $A$ with each term of the form $T_j\left(x,y\right) = y^{r_j}\left(\log y\right)^{s_j} g_j\left(x,y\right)$ specified in Remark \ref{rem: superint tend to 0}; thus $g_j\in\C^{\exp}\left(A\right)$ with $|g_j\left(x,y\right)| \leq \eta_j\left(x\right)$ on $A$ for some continuous function $\eta_j:\Pi_m\left(A\right)\to[0,+\infty)$, and $g_j\left(x,y\right) = f_j\left(x\right) \mathrm{e}^{\mathrm{i}\phi_j\left(x,y\right)}$ when $r_j \geq -1$, with $\phi_j\left(x,y\right)$ distinct polynomials in $y^{1/d}$
for some integer $d\ge 0$, such that $\phi_j\left(x,0\right)=0$ for all
$x\in \Pi_m\left(A\right)$. 
  Each function $f_j$ can be taken to be analytic on $A$ by Remark \ref{rem:piecewise analytic} and not identically zero, and $\Pi_{m}\left(A\right)$ is connected since $A$ is a subanalytic cell.  We claim that there exists $g\in\C^{\exp}\left(\Pi_m\left(A\right)\right)$ such that $\lim_{y\to+\infty}f\left(x,y\right) = g\left(x\right)$ for all $x\in\Pi_m\left(A\right)$.  This claim and \eqref{eq:SeqConverge} imply that $g\left(x\right) = h\left(x\right)$ for almost all $x\in\Pi_m\left(A\right)$, and hence $\lim_{y\to+\infty}\|{f_y}\restriction{ \Pi_m\left(A\right)}-g\|_p = 0$ by \eqref{eq:LpComplete}.  So we will be done once we prove the claim.

Let $E = \{\left(r_j,s_j\right) : j\in J\}$, and for each $\left(r,s\right)\in E$ let $J_{\left(r,s\right)} = \{j\in J : \left(r_j,s_j\right) = \left(r,s\right)\}$.  Thus
\begin{equation}\label{eq:Prep:LikeTerms}
f\left(x,y\right) = \sum_{\left(r,s\right)\in E} y^r \left(\log y\right)^s S_{\left(r,s\right)}\left(x,y\right)
\end{equation}
where for each $\left(r,s\right)\in E$,
\begin{equation}
S_{\left(r,s\right)}\left(x,y\right) = \sum_{j\in J_{\left(r,s\right)}} g_j\left(x,y\right).
\end{equation}
For each $\left(r,s\right)\in E$, define $\eta_{\left(r,s\right)}:\Pi_m\left(A\right)\to[0,+\infty)$ by $\eta_{\left(r,s\right)}\left(x\right) = \sum_{j\in J_{\left(r,s\right)}} \eta_j\left(x\right)$, and observe that $\eta_{\left(r,s\right)}$ is continuous and that $\left|S_{\left(r,s\right)}\left(x,y\right)\right|\leq \eta_{\left(r,s\right)}\left(x\right)$ on $A$.

Let $\left(\bar{r},\bar{s}\right)$ be the lexicographic maximum element of $E$.  If $\bar{r} < 0$, then $\lim_{y\to+\infty} f\left(x,y\right) = 0$ for all $x\in\Pi_m\left(A\right)$, and we are done.  If $\bar{r} = \bar{s} = 0$ and $J$ is a singleton, say $J = \{j_0\}$, and $\phi_{j_0} = 0$, then
\[
f\left(x,y\right) = f_{j_0}\left(x\right) + \sum_{\left(r,s\right)\in E \setminus\{\left(0,0\right)\}} y^r \left(\log y\right)^s g_j\left(x,y\right),
\]
with $r < 0$ for all $\left(r,s\right)\in E\setminus\{\left(0,0\right)\}$, so $\lim_{y\to+\infty} f\left(x,y\right) = f_{j_0}\left(x\right)$ on $\Pi_m\left(A\right)$, and we are also done.  The two remaining cases are when $\bar{r} > 0$ or $\bar{s} > 0$, or when $\bar{r} = \bar{s} = 0$ and $\phi_j\neq 0$ for some $j\in J_{\left(0,0\right)}$.  We will complete the proof by showing that these two remaining cases are impossible.

We may assume that $\bar{r}\geq 0$, since this is a common assumption of the two remaining cases.  Notice that  
\[
S_{\left(\bar{r},\bar{s}\right)}\left(x,y^d\right) = \sum_{j\in J_{\left(\bar{r},\bar{s}\right)}} f_j\left(x\right) \mathrm{e}^{\mathrm{i}\phi_j\left(x,y^d\right)}
\]
is of the form hypothesized in Lemma \ref{lem:Param:OscPrep}.  Therefore, we can apply Lemma \ref{lem:Param:Dovetail} and
find
$\varepsilon > 0$, $\delta >0$, a compact set $K\subset \Pi_m\left(A\right)$,
a strictly increasing sequence $\left(y_j\right)_{j\in\NN}$ in $\left(1,+\infty\right)$ tending to $+\infty$ with $K\times [y_0,+\infty)\subset A$, 
a sequence $\left(X_j\right)_{j\in\NN}$ of Lebesgue measurable subsets of $K$ such that for all $j\in \NN$, $\vol_m\left(X_j\right) \ge \delta $, 
$ X_{2j+1}\subset X_{2j} \subset K$ and
$$
  \forall \ x\in X_{2j+1}, \
\vert S_{\left(r,s\right)}\left(x,y_{2j}\right) \vert \ge 2\varepsilon 
  \ \mathrm{ and } \
  \vert S_{\left(r,s\right)}\left(x,y_{2j+1}\right) - S_{\left(r,s\right)}\left(x,y_{2j}\right) \vert \geq 3 \varepsilon.
$$
The set $K$ is compact, each function $\eta_{\left(r,s\right)}$ is continuous, and $\lim_{y\to+\infty} y^{r-\bar{r}}\left(\log y\right)^{s-\bar{s}} = 0$ for all $\left(r,s\right)\in E\setminus\{\left(\bar{r},\bar{s}\right)\}$, so by replacing $\left(y_j\right)_{j\in\NN}$ with a tail of the sequence, we may assume that for all $y \geq y_0$,
$$
\max \left\{\sum_{\left(r,s\right)\in E\setminus\{\left(\bar{r},\bar{s}\right)\}} y^{r-\bar{r}}\left(\log y\right)^{s-\bar{s}} \eta_{\left(r,s\right)}\left(x\right) : x\in K \right\} \le \varepsilon.
$$
Observe that for all $j\in\NN$ and $x\in X_{2j}$,
\[
\left| S_{\left(\bar{r},\bar{s}\right)}\left(x,y_{2j}\right) + \sum_{\left(r,s\right)\in E\setminus\{\left(\bar{r},\bar{s}\right)\}} y_{2j}^{r-\bar{r}} \left(\log y_{2j}\right)^{s-\bar{s}} S_{\left(r,s\right)}\left(x,y_{2j}\right) \right| 
\geq 2\varepsilon - \varepsilon = \varepsilon,
\]
so
\begin{align*}
|f\left(x,y_{2j}\right)| &
    = y_{2j}^{\bar{r}}\left(\log y_{2j}\right)^{\bar{s}}
  \left| S_{\left(\bar{r},\bar{s}\right)}\left(x,y_{2j}\right) + \sum_{\left(r,s\right)\in E\setminus\{\left(\bar{r},\bar{s}\right)\}} y_{2j}^{r-\bar{r}} \left(\log y_{2j}\right)^{s-\bar{s}} S_{\left(r,s\right)}\left(x,y_{2j}\right) \right| \\
 &    \geq y_{2j}^{\bar{r}}\left(\log y_{2j}\right)^{\bar{s}}\varepsilon.
\end{align*}
In consequence, if $\bar{r} > 0$ or $\bar{s} > 0$, then
$$
\|f_{y_{2j}} - h\|_p
\geq \|f_{y_{2j}}\|_p - \|h\|_p
\geq y_{2j}^{\bar{r}}\left(\log y_{2j}\right)^{\bar{s}}\varepsilon\delta - 
\|h\|_p 
\underset{n\to+\infty}{\longrightarrow}
+\infty,
$$
 which contradicts \eqref{eq:LpComplete}.  So we may suppose that $\bar{r} = \bar{s} = 0$ and $\phi_j\neq 0$ for some $j\in J_{\left(0,0\right)}$.  Thus on $A$
 $$
f\left(x,y\right) = S_{\left(0,0\right)}\left(x,y\right) + \sum_{\left(r,s\right)\in E\setminus\{\left(0,0\right)\}} y^{r}\left(\log y\right)^s S_{\left(r,s\right)}\left(x,y\right).
$$ 
It follows that for all $j\in\NN$ and $x\in X_{2j+1}$,
$$
\begin{aligned}
\left|f\left(x,y_{2j}\right) - f\left(x,y_{2j+1}\right)\right| 
& \geq \left|S\left(x,y_{2j}\right) - S\left(x,y_{2j+1}\right) \right|  \\    
&     - \left|\sum_{\left(r,s\right)\in E\setminus\{\left(0,0\right)\}} y_{2j}^{r} \left(\log y_{2j}\right)^s S_{\left(r,s\right)}\left(x,y_{2j}\right) \right| \\
& - \left|\sum_{\left(r,s\right)\in E\setminus\{\left(0,0\right)\}} y_{2j+1}^{r} \left(\log y_{2j+1}\right)^s S_{\left(r,s\right)}\left(x,y_{2j+1}\right)\right| \\
 & \geq 3\varepsilon - \varepsilon - \varepsilon = \varepsilon.
\end{aligned}
$$
Finally we obtain, for all $j\in\NN$,
$$
\|f_{2j} - f_{y_{2j+1}}\|_p \geq \varepsilon\delta,
$$
which contradicts the fact that $\left(f_y\right)_{y\in\RR}$ is Cauchy in $L^p\left(X\right)$ as $y\to+\infty$.
\end{proof}

\bibliographystyle{amsplain}
\bibliography{bibli}

\def\cprime{$'$}
\providecommand{\bysame}{\leavevmode\hbox to3em{\hrulefill}\thinspace}
\providecommand{\MR}{\relax\ifhmode\unskip\space\fi MR }
\providecommand{\MRhref}[2]{%
  \href{http://www.ams.org/mathscinet-getitem?mr=#1}{#2}
}
\providecommand{\href}[2]{#2}
\begin{thebibliography}{10}

\bibitem{abramowitz:handbook_mathematical_functions}
M.~Abramowitz and I.~A. Stegun, \emph{Handbook of mathematical functions with
  formulas, graphs, and mathematical tables}, Dover Publications Inc., 1965.

\bibitem{arnold_gusein_varchenko:singularities_2}
V.~I. Arnol'd, S.~M. Guse{\u\i}n-Zade, and A.~N. Varchenko, \emph{Singularities
  of differentiable maps. {V}ol. {II}}, Monographs in Mathematics, vol.~83,
  Birkh\"auser Boston, Inc., Boston, MA, 1988.

\bibitem{bm_semi_subanalytic}
E.~Bierstone and P.~D. Milman, \emph{Semianalytic and subanalytic sets}, Inst.
  Hautes {\'E}tudes Sci. Publ. Math. (1988), no.~67, 5--42.

\bibitem{BlochEsn}
Spencer Bloch and H{\'e}l{\`e}ne Esnault, \emph{Gau\ss-{M}anin determinant
  connections and periods for irregular connections}, Geom. Funct. Anal.
  (2000), no.~Special Volume, Part I, 1--31, GAFA 2000 (Tel Aviv, 1999).
  \MR{1826247 (2002j:14012)}

\bibitem{CGH2}
R.~Cluckers, J.~Gordon, and I.~Halupczok, \emph{Local integrability results in
  harmonic analysis on reductive groups in large positive characteristic}, Ann.
  Sci. \'Ecole Norm. Sup. (4) \textbf{47} (2014), no.~6, 1163--1195.

\bibitem{cluckers_hales_loeser:transfer_principle_fundamental}
R.~Cluckers, T.~Hales, and F.~Loeser, \emph{Transfer principle for the
  fundamental lemma}, On the stabilization of the trace formula, Stab. Trace
  Formula Shimura Var. Arith. Appl., vol.~1, Int. Press, Somerville, MA, 2011,
  pp.~309--347.

\bibitem{cluckers_loeser:constructible_exponential_fuctions}
R.~Cluckers and F.~Loeser, \emph{Constructible exponential functions, motivic
  {F}ourier transform and transfer principle}, Ann. of Math. (2) \textbf{171}
  (2010), no.~2, 1011--1065.

\bibitem{cluckers-miller:stability-integration-sums-products}
R.~Cluckers and D.~J. Miller, \emph{Stability under integration of sums of
  products of real globally subanalytic functions and their logarithms}, Duke
  Math. J. \textbf{156} (2011), no.~2, 311--348.

\bibitem{cluckers_miller:loci_integrability}
\bysame, \emph{Loci of integrability, zero loci, and stability under
  integration for constructible functions on {E}uclidean space with {L}ebesgue
  measure}, Int. Math. Res. Not. IMRN (2012), no.~14, 3182--3191.

\bibitem{clr}
G.~Comte, J.-M. Lion, and J.-P. Rolin, \emph{Nature log-analytique du volume
  des sous-analytiques}, Illinois J. Math. \textbf{44} (2000), no.~4, 884--888.

\bibitem{stasica_denkowska}
S.~Denkowska and J.~Stasica, \emph{Ensembles sous-analytiques {\`a} la
  polonaise}, Travaux en Cours, Hermann, 2008.

\bibitem{vdd:tarski:general}
{L. van den} Dries, \emph{A generalization of the {T}arski-{S}eidenberg
  theorem, and some nondefinability results}, Bull. Amer. Math. Soc. (N.S.)
  \textbf{15} (1986), no.~2, 189--193.

\bibitem{vdd:elementary_theory_restricted_elementary}
\bysame, \emph{On the elementary theory of restricted elementary functions}, J.
  Symbolic Logic \textbf{53} (1988), no.~3, 796--808.

\bibitem{vdd:tame}
\bysame, \emph{Tame topology and o-minimal structures}, London Mathematical
  Society Lecture Note Series, vol. 248, Cambridge University Press, Cambridge,
  1998.

\bibitem{dmm:exp}
{L. van den} Dries, A.~Macintyre, and D.~Marker, \emph{The elementary theory of
  restricted analytic fields with exponentiation}, Ann. of Math. (2)
  \textbf{140} (1994), no.~1, \ 183--205.

\bibitem{vdd:mill:omin}
{L. van den} Dries and C.~Miller, \emph{Geometric categories and o-minimal
  structures}, Duke Math. J. \textbf{84} (1996), no.~2, 497--540.

\bibitem{gabriel:proj}
A.~Gabrielov, \emph{Projections of semianalytic sets}, Funkcional. Anal. i
  Prilo\v zen. \textbf{2} (1968), no.~4, 18--30.

\bibitem{gasquet-witomski:fourier-analysis}
C.~Gasquet and P.~Witomski, \emph{Fourier analysis and applications}, Texts in
  Applied Mathematics, vol.~30, Springer-Verlag, New York, 1999.

\bibitem{hormander:fourier_integral_operators}
L.~H{\"o}rmander, \emph{The analysis of linear partial differential operators.
  {IV}}, Classics in Mathematics, Springer-Verlag, Berlin, 2009, Fourier
  integral operators, Reprint of the 1994 edition.

\bibitem{HK}
E.~Hrushovski and D.~Kazhdan, \emph{Integration in valued fields}, Algebraic
  geometry and number theory, Progr. Math., vol. 253, Birkh\"auser Boston,
  Boston, MA, 2006, pp.~261--405.

\bibitem{kaiser:integration_semialgebraic_nash}
T.~Kaiser, \emph{Integration of semialgebraic functions and integrated {N}ash
  functions}, Math. Z. \textbf{275} (2013), no.~1-2, 349--366.

\bibitem{kontsevitch_zagier:periods}
M.~Kontsevich and D.~Zagier, \emph{Periods}, Mathematics unlimited---2001 and
  beyond, Springer, Berlin, 2001, pp.~771--808.

\bibitem{Kui}
L.~Kuipers and H.~Niederreiter, \emph{Uniform distribution of sequences},
  Wiley-Interscience [John Wiley \& Sons], New York-London-Sydney, 1974, Pure
  and Applied Mathematics.

\bibitem{lr:gabriel}
J.-M. Lion and J.-P. Rolin, \emph{Th{\'e}or{\`e}me de {G}abrielov et fonctions
  log-exp-alg{\'e}briques}, C. R. Acad. Sci. Paris S{\'e}r. I Math.
  \textbf{324} (1997), no.~9, 1027--1030.

\bibitem{lr:prep}
\bysame, \emph{Th{\'e}or{\`e}me de pr{\'e}paration pour les fonctions
  logarithmico-exponentielles}, Ann. Inst. Fourier (Grenoble) \textbf{47}
  (1997), no.~3, 859--884.

\bibitem{LiRo1998}
\bysame, \emph{Int\'egration des fonctions sous-analytiques et volumes des
  sous-ensembles sous-analytiques}, Ann. Inst. Fourier (Grenoble) \textbf{48}
  (1998), no.~3, 755--767.

\bibitem{lo:ana}
S.~{\L}ojasiewicz, \emph{Sur les ensembles semi-analytiques}, Actes du
  {C}ongr\`es {I}nternational des {M}ath\'ematiciens ({N}ice, 1970), {T}ome 2,
  Gauthier-Villars, Paris, 1971, pp.~237--241.

\bibitem{malgrange:integrales_asymptotiques}
B.~Malgrange, \emph{Int\'egrales asymptotiques et monodromie}, Ann. Sci.
  \'Ecole Norm. Sup. (4) \textbf{7} (1974), 405--430 (1975).

\bibitem{miller_dan:preparation_theorem_weierstrass_systems}
D.~J. Miller, \emph{A preparation theorem for {W}eierstrass systems}, Trans.
  Amer. Math. Soc. \textbf{358} (2006), no.~10, 4395--4439 (electronic).

\bibitem{parusinski:preparation}
A.~Parusi{\'n}ski, \emph{On the preparation theorem for subanalytic functions},
  New developments in singularity theory ({C}ambridge, 2000), NATO Sci. Ser. II
  Math. Phys. Chem., vol.~21, Kluwer Acad. Publ., Dordrecht, 2001,
  pp.~193--215.

\bibitem{rudin:realandcomplex}
Walter Rudin, \emph{Real and complex analysis}, third ed., McGraw-Hill Book
  Co., New York, 1987.

\bibitem{stein:harmonic_analysis}
E.~M. Stein, \emph{Harmonic analysis: real-variable methods, orthogonality, and
  oscillatory integrals}, Princeton Mathematical Series, vol.~43, Princeton
  University Press, Princeton, NJ, 1993.

\bibitem{varchenko:newton_polyhedra_oscillatory_integrals}
A.~N. Varchenko, \emph{Newton polyhedra and estimates of oscillatory
  integrals}, Funct. Anal. Appl. \textbf{18} (1976), no.~3, 175--196.

\bibitem{Wey}
Hermann Weyl, \emph{\"uber die {G}leichverteilung von {Z}ahlen mod. {E}ins},
  Math. Ann. \textbf{77} (1916), no.~3, 313--352.

\end{thebibliography}

\end{document}